\documentclass[reqno,11pt]{amsart}

\usepackage[english]{babel}
\usepackage[utf8]{inputenc}
\usepackage[T1]{fontenc}

\usepackage[margin=1in, letterpaper]{geometry}
\usepackage{amssymb,amsmath,amsfonts,amsthm}
\usepackage{bm,bbm,mathrsfs}
\usepackage[dvipsnames]{xcolor}
\usepackage{url}
\usepackage{enumitem}
\usepackage{mathtools}
\usepackage{placeins}
\usepackage{extarrows}
\usepackage{subfig}
\usepackage[font=small,labelfont=bf]{caption}
\usepackage{floatrow}
\usepackage{verbatim}
\usepackage{wrapfig}
\usepackage{tikz}
\usetikzlibrary{arrows,arrows.meta,cd,decorations.pathmorphing,backgrounds,positioning,fit,matrix}
\usepackage{xcolor,colortbl,graphics,graphicx}
\usepackage[colorlinks=true, allcolors=Blue]{hyperref}
\usepackage[noabbrev,capitalize]{cleveref}
\crefname{equation}{}{}

\newtheorem{theorem}{Theorem}[section]
\newtheorem{proposition}[theorem]{Proposition}
\newtheorem{lemma}[theorem]{Lemma}

\newtheorem{corollary}[theorem]{Corollary}

\newtheorem*{question*}{Question}
\Crefname{question}{Question}{Questions}

\newtheorem{knowntheorem}{Theorem}

\newtheorem{knownlemma}{Lemma}

\theoremstyle{definition}
\newtheorem{definition}[theorem]{Definition}
\newtheorem{notation}[theorem]{Notation}
\newtheorem*{notation*}{Notation}
\newtheorem{problem}[theorem]{Problem}
\newtheorem{open}[theorem]{Open Problem}

\newtheorem*{example*}{Example}

\theoremstyle{remark}
\newtheorem*{remark}{Remark}

\numberwithin{equation}{section}
\allowdisplaybreaks

\newcommand{\Z}{\mathbf{Z}}
\newcommand{\R}{\mathbf{R}}
\newcommand{\Q}{\mathbf{Q}}
\newcommand{\C}{\mathbf{C}}
\newcommand{\N}{\mathbf{N}}
\renewcommand{\P}{\mathbf{P}}
\renewcommand{\H}{\mathbf{H}}

\newcommand{\I}{\mathbf{I}}
\newcommand{\E}{\mathbf{E}}

\newcommand{\mE}{\mathcal{E}}
\newcommand{\mF}{\mathcal{F}}

\newcommand{\mD}{\mathcal{D}}

\newcommand{\mR}{\mathcal{R}}

\newcommand{\mP}{\mathcal{P}}
\newcommand{\mQ}{\mathcal{Q}}
\newcommand{\mC}{\mathcal{C}}
\newcommand{\mT}{\mathcal{T}}

\newcommand{\one}{\mathbbm{1}}
\renewcommand{\Re}{\textnormal{Re}}
\renewcommand{\Im}{\textnormal{Im}}

\newcommand{\eps}{\varepsilon}
\newcommand{\uminus}{\mathbin{\ooalign{$\cup$\cr%
   \hfil\raise0.42ex\hbox{$\scriptscriptstyle\minus$}\hfil\cr}}}

\renewcommand{\emptyset}{\varnothing}

\renewcommand{\bar}{\overline}
\renewcommand{\hat}{\widehat}
\renewcommand{\tilde}{\widetilde}

\newcommand{\supp}{\textnormal{supp}}

\newcommand{\sq}{_{\substack{ \\ \square}}}
\newcommand{\gow}{_{\substack{ \\ \square^2}}}

\renewcommand{\pmod}[1]{\ (\textnormal{mod } #1)}

\newcommand{\cremph}[1]{\emph{\cref{#1}}}

\newcommand{\Tr}{\textnormal{Tr}}

\newcommand{\Mat}{\textnormal{Mat}}
\newcommand{\Part}{\textnormal{Part}}
\newcommand{\ex}{\textnormal{ex}}
\newcommand{\Hom}{\textnormal{Hom}}
\newcommand{\Inj}{\textnormal{Inj}}
\newcommand{\Energy}{\mE}
\newcommand{\Complexity}{\mC}
\newcommand{\rank}{\textnormal{rank}}
\newcommand{\diag}{\textnormal{diag}}

\newenvironment{psmall}
  {\left(\begin{smallmatrix}}
  {\end{smallmatrix}\right)}
\newenvironment{bsmall}
  {\left[\begin{smallmatrix}}
  {\end{smallmatrix}\right)}

\newcommand{\twoline}[2]{$\substack{\text{\normalsize \vphantom{y}#1} \\ \text{\normalsize \vphantom{y}#2}}$}

\newcommand{\twolinemath}[2]{\let\scriptstyle\textstyle \substack{{#1} \\ {#2}}}
\definecolor{firebrick}{rgb}{0.7, 0.13, 0.13}
\definecolor{myBlue}{rgb}{0, 0, 0.6}

\title{The sparse regularity method with Schatten norms and entropy}
\author[Alexandru Pascadi]{Alexandru Pascadi}
\address{Mathematical Institute, Radcliffe Observatory quarter, Woodstock Road, Oxford OX2 6GG, England}
\email{pascadi@maths.ox.ac.uk}

\begin{document}
\maketitle

\begin{abstract} 
We introduce a regularity method for sparse graphs, with new regularity and counting lemmas which use the Schatten--von-Neumann norms to measure uniformity. This leads to $k$-cycle removal lemmas in subgraphs of mildly-pseudorandom graphs, and also in graphs lacking a quasi-smooth family of bipartite subgraphs, extending results of Conlon, Fox, Sudakov and Zhao. We give some applications in additive combinatorics: one about translation-invariant linear equations in subsets of mildly-pseudorandom sets, one about such equations in generalized Sidon sets, and one about polygonal patterns in subsets of $\Z^2$ with few parallelograms (giving a two-dimensional analogue for a result of Prendiville). Separately, our regularity lemma implies a dense graph removal lemma with mild constant dependencies, in graphs whose spectral $L^{2-\eps}$ norms are small.
\end{abstract}


\section{Introduction} \label{sec:introduction}



A famous result in graph theory, initially motivated by the Rusza--Szemer\'edi problem \cite{ruzsa1978triple} and then by its application to Roth's theorem for $3$-term arithmetic progressions \cite{roth1953certain}, is the \emph{triangle removal lemma} reproduced below.

\begin{knowntheorem}[Triangle removal lemma] \label{thm:dense-triangle-removal} 
For any $\epsilon > 0$, there exists $\delta > 0$ such that any $n$-vertex graph $G$ with fewer than $\delta n^3$ triangles can be made triangle-free by removing at most $\epsilon n^2$ edges.
\end{knowntheorem}

This fact can be generalized to an $H$-removal lemma for any subgraph $H$ of a (weighted) graph $G$ \cite{erdos1986asymptotic,fox2011new,conlon2013graph}, and also to the setting of hypergraphs \cite{gowers2007hypergraph,tao2006variant}, the latter leading to a proof of Szemer\'edi's theorem for $k$-term arithmetic progressions \cite{szemeredi1975sets,tao2007ergodic}. 
We note that the known dependency of $\delta^{-1}$ on $\eps^{-1}$ (or on $\log \eps^{-1}$ \cite{fox2011new}) is at best tower-exponential, unless $H$ belongs to a special family. The \emph{regularity method} is the classical template used to prove such results, typically consisting of:
\begin{itemize}
    \item[1.] A \emph{regularity lemma}, which approximates $G$ by a simpler (weighted) graph $\tilde{G}$;
    \item[2.] A \emph{counting lemma}, which shows that $G$ and $\tilde{G}$ have about the same number of (homomorphic) copies of $H$;
    \item[3.] A \emph{removal argument}, which removes a few blocks of edges from $\tilde{G}$ to eliminate all homomorphic copies of $H$. After removing a few more edges, the same will hold for $G$. 
\end{itemize}

A shortcoming of results like \cref{thm:dense-triangle-removal} is that they are only nontrivial for \emph{dense} graphs, which have $\gg n^2$ edges; indeed, a graph with $\le \epsilon n^2$ edges can always be made triangle-free by removing all edges. By contrast, a triangle removal lemma for \emph{sparse} graphs $G$ would assume a total of $\delta p^3 n^3$ triangles, and would remove at most $\epsilon pn^2$ edges, where $p \in (0, 1)$ can decay to $0$ as $n \to \infty$ (these values can be justified by considering the Erd\"os--R\'enyi random graph $G(n, p)$, where each edge is included independently with probability $p$). But such results require additional assumptions on the graphs in question, and finding appropriate assumptions has been the subject of extensive literature \cite{conlon2014green,tao2006gaussian,tao2007ergodic,kohayakawa1997onk,kohayakawa1997szemeredi,scott2011szemeredi,conlon2021regularity,conlon2021graphs}. Despite their variety, most \emph{sparse regularity methods} contain:
\begin{itemize}
    \item[1.] A \emph{structure theorem} / \emph{dense model theorem}, often based on a (weak) \emph{sparse regularity lemma}. This is used to construct a \emph{dense} graph $\tilde{G}$ approximating the \emph{sparse} graph $G$; 
    \item[2.] A \emph{sparse counting lemma}, which plays the same role as the dense counting lemma (but requires additional assumptions about the sparse graphs involved);
    \item[3.] A \emph{transference argument}, which applies the dense $H$-removal lemma to $\tilde{G}$, and removes a few more edges to conclude an $H$-removal lemma for $G$.
\end{itemize}

A broad category of sparse removal lemmas assume that $G$ is a relatively-dense subgraph of a random or pseudorandom (hyper)graph, for an appropriate notion of (pseudo)randomness. Green and Tao used similar ideas to prove their celebrated theorem that the primes contain arbitrarily long arithmetic progressions \cite{green2008primes}, and Conlon, Fox and Zhao gave a simplified exposition of the proof using a more general pseudorandomness condition \cite{conlon2014green}; multidimensional analogues of this result require more restrictive (and complex) notions of pseudorandomness \cite{tao2006gaussian,tao2015multi}. Separately, much of the recent progress on sparse removal lemmas has focused on obtaining counting lemmas for subgraphs of random graphs, related to the resolution of the K\L R conjecture \cite{conlon2014klr,balogh2015independent,saxton2015hypergraph}.

A more recent direction of research concerns removal lemmas in sparse graphs which contain few copies of certain bipartite graphs, particularly in graphs with few $4$-cycles. This was partly motivated by the work of Chung, Graham and Wilson on quasi-random graphs \cite{chung1989quasi}, and pursued by Conlon, Fox, Sudakov and Zhao in a series of two papers \cite{conlon2021regularity,conlon2021graphs}; their work uses a regularity lemma similar to that of Scott \cite{scott2011szemeredi}, and makes its main contribution through the counting lemmas.

The goal of this paper is to introduce a new variant of the sparse regularity method, leading to improved $H$-removal lemmas in both types of graphs described above, when $H$ is a cycle or a path; these cases also have additive-combinatorial applications about sparse subsets of abelian groups. 

To highlight the difficulties faced by sparse regularity methods and to introduce our approach, we first state a version of Frieze and Kannan's weak regularity lemma \cite{frieze1999quick}, following the analytic form due to Lov\'asz \cite{lovasz2007szemeredi}. Here, graphs are replaced with functions $f : X \times Y \to \C$ (where $(X, \P_X)$ and $(Y, \P_Y)$ are finite probability spaces with full $\sigma$-algebras), and we will use the \emph{cut norm} \cite{lovasz2012large}
\begin{equation} \label{eq:cut}
    ||f||\sq := \sup_{\substack{A \subset X \\ B \subset Y}} \left\vert \E_{x,y} \left[f(x,y) \one_{A}(x) \one_{B}(y) \right] \right\vert,
\end{equation}
which has functioned in more recent literature  as a substitute for the \emph{Gowers box norm}  \cite{gowers2001new,gowers2007hypergraph}
\begin{equation} \label{eq:gowers}
    ||f||\gow := \E_{x_1,x_2,y_1,y_2}\left[
    f(x_1,y_1)\bar{f(x_2,y_1)}f(x_2,y_2)\bar{f(x_1,y_2)} \right]^{1/4}.
\end{equation}
(In the expected values from \cref{eq:cut,eq:gowers}, it is understood that $x_1, x_2, y_1, y_2$ are sampled independently from the appropriate probability spaces.)

\begin{knowntheorem}[Weak regularity lemma for dense graphs \cite{frieze1999quick, lovasz2007szemeredi}] \label{thm:weak-dense-regularity} 
Let $(X, \P_X), (Y, \P_Y)$ be finite probability spaces (with full $\sigma$-algebras), $f : X \times Y \to \C$, and $\eps \in (0, 1)$. Then there exist partitions $\mP$ of $X$ and $\mQ$ of $Y$, each with at most $2^{O(1/\eps^{O(1)})}$ parts, such that
\[
    \left\vert \left\vert f - \E(f \mid \mP \otimes \mQ)
    \right\vert\right\vert\sq \le \eps ||f||_{L^\infty}.
\]
(Here we denoted $\mP \otimes \mQ := \{A \times B : A \in \mP, B \in \mQ\}$, and the conditional expectation $\E(f \mid \mP \otimes \mQ)$ is obtained by averaging $f$ over parts of $\mP \otimes \mQ$; see \cref{subsec:partitions}.)
\end{knowntheorem}

\begin{remark}
The reader may imagine that $f = \one_G/p : V \times V \to [0, \infty)$, where $\one_G$ is the indicator function of the edges of a graph $G$ with vertex set $V$, $V$ is equipped with the uniform probability measure, and $p = \E[\one_G]$ is roughly the edge density of $G$.
\end{remark}

There are three key quantities in this theorem and its generalizations. The first is a \emph{density norm} depending on $f$ (in this case, the $L^\infty$ norm), which dictates the theorem's regime of applicability, since this norm should be bounded for the conclusion to be helpful. Indeed, when $f = \one_{G}/p$ and $p = \E[\one_G]$, one has $||f||_{L^\infty} = 1/p$, so \cref{thm:weak-dense-regularity} is only useful for \emph{dense} graphs $G$ (with edge density $p \gg 1$). The second is a \emph{uniformity norm} (in this case, the cut norm), the smallness of which is relevant in the counting lemma. The third is the \emph{complexity} of the partition $\mP \otimes \mQ$ (in this case, its cardinality); bounding this quantity makes a transference argument possible, since low-complexity functions $\E(f \mid \mP \otimes \mQ)$ are easier to approximate by $L^\infty$-bounded functions. In this context, a \emph{sparse} regularity method may start by obtaining analogues of \cref{thm:weak-dense-regularity} which replace $||\cdot||_{L^\infty}$ with a weaker density norm, and $||\cdot||\sq$ with a stronger uniformity norm.

A key idea in our work is to use a family of norms $||f||_{S^q}$, for $q \in [1, \infty]$, as \emph{both} density norms and uniformity norms. Here, $||f||_{S^q}$ is the Schatten $q$-norm (i.e., the $L^q$ norm on the singular values) of a matrix associated to $f$, appropriately normalized with respect to the probability spaces; we will also refer to $||f||_{S^q}$ as the Schatten $q$-norm of $f$. Using singular values to study graph regularity has of course been the topic of multiple papers \cite{frieze1999simple,fox2019fast,szegedy2011limits,bodwin2022unified,nikiforov2012extremal}, most of these focusing on the maximal singular value $||f||_{S^\infty}$. We extend \cref{thm:weak-dense-regularity} to the setting of these norms, with no cutoffs on $f$; this requires a different notion of complexity for our partitions: Shannon entropy (see \cref{subsec:partitions}). 

\begin{theorem}[Sparse regularity lemma]\label{thm:regularity-intro} 
Let $(X, \P_X), (Y, \P_Y)$ be finite probability spaces (with full $\sigma$-algebras), $f : X \times Y \to \C$, $\eps \in (0, 1)$, and $1 \le q < r \le \infty$. Then there exist partitions $\mP$ of $X$ and $\mQ$ of $Y$, each with entropy at most $O_{q,r}(\eps^{-O_{q,r}(1)})$, such that
\[
    \left\vert \left\vert f - \E(f \mid \mP \otimes \mQ)
    \right\vert\right\vert_{S^r} \le \eps ||f||_{S^q}.
\]
\end{theorem}

\begin{remark} The crucial feature of \cref{thm:regularity-intro} is the strength of the uniformity norm, which can get arbitrarily close to that of the density norm. This is possible because the Schatten $q$-norms essentially interpolate between a measure of density (when $q \le 2$) and a measure of uniformity (when $q = \infty$). Naturally, the bound for the entropies $\H(\mP)$, $\H(\mQ)$ blows up as $r$ approaches $q$.
\end{remark}

\begin{remark}
The case $q = 2$ of \cref{thm:regularity-intro} is a dense regularity lemma (generalizing \cref{thm:weak-dense-regularity}), applicable when $||f||_{L^2} \ll 1$. Surprisingly, \cref{thm:regularity-intro} is also interesting in the ``super-dense'' regime $q < r = 2$, where the norm $||f - \E(f \mid \mP \otimes \mQ)||_{L^2}$ is shown to be small; see \cref{thm:super-dense-removal}.
\end{remark}

We give a generalization of \cref{thm:regularity-intro}, which approximates more functions simultaneously, in \cref{thm:regularity}; we regard these sparse regularity lemmas and the sparse graph removal lemmas that they lead to (discussed in the next subsection) as our main results. 
Indeed, equipped with this regularity lemma, a compatible counting lemma (\cref{lem:counting}) will follow from standard properties of the Schatten norms. Our transference step will use a \emph{dense pairs condition} similar to, but more general than that of Conlon, Fox, Sudakov and Zhao from \cite{conlon2021regularity}. For paths and even-length cycles, we can do better by entirely avoiding the transference step, and using properties related to Sidorenko's conjecture \cite{sidorenko1991inequalities}; \cref{fig:results} outlines the logical dependencies between our main results.

\begin{figure}[ht]
\centering 
\begin{tikzpicture}[
square/.style={rectangle, draw=black!60, fill=white, very thick, minimum size=5mm},
important/.style={rectangle, double, draw=black!60, fill=white, very thick, minimum size=5mm},
]
\node[square,fill=yellow!10] (counting) 
    {\twoline{Counting lemma}{(\cremph{lem:counting})}};
\node[important,fill=blue!8] (regularity) [right=1cm of counting]
    {\twoline{Regularity lemma}{(\cremph{thm:regularity})}};
\node[square,fill=blue!8] (super-dense-removal) [yshift=-0.5cm][right = 2.2cm of regularity]
    {\twoline{``Super-dense'' graph removal}{lemma (\cremph{thm:super-dense-removal})}};
\node[square,fill=teal!9] (structure) [xshift=2cm] [below=of counting]
    {\twoline{Reduction to low complexity}{(\cremph{prop:reduction-to-structure})}};
\node[square,fill=red!12] (dense-removal) [yshift=0.5cm] [below=of super-dense-removal]
    {\twoline{Dense cycle removal lemma \vphantom{y}}{(\cremph{thm:dense-k-cycle-removal})}};
\node[important,fill=orange!15] (sparse-removal) [xshift=4cm][below=of structure]
    {\twoline{Sparse cycle removal lemma}{ (\cremph{thm:sparse-k-cycle-removal})}};
\node[square,fill=orange!15] (sparse-removal-2) [right=2cm of sparse-removal]
    {\cremph{thm:sparse-removal-odd-cycles}};
\node[square,fill=teal!9] (transference-free) [xshift=-2.8cm] [below=of structure]
    {\twoline{Transference-free results \vphantom{p}}{({\cremph{thm:expected-cycles-paths}\vphantom{y}})}};

\draw[-triangle 45] (regularity.south) |- ([yshift=0.5cm]structure.north) -| (structure.north);
\draw[-triangle 45] (counting.south) |- ([yshift=0.5cm]structure.north) -| (structure.north);
\draw[-triangle 45] (regularity.east) -- (super-dense-removal.west);
\draw[-triangle 45] (structure.west) -| (transference-free.north);
\draw[-triangle 45] (structure.south) |- ([yshift=0.5cm]sparse-removal.north) -| (sparse-removal.north);
\draw[-triangle 45] (dense-removal.south) |-
  ++(0,-0.45cm) 
  node[xshift=-4.1cm,yshift=0.2cm]{``transference''}
  |- ([yshift=0.5cm]sparse-removal.north) 
  -| (sparse-removal.north);
\draw[-triangle 45] (sparse-removal) -- (sparse-removal-2);
\end{tikzpicture}
\caption{Relationships between main results (\emph{arrows show logical implications}).}
\label{fig:results}
\end{figure}
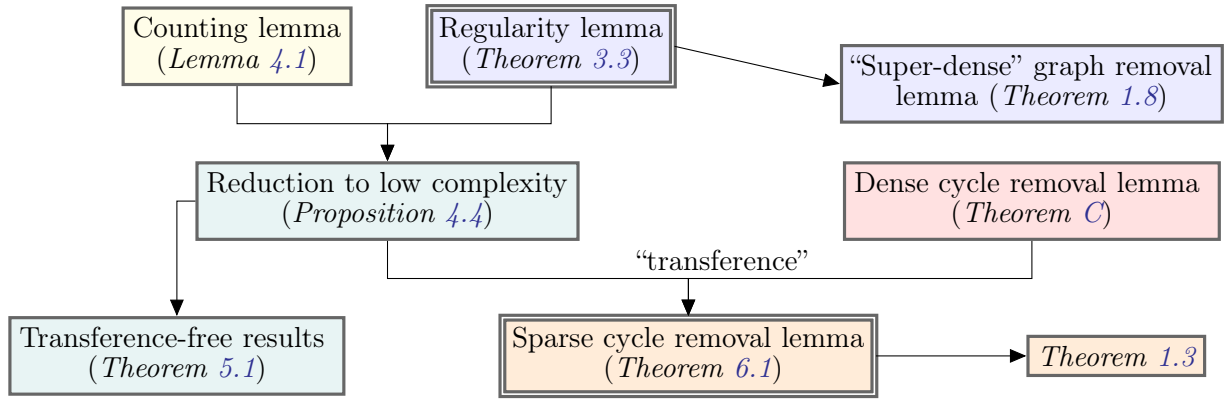

To better understand the regime of applicability of our work, let us mention a combinatorial characterization of $||f||_{S^q}$ when $q = 2k$ is an even positive integer: using indices modulo $k$, one has
\begin{equation} \label{eq:Schatten-2k}
    ||f||_{S^{2k}}^{2k} = \E_{\substack{x_1, \ldots, x_k \\ y_1, \ldots, y_k}} \left[\prod_{i=1}^k f(x_i, y_i) \bar{f(x_{i+1}, y_i)} \right],
\end{equation}
which generalizes \cref{eq:gowers}, and can be viewed as a normalized count of homomorphic copies of $2k$-cycles in the weighted bipartite graph induced by $f$ (this is closely related to the \emph{trace method} from spectral graph theory \cite{dsouza2013combinatorial,cioabua2006closed}). Furthermore, due to the characterization of spectral norms, one has
\begin{equation} \label{eq:Schatten-infty}
    ||f||_{S^\infty} = \sup_{||g||_{L^2} = ||h||_{L^2} = 1} \left\vert \E_{x,y} \left[f(x,y) g(x) h(y) \right] \right\vert,
\end{equation}
where $g$ and $h$ range over the complex functions on $X$, respectively $Y$, with $L^2$ norms equal to $1$; note the resemblance to the cut norm in \cref{eq:cut}. In particular,
\begin{itemize}
    \item[$(i)$.] $||f||_{S^2}$ is the $L^2$ norm of $f$;
    \item[$(ii)$.] $||f||_{S^4}$ is the $\square^2$ Gowers box norm of $f$;
    \item[$(iii)$.] $||f||_{S^6}$ is the cube root of the $L^2$ norm of the \emph{dual} function of $f$ (see, e.g., \cite{tao2007ergodic,tao2006gaussian});
    \item[$(iv)$.] $||f||_{S^\infty}$, which we may call the \emph{spectral norm} of $f$, is at least equal to the cut norm $||f||\sq$. 
\end{itemize}
Moreover, as they are the $L^q$ norms of the same sequence, the Schatten $q$-norms of a given function are nonincreasing in $q$ and obey interpolation bounds. 
Crucially, for $q > 2$, the norm $||\one_{G}/p||_{S^q}$ often remains bounded as $p = \E[\one_G] \searrow 0$, which makes \cref{thm:regularity-intro} applicable to sparse graphs. 


\subsection{Results about graphs} 
To demonstrate the applicability of our Schatten-norm-based regularity method (including the regularity lemma in \cref{thm:regularity} and the counting lemma in \cref{lem:counting}), we use it to derive various new graph removal lemmas, particularly in sparse settings.

\begin{remark}
\cref{lem:counting} also partially answers a question of Conlon, Fox, Sudakov and Zhao \cite{conlon2021graphs} about two-sided $H$-counting lemmas in $4$-cycle-free graphs, when $H$ is a cycle or a path; see \cref{cor:two-sided-counting} and the remark preceding it. In fact, \cref{cor:two-sided-counting} applies to any graph with $O(n^2)$ $4$-cycles, but it requires a slightly stronger uniformity assumption than in \cite{conlon2021graphs}. 
\end{remark}

\begin{notation}[Graphs] \label{not:graphs}
Unless specified otherwise, all graphs mentioned in this paper are finite, unweighted, undirected, and simple. We denote by $V(G)$ and $E(G)$ the sets of vertices and edges of a graph $G$; thus the elements of $E(G)$ are two-element subsets of $V(G)$. We define $\one_G : V(G) \times V(G) \to \{0, 1\}$ by setting $\one_G(x, y) = 1$ iff $\{x, y\} \in E(G)$, and we write $e_G(V, W) := \sum_{x \in V, y \in W} \one_G(x, y)$ for any $V, W \subset V(G)$. Here are some particular graphs:
\begin{itemize}
    \item For $k \ge 3$, we denote by $C_k$ the cycle of length $k$ (with $k$ vertices and $k$ edges).
    \item For $k \ge 1$, we denote by $P_k$ the path of length $k$ (with $k+1$ vertices and $k$ edges).
    \item For $s, t \ge 1$, we denote by $K_{s,t}$ the complete bipartite graph on $s + t$ vertices (e.g., $K_{2,2} \cong C_4$).
\end{itemize}
We also distinguish between several meanings for counting instances of a graph in another graph:
\begin{itemize}
    \item[(1).] By the number of \emph{copies} of a graph $H$ in a graph $G$ (or simply the number of $H$'s in $G$), we mean the number of subgraphs of $G$ which are isomorphic to $H$. 
    \item[(2).] We write $\Inj(H, G)$ for the number of \emph{injective copies} of $H$ in $G$, i.e.\ the number of monomorphisms from $H$ to $G$ (or embeddings of $H$ in $G$).
    \item[(3).] We write $\Hom(H, G)$ for the number of \emph{homomorphic copies} of $H$ in $G$, i.e.\ the number of homomorphisms from $H$ to $G$.
\end{itemize}
Formally, we write
\[
    \Hom(H, G) := \sum_{x_v \in V(G), \forall v \in V(H)}\  \prod_{\{u, v\} \in E(H)} \one_G(x_u, x_v),
\]
\[
    \Inj(H, G) := \sum_{x_v \in V(G), \forall v \in V(H)}\  \one_{x_u \neq x_v,\forall u \neq v} \prod_{\{u, v\} \in E(H)} \one_G(x_u, x_v).
\]
Hence when $V(G)$ is equipped with the uniform probability distribution, \cref{eq:Schatten-2k} implies that
    \begin{equation}\label{eq:2k-cycles}
        \forall k \in \Z_{\ge 2} : \qquad 
        \Hom(C_{2k}, G) = n^{2k} ||\one_{G}||_{S^{2k}}^{2k}.
    \end{equation}
\end{notation}

\begin{remark}
$\Inj(H, G)$ is the same as the number of copies of $H$ in $G$ up to a factor depending on $H$; however, $\Inj(H, G)$ can be much smaller than $\Hom(H, G)$ when $G$ is sparse. Also, in light of \cref{eq:2k-cycles}, a simple consequence of the monotonicity of Schatten $q$-norms in $q$ is that for a fixed $n$-vertex graph $G$, the quantity $\frac{1}{n}\Hom(C_{2k},G)^{1/2k}$ is nondecreasing in $k$.
\end{remark}

Throughout this subsection, $n$ is understood to be a positive integer, and we write $o(1) = o_{n \to \infty}(1)$ for a function of $n$ which vanishes as $n \to \infty$ (see \cref{subsec:asymptotic} for other asymptotic notation); thus one can understand our results either as $\epsilon$-$\delta$ statements, or as statements about an implicit sequence of graphs $\{G_n\}_{n \ge 1}$. We will also use a parameter $p > 0$ which may depend on $n$, and which essentially plays the role of the edge density of $G$ (although we do not explicitly require that $p = \E[\one_G]$); thus for our graphs to be sparse, we need $p = o(1)$. 

Our main removal lemma, which uses a transference argument based on a dense pairs condition (and has tower-exponential implied dependencies), is given below; see \cref{thm:sparse-k-cycle-removal} for a more general version, applicable to weighted and $k$-partite graphs. 

\begin{theorem}[Sparse removal lemma for odd-length cycles] \label{thm:sparse-removal-odd-cycles}
For any $\epsilon, K, C > 0$ and integer $k \ge 2$, there exist $n_0, \delta, \tau > 0$ such that the following holds. Let $n \ge n_0$, $p > 0$, and $G$ be an $n$-vertex graph such that
\[
    \Hom(C_{2k}, G) \le C p^{2k} n^{2k} \qquad \text{and} 
    \qquad 
    \Inj(C_{2k+1}, G) \le \delta p^{2k+1} n^{2k+1}.
\]
Suppose that for any partition of $V(G)$ into parts $V_1, V_2, \ldots$ with at least $\tau n$ elements each, a total of at most $\epsilon p n^2/8$ edges of $G$ lie between `dense' pairs $(V_i, V_j)$ with $i \neq j$ and $e_G(V_i, V_j) \ge K p |V_i| |V_j|$. Then $G$ can be made $\{C_3, C_5, \ldots, C_{2k+1}\}$-free by removing at most $\epsilon pn^2$ edges.
\end{theorem}

\begin{remark}
A similar result holds for paths and even-length cycles, except that in the conclusion, the graph $G$ becomes empty after removing $\epsilon p n^2$ edges (essentially because paths and even-length cycles are bipartite). In that case, it makes more sense to state the contrapositive statement (concluding that $G$ contains many injective copies of $P_j$ and $C_{2k}$), given in \cref{thm:expected-cycles-paths}. We note that \cref{thm:expected-cycles-paths} avoids the transference step (and thus also the tower-exponential dependencies).
\end{remark}

\begin{remark}
\cref{thm:sparse-removal-odd-cycles} may be compared to \cite[Proposition 4.2]{conlon2021regularity}, which can be recovered by taking $k = 2$ and $p \gg n^{-1/2}$. The dense pairs condition in \cite[Proposition 4.2]{conlon2021regularity} is more restrictive, since:
\begin{itemize} 
    \item[(1).] It applies to all partitions with at most $M$ parts, while our condition only concerns partitions with large parts (and one can always take $M = 1/\tau$);
    \item[(2).] It includes the contribution of the diagonal pairs $(V_i, V_i)$.
\end{itemize} 
Point $(1)$ will be relevant for \cref{cor:subgraphs-pseudorandom}, while $(2)$ will be relevant for \cref{cor:graphs-few-4-cycles,cor:smooth-free-graphs}.
\end{remark}

\begin{corollary}[Subgraphs of mildly-pseudorandom graphs] \label{cor:subgraphs-pseudorandom}
Fix $k \in \Z_{\ge 2}$. Let $p > 0$ and $G$ be an $n$-vertex graph with $\one_{G} \ll p\nu$, for some $\nu : V(G) \times V(G) \to [0, \infty)$ satisfying
\[
    ||\nu||_{S^{2k}} \ll 1 
    \qquad\qquad 
    \text{and} 
    \qquad\qquad 
    ||\nu - 1||\sq = o(1).
\]
If $G$ has $o(p^{2k+1} n^{2k+1})$ $C_{2k+1}$'s, it can be made $\{C_3, C_5, \ldots, C_{2k+1}\}$-free by removing $o(pn^2)$ edges.
\end{corollary}

\begin{remark}
To make sense of the assumptions about $\nu$ in \cref{cor:subgraphs-pseudorandom}, $V(G)$ is implicitly equipped with uniform probability. Note that one can replace these two assumptions with the single (stronger) condition $||\nu - 1||_{S^{2k}} = o(1)$; since $k \ge 2$, this is more general than the standard linear forms conditions \cite{conlon2014green,tao2006gaussian,tao2007ergodic}. For a better comparison with other notions of pseudorandomness, see the remark after \cref{cor:relative-k-cycle-removal} (which also considers the case of removing triangles, i.e., $2k+1 = 3$).
\end{remark}

\begin{remark}
One can view $p \nu$ as the indicator function of a weighted graph $\Gamma$, of which $G$ is a subgraph. Conlon, Fox and Zhao \cite{conlon2014extremal} proved similar $H$-removal lemmas (for a general graph $H$) under the assumption that $\Gamma$ is \emph{$(p, o(p^{c(H)}n))$-jumbled}, essentially meaning that $||\nu - 1||_{S^\infty} = o(p^{c(H) - 1})$; here $c(H) = 1 + 1/(2k-2)$ when $H$ is a cycle of length $2k+1 \ge 5$. Our pseudorandomness constraint $||\nu - 1||\sq = o(1)$ is strictly weaker (both in the type of norm bounded and in the bound itself), though it comes at the expense of the (fairly weak) additional assumption $||\nu||_{S^{2k}} \ll 1$.
\end{remark}

When $p \ge n^{-1/2}$, our removal lemmas also apply to graphs with few $4$-cycles; the corollary below is closely related to the results in \cite{conlon2021regularity,conlon2021graphs}, although we use a different regularity method.

\begin{corollary}[Graphs with not too many $4$-cycles] \label{cor:graphs-few-4-cycles}
Fix integers $j \ge 1$, $k \ge 2$, $\ell \ge 3$. For any $p \ge n^{-1/2}$ and any $n$-vertex graph $G$ with $O(p^4 n^4)$ $C_4$'s, the following hold true.
\begin{itemize}
    \item[$(i)$.] If $G$ has $o(p^{j} n^{j+1})$ $P_{j}$'s, then it has $o(pn^2)$ edges.
    \item[$(ii)$.] If $G$ has $o(p^{2\ell}n^{2\ell})$ $C_{2\ell}$'s, then it has $o(pn^2)$ edges.
    \item[$(iii)$.] If $G$ has $o(p^{2k+1} n^{2k+1})$ $C_{2k+1}$'s and every bipartite subgraph of $G$ has at most $o(p^4n^4)$ $C_4$'s,
    then $G$ can be made $\{C_3, C_5, \ldots, C_{2k+1}\}$-free by removing $o(pn^2)$ edges.
\end{itemize}
\end{corollary}

\begin{remark}
Parts $(i)$ and $(ii)$ of \cref{cor:graphs-few-4-cycles} follow from \cref{thm:expected-cycles-paths}, while part $(iii)$ follows from \cref{thm:sparse-removal-odd-cycles}; the latter generalizes the fact that 
\begin{itemize}
    \item[$(iii').$] \emph{If $G$ has $o(p^4n^4)$ $C_4$'s and $o(p^{2k+1} n^{2k+1})$ $C_{2k+1}$'s, then it can be made $\{C_3, C_5, \ldots, C_{2k+1}\}$-free by removing $o(pn^2)$ edges.}
\end{itemize}
When $2k+1 = 5$, $(iii')$ recovers the main result of Conlon, Fox, Sudakov and Zhao's first aforementioned paper (see \cite[Theorem 1.2]{conlon2021regularity}). For a general $k$, $(iii')$ essentially follows from the same authors' follow-up work on counting lemmas in $C_4$-free graphs. 
Another generalization of $(iii')$, which applies to $k$-partite graphs (as required in arithmetic applications), is given in \cref{cor:k-cycle-removal-few-4-cycles}. 
\end{remark}

In particular, \cref{cor:graphs-few-4-cycles} applies to $C_4$-free graphs. More generally, we give an application of \cref{thm:sparse-removal-odd-cycles} to graphs lacking a \emph{quasi-smooth} family of bipartite graphs, following \cite{jiang2021bipartite,allen2014turan}.

\begin{definition}[Extremal numbers and quasi-smooth families] \label{def:quasi-smooth}
Given positive integers $m, n$ and a family of bipartite graphs $\mF$, we denote by $\ex(n, \mF)$ (respectively, $\ex(m, n, \mF)$) the maximal number of edges in a graph on $n$ vertices (respectively, a bipartite graph on $m+n$ vertices) which is $\mF$-free, meaning that it contains no (injective) copies of any graph in $\mF$. We then say that $\mF$ is $(\alpha, \beta)$-quasi-smooth for some reals $2 > \alpha > \beta \ge 1$ iff for $n \ge m \ge 1$, one has
\[
    \ex(m, n, \mF) \ll m n^{\alpha - 1} + n^\beta.
\]
\end{definition}

\begin{remark}
We do not explicitly require that $\ex(n, \mF) \gg n^\alpha$ (as in \cite{jiang2021bipartite}), although our results will only be nontrivial in that case.
So by our definition, given integers $t \ge s \ge 1$, F\"uredi \cite{furedi1996upper} showed that $\{K_{s,t}\}$ is $(\alpha, \beta)$-quasi-smooth with $\alpha = 2 - 1/s$ and $\beta = 2 - 2/s$; we only know the corresponding lower bound $\ex(n, \{K_{s,t}\}) \gg_{s,t} n^\alpha$ when $t$ is large enough in terms of $s$ \cite{alon1999norm}. 
\end{remark}

\begin{corollary}[Graphs free of a quasi-smooth family] \label{cor:smooth-free-graphs}
Fix an integer $k \ge 2$,  
reals $2 > \alpha > \beta \ge 1$, and let $\mF$ be an $(\alpha, \beta)$-quasi-smooth family. Then an $\mF$-free $n$-vertex graph $G$ with 
\[
    \Hom(C_{2k}, G) = O(n^{2k(\alpha - 1)})
    \qquad 
    \text{and} 
    \qquad 
    \Inj(C_{2k+1}, G) = o(n^{(2k+1)(\alpha - 1)})
\]
can be made $\{C_3, C_5, \ldots, C_{2k+1}\}$-free by removing $o(n^\alpha)$ edges.
\end{corollary}

\begin{remark}
As a possible extension of their work on $C_4$-free graphs, Conlon, Fox, Sudakov and Zhao ask \cite{conlon2021graphs} about the existence of similar removal lemmas in $H$-free graphs for various bipartite graphs $H$, including $C_{2k}$ and $K_{s,t}$ for $t \ge s > 2$. \cref{cor:smooth-free-graphs} gives a partial answer to this question in the latter case (taking $\mF = \{K_{s,t}\}$, $\alpha = 2 - 1/s$ and $\beta = 2 - 2/s$); we expect the same to hold true when $\mF = \{C_{2k}\}$, but this is not known to be quasi-smooth for any $k \ge 3$ \cite{jiang2021bipartite,allen2014turan}.
\end{remark}

Separately from our results for sparse graphs, we also state a natural consequence of our regularity lemma in the regime where $q < r = 2$; this is a graph removal lemma for a family of ``super-dense'' graphs (with bounded Schatten $q$-norm for some $q < 2$), involving no tower-exponential bounds. Its proof is left to \cref{sec:super-dense-removal}.


\begin{theorem}[Super-dense graph removal lemma] \label{thm:super-dense-removal}
For any graph $H$, $q \in [1, 2)$, $C > 0$ and $\epsilon \in (0, 1)$, there exists 
\[
    \delta \ge \exp \left(- O_{H,q,C}\left(\epsilon^{-O_{H,q}(1)} \right) \right)
\]
such that the following holds. Let $G$ be an $n$-vertex graph satisfying $||\lambda^G||_q \le C n$, where $(\lambda^G_i)_{i \ge 1}$ are the eigenvalues of $G$'s adjacency matrix. 
If $\Hom(H, G) \le \delta n^{|V(H)|}$, then one remove at most $\epsilon n^2$ edges from $G$ to eliminate all homomorphic copies of $H$.
\end{theorem}

\begin{remark}
As shown in \cite{nikiforov2012extremal}, most graphs $G$ have $||\lambda^G||_q \asymp_q n^{\frac{1}{2} + \frac{1}{q}}$ when $q \in [1, 2)$, which is not $O(n)$; however, \cref{thm:super-dense-removal} should be interesting when $q = 2 - \eps$, since one always has $||\lambda^G||_2 \le n$. 
\end{remark}

\subsection{Corollaries in additive combinatorics}

We now give examples of applications for each of \cref{cor:subgraphs-pseudorandom,cor:graphs-few-4-cycles,cor:smooth-free-graphs} (more precisely, of their $k$-partite generalizations), to finding translation-invariant patterns in certain sparse subsets of abelian groups. This makes use of bipartite graphs with weights of the form $f(x, y) := \varphi(x+y)$, for some function $\varphi : G \to [0, \infty)$, where $G$ is an abelian group of order $n$. Equipping $G \times G$ with the uniform probability measure, we extend the meaning of the cut and Schatten norms by setting
\[
    ||\varphi||\sq := ||f||\sq \qquad\quad \text{and} \qquad\quad 
    ||\varphi||_{S^q} := ||f||_{S^q},
    \qquad 
    \forall q \in [1, \infty].
\]
In this context, $||\varphi||_{S^q}$ is a normalized $L^q$ norm on the Fourier coefficients of $\varphi$, due to \cref{lem:Fourier-singular}. When $q = 2k$ is an even positive integer, we also have 
\[
    n^{2k-1} ||\varphi||_{S^{2k}}^{2k} = \sum_{a_1+\cdots+a_k = b_1+\cdots+b_k} \prod_{i=1}^k \varphi(a_i) \bar{\varphi(b_i)},
\]
which is often referred to as \emph{$k$-fold additive energy} (especially when $\varphi$ is an indicator function).

We first give an application in the setting of a pseudorandom majorant. As in \cref{cor:subgraphs-pseudorandom}, the key feature is that the pseudorandomness condition is mild and easy to verify, compared to the linear forms conditions in \cite{conlon2014green,tao2006gaussian,tao2007ergodic} and the jumbledness requirements in \cite{conlon2014extremal}. In particular, one can apply the following result to find solutions to linear equations in the primes (see also \cref{cor:linear-patterns-pseudorandom-weighted} and the remark after \cref{cor:relative-k-cycle-removal} for the case $k = 3$). 


\begin{corollary}[Linear patterns in subsets of mildly-pseudorandom sets] \label{cor:linear-patterns-pseudorandom}
For any $\epsilon, C > 0$ and $k, \ell \in \Z$ with $k > 2\ell \ge 4$, there exists $\delta > 0$ such that the following holds. Let $a_1, \ldots, a_k$ be nonzero integers with $a_1 + \cdots + a_k = 0$, $G$ be an abelian group of order $n$ with $\gcd(n, a_1 \cdots a_k) = 1$, and $S \subset G$, $\nu : G \to [0, \infty)$ satisfy
\begin{itemize}
    \item[$(i)$] \emph{($\nu$ majorizes $\one_S$ after scaling).}  $\epsilon \one_S \le \frac{|S|}{n}\nu$,
    \item[$(ii)$] \emph{($\nu$ is mildly-pseudorandom).}  $||\nu||_{S^{2\ell}} \le C$ and $||\nu - 1||\sq \le \delta$. 
\end{itemize}
Then $S$ contains $\ge \delta |S|^k/n$ solutions in $x_1, \ldots, x_k$ to the linear equation $a_1 x_1 + \cdots + a_k x_k = 0$. 
\end{corollary}

\begin{remark}
Choosing $C > 2$ and $\delta < 1$, condition $(ii)$ follows from the simpler constraint $||\nu - 1||_{S^{2\ell}} \le \delta$.
\end{remark}

\begin{remark}
The equation $a_1 x_1 + \cdots + a_k x_k = 0$ has $|S|$ trivial solutions $x_1 = \cdots = x_k \in S$, so \cref{cor:linear-patterns-pseudorandom} only has nontrivial content when $\delta |S|^k / n > |S|$ (i.e., $|S| \gg n^{1/(k-1)}$). Also, note that a random subset of $G$ with $|S|$ elements is expected to contain $|S|^k/n$ solutions to such equations.
\end{remark}

Before we state our other applications, let us recall that a subset $S$ of an abelian group is called \emph{Sidon} if and only if it contains no nontrivial solutions to the equation $x + t = y + z$ (the trivial solutions being $(x, t) = (y, z)$ and $(x, t) = (z, y)$); this is equivalent to the absence of $4$-cycles in the associated bipartite graph, where $(x, y)$ is an edge whenever $x + y \in S$. Such sets have $|S| \ll \sqrt{n}$.

Prendiville \cite{prendiville2022solving} proved a result similar to \cref{cor:linear-patterns-pseudorandom}, where $S$ is a large ``almost-Sidon'' subset of $[n] := \{1, \ldots, n\}$ (meaning that $S$ contains few nontrivial solutions to $x + t = y + z$, and $|S| \gg \sqrt{n}$). One can also deduce this result, with weaker implied bounds, from the $k$-partite version of \cref{cor:graphs-few-4-cycles}. Below we consider another natural weakening of the Sidon condition, noting that a set is Sidon iff it contains at most one pair of the form $(x, x+h)$ for each $h \neq 0$.

\begin{corollary}[Linear patterns in generalized Sidon sets] \label{cor:linear-patterns-generalized-Sidon}
For any $\epsilon, C > 0$, $k \in \Z_{\ge 5}$, and nonzero integers $a_1, \ldots, a_k$ with $a_1 + \cdots + a_k = 0$, there exists $\delta > 0$ such that the following holds. Let $n \in \Z_{\ge 1}$ and $S \subset [n]$ satisfy
\begin{itemize}
    \item[$(i)$] \emph{($S$ is large).} $|S| \ge \epsilon n^{1/2}$,
    \item[$(ii)$] \emph{(Sidon-type condition).} For each $h \neq 0$, $S$ contains at most $C$ pairs of the form $(x, x+h)$.  
\end{itemize}
Then $S$ contains $\ge \delta |S|^k/n$ solutions in $x_1, \ldots, x_k$ to the linear equation $a_1 x_1 + \cdots + a_k x_k = 0$.
\end{corollary}

\begin{remark}
\cref{cor:linear-patterns-generalized-Sidon} follows from (the $k$-partite version of) \cref{cor:smooth-free-graphs} and the fact that $K_{2,t}$ is quasi-smooth for all $t \ge 2$. One obtains a more general statement using a general complete bipartite graph $K_{s,t}$ instead; see \cref{cor:linear-patterns-st-Sidon}. 
\end{remark}

Finally, as an application of \cref{cor:graphs-few-4-cycles}, we give a two-dimensional analogue of Prendiville's result \cite{prendiville2022solving}, related to the multidimensional Roth theorem \cite{solymosi2003note,solymosi2004note,gowers2007hypergraph}. Given slopes $s_1, \ldots, s_k \in \Q \cup \{\infty\}$, let us call a sequence of lattice points $P = (p_1, \ldots, p_k) \in (\Z^2)^k$ an $(s_1, \ldots, s_k)$-polygon iff for each $i \in \Z/k\Z$, the points $p_i$ and $p_{i+1}$ lie on a line of slope $s_i$. We say that $P$ is non-degenerate when $p_1, \ldots, p_k$ are pairwise distinct.

\begin{corollary}[Polygonal patterns in large sets with few parallelograms] \label{cor:polygonal-patterns}

For any $\epsilon > 0$, $k \in \Z_{\ge 5}$, and slopes $s_i \in \Q \cup \{\infty\}$ with $s_i \neq s_{i+1}$ for $i \in \Z/k\Z$, there exists $\delta > 0$ such that the following holds. Let $n \in \Z_{\ge 1}$ and $S \subset [n] \times [n]$ satisfy
\begin{itemize}
    \item[$(i)$] \emph{($S$ is large).} $|S| \ge \epsilon n^{3/2}$,
    \item[$(ii)$] \emph{(Few parallelograms).} For each $i \in \Z/k\Z$, $S$ contains at most $\delta |S|^4/n^4$ non-degenerate $(s_i, s_{i+1}, s_i, s_{i+1})$-polygons.
\end{itemize}
Then $S$ contains at least $\delta |S|^k / n^k$ (possibly degenerate) $(s_1, \ldots, s_k)$-polygons.
\end{corollary}

\begin{remark}
It can be shown, using \cref{eq:4-cycle-counts}, that any set $S \subset [n] \times [n]$ which satisfies condition $(ii)$ of \cref{cor:polygonal-patterns} with small enough $\delta$ has $|S| \ll n^{3/2}$; thus \cref{cor:polygonal-patterns} concerns the densest (up to a constant) sets with few parallelograms.
\end{remark}

\section{Notation and preliminaries} \label{sec:preliminaries}

\subsection{Asymptotic notation, sets and functions} \label{subsec:asymptotic}
We keep the notation $[n] := \{1, \ldots, n\}$ when $n$ is a positive integer. We may add restrictions in the subscripts of $\Z, \Q, \R, \C$ to denote the appropriate subsets; for instance, $\Z_{\ge a}$ denotes the set of all integers greater than or equal to $a$. Given a set $S \subset \C$ and some $k \in \C$, we also write $kS := \{ks : s \in S\}$; for instance, $2\Z$ denotes the set of all even integers. For $r \in \R$, we write $\lfloor r \rfloor$ for the greatest integer smaller than or equal to $r$.

Given any functions $f : X \to \C$ and $g : X \to [0, \infty)$, we write $f = O(g)$ or $f \ll g$ iff there is an absolute constant $C$ such that $|f| \le Cg$ pointwise. We write $f \asymp g$ when $f \ll g$ and $g \ll f$. If the implied constants may depend on a parameter $q$, we indicate this using subscripts: $f = O_q(g)$, $f \ll_q g$, respectively $f \asymp_q g$. If $X = \Z_{\ge N}$ for some $N \ge 1$, we write $f = o(g) = o_{n \to \infty}(g)$ iff $f(n)/g(n) \to 0$ as $n \to \infty$. 

Given a subset $S \subset X$ of an implicit ambient space, we write $\one_S : X \to \{0, 1\}$ for the indicator function of $S$. When $S$ is an statement such as $x \neq y$, we may write $\one_S$ for its truth value. We also keep all the graph-related notations from \cref{not:graphs}; in particular, when $G$ is a graph, we have $\one_G(x, y) = \one_{x \neq y} \one_{\{x,y\} \in E(G)}$. 

Given a measurable function $f : (X, \Sigma, \mu) \to \C$ and some $q \in [1, \infty]$, we denote the $L^q$ norm of $f$ by $||f||_{L^q}$ when $\mu$ is a probability measure, and by $||f||_q$ when $\mu$ is the counting measure (to avoid potential confusion). We also write $||v||_q$ for the $L^q$ norms of vectors $v \in \C^n$.

Finally, we define the Fourier transform of a function $\varphi : G \to \C$ on a finite abelian group $G$ by
\begin{equation} \label{eq:Fourier}
    \hat{\varphi}(\chi) := \sum_{g \in G} f(g) \bar{\chi(g)},
    \qquad\qquad 
    \forall \chi \in \hat{G}.
\end{equation}

\subsection{Probability spaces, partitions and entropy} \label{subsec:partitions}

We will mostly work with finite probability spaces $(X, \Sigma, \P)$. If the $\sigma$-algebra $\Sigma$ is not specified, it should be understood that $\Sigma = 2^X$ (so $\P$ can be identified with a function from $X$ to $[0, 1]$ with total sum $1$). If the probability measure $\P$ is not specified, it should be understood that $\P$ is the uniform probability measure on $X$ (especially when $X = V(G)$ for some graph $G$). Given an absolutely integrable function $f : X \to \C$, we write $\E[f]$ for the expected value of $f$, and we may omit the square brackets when the content of the expectation is unambiguous. Given an event $A \in \Sigma$ with positive probability, we also write
\[
    \E[f \mid A] := \frac{\E[f \one_A]}{\E[\one_A]}
\]
for the conditional expected value of $f$ given $A$. If $X = X_1 \times \cdots \times X_k$ where $(X_i, \Sigma_i, \P_i)$ are probability spaces, then in an expression like $\E_{x_1, \ldots, x_k} f(x_1, \ldots, x_k)$, it should be understood that the $x_i$'s are sampled independently from $(X_i, \P_i)$ (thus $\E_{x_1, \ldots, x_k} f(x_1, \ldots, x_k) = \E[f]$, using the implicit product probability measure on $X_1 \times \cdots \times X_k$).

A \emph{partition} $\mP$ of a set $X$ is a collection of pairwise disjoint subsets of $X$ (called \emph{parts}) with union equal to $X$. Hence $|\mP|$ denotes the cardinality of this collection, i.e., the number of parts. We also denote $\diag(\mP) := \bigcup_{A \in \mP} (A \times A)$, which is a subset of $X \times X$.

If $(X, \Sigma)$ is a measurable space, we will only work with measurable partitions of $X$ (meaning that all parts belong to $\Sigma$); conversely, we write $\Sigma(\mP)$ for the smallest $\sigma$-algebra containing all the parts of a partition $\mP$. Given a measurable function $f : (X, \Sigma) \to S$, we write $\Part(f)$ for the partition of $X$ induced by $f$, i.e., $\Part(f) := \{f^{-1}(s) : s \in S\}$. Given an absolutely integrable function $f : (X, \Sigma, \P) \to S$, we write $\E(f \mid \mP) := \E(f \mid \Sigma(\mP))$ for the conditional expectation of $f$ with respect to $\mP$; if $X$ is finite, we can view $\E(f \mid \mP)$ as an averaged function over parts of $\mP$:
\begin{equation} \label{eq:conditional-exp}
    \E(f \mid \mP)(x) = \E[f \mid P_x],
    \qquad 
    \text{where } x \in P_x \in \mP.
\end{equation}
Given two partitions $\mP, \mQ$ of a set $X$, we say that $\mP$ is finer than $\mQ$ iff $\Sigma(\mP)$ is richer than $\Sigma(\mQ)$ (i.e., $\Sigma(\mP)$ contains all parts of $\mQ$). We write $\mP \vee \mQ$ for the common refinement of $\mP$ and $\mQ$, i.e., the coarsest partition which is finer than or equal to both $\mP$ and $\mQ$. Given partitions $\mP$ of $X$ and $\mQ$ of $Y$, we write $\mP \otimes \mQ$ for the product partition $\{A \times B : A \in \mP, B \in \mQ\}$ of $X \times Y$.

Finally, given a probability space $(X, \Sigma, \mP)$ and a measurable partition $\mP$ of $X$ with finitely many parts, we define its (Shannon) \emph{entropy} by
\[
    \H(\mP) := \sum_{A \in \mP} \P(A) \log \frac{1}{\P(A)},
\]
which increases when we pass to a finer partition; thus we have $0 = \H(\{X\}) \le \H(\mP) \le \H(\{\{x\} : x \in X\}) = \log |X|$. Given a finitely-valued random variable $f : X \to S$, we also define $\H(f) := \H(\Part(f)) = - \sum_{s \in S} \P(f = s) \log \P(f = s)$ (this is the most common definition of entropy). We will only use the inequality
\begin{equation} \label{eq:entropy-inequalities}
    \H(\mP \vee \mQ) \le \H(\mP) + \H(\mQ),
\end{equation}
which translates to $\H(f, g) \le \H(f) + \H(g)$ in random-variable notation; this follows from Jensen's inequality, and the difference $\I(\mP : \mQ) : \H(\mP) + \H(\mQ) - \H(\mP \vee \mQ)$ is called the \emph{mutual information} between $\mP$ and $\mQ$. We note that this property is also true for the simpler function $\log |\mP|$, but entropy is less sensitive to parts with small probability (which is crucial to our regularity lemma).

We write $M^T$ and $M^*$ for the transpose and the conjugate transpose of a matrix $M$. We now introduce a couple of matrices related to random variables and partitions, all in finite probability spaces with full $\sigma$-algebras.

\begin{definition}[Normalized matrices] \label{def:matrices}
Given finite probability spaces $(X, \P_X)$, $(Y, \P_Y)$, a function $f : X \times Y \to \C$, and an ordering $(x_i)_{1 \le i \le |X|}$ and $(y_j)_{1 \le j \le |Y|}$ of the elements of $X$ and $Y$, we define $\Mat(f)$ as the $|X| \times |Y|$ complex matrix with entries
\[
    \Mat(f)_{i,j} := f(x_i, y_j) \sqrt{\P_X(x_i) \P_Y(y_j)}.
\]
\end{definition}

\begin{remark}
To motivate this normalization, one can check that given any $k \ge 2$, finite probability spaces $(X_i, \P_i)$ (with orderings of their elements), and functions $f_i : X_i \to \C$ for $i \in \Z/k\Z$, one has
\begin{equation} \label{eq:k-cycle-count}
    \E_{x_1, \ldots, x_k} \prod_{i=1}^k f_i(x_i, x_{i+1}) = \Tr\left(\Mat(f_1)\cdots \Mat(f_k) \right).
\end{equation}
(In fact, \cref{eq:k-cycle-count} would hold if we used any normalization $\Mat(f)_{i,j} := f(x_i, y_j) \P_X(x_i)^\alpha \P_Y(y_j)^{1-\alpha}$ for $\alpha \in [0, 1]$; but having a symmetric normalization with $\alpha = 1/2$ leads to helpful additional properties, such as $\Mat(f)^T = \Mat(f^T)$ when $f^T(y, x) := f(x, y)$.)
\end{remark}

\begin{definition}[Averaging matrices]
Let $(X, \P_X)$ be a finite probability space, and let $(x_i)_{1 \le i \le |X|}$ be an ordering of the elements of $X$. Then for any partition $\mP$ of $X$, we define the matrix $A_{\mP} \in [0, \infty)^{n \times n}$ (which also implicitly depends on $\P_X$ and on the chosen ordering $(x_i)_i$) by
\[
    (A_{\mP})_{i,j} := \begin{cases}
    \frac{\sqrt{\P_X(x_i)\P_X(x_j)}}{\P_X(P)}, & \exists\ P \in \mP :\ x_i, x_j \in P, \\
    0, &\text{otherwise}.
    \end{cases}
\]
Hence $A_\mP$ is a block-diagonal matrix with $|\mP|$ diagonal blocks, the sum of the squared entries in each block being $1$, and each block having rank $1$ (in particular, $\rank (A_\mP) = |\mP|$).
\end{definition}

\begin{lemma}[Properties of averaging matrices] \label{lem:averaging-matrices}
Fix finite probability spaces $(X, \P_X)$ and $(Y, \P_Y)$, as well as orderings of their elements. With the understanding that $\mP$ and $\mR$ are partitions of $X$ and $\mQ$ is a partition of $Y$, the following hold true.
\begin{itemize}
    \item[$(i)$.] For any $f : X \times Y \to \C$, one has
    \[
        \Mat(\E(f \mid \mP \otimes \mQ)) = A_{\mP} \Mat(f) A_{\mQ}.
    \]
    \item[$(ii)$.] If $\mR$ is finer than or equal to $\mP$, then
    \[
        A_\mP A_{\mR} = A_{\mR} A_\mP = A_\mP.
    \]
    In particular, one has $A_\mP^2 = A_\mP$, so $A_\mP$ is a projection. Another consequence of this is that $(A_{\mR} - A_\mP)^2 = A_{\mR} - A_\mP$, so $A_{\mR} - A_\mP$ is also a projection.
    \item[$(iii)$.] For any complex matrix $M$ with $|X|$ rows, denoting $||A||_{S^2} := \Tr(A A^*)^{1/2}$ for now, one has
    \begin{equation} \label{eq:energy-increment-matrices}
        ||(A_{\mR} - A_\mP)M||_{S^2}^2 = ||A_{\mR} M||_{S^2}^2 - ||A_\mP M||_{S^2}^2.
    \end{equation}
    \item[$(iv)$.] The nonzero singular values (and eigenvalues) of $A_\mP$ consist of $|\mP|$ ones. If $\mR$ is finer than $\mP$, then the nonzero singular values (and eigenvalues) of $A_\mR - A_\mP$ consist of $|\mR| - |\mP|$ ones.
\end{itemize}
\end{lemma}

\begin{proof}
Parts $(i)$ and $(ii)$ follow quickly from our definitions and \cref{eq:conditional-exp}.
Part $(iii)$ is Pythagoras' theorem; more precisely, we can expand
\[
\begin{aligned}
    \Tr\left((A_\mR - A_\mP)M M^* (A_\mR - A_\mP) \right)
    &= 
    \Tr\left((A_\mR - A_\mP)^2 M M^*\right)
    \\
    &= 
    \Tr\left((A_\mR - A_\mP) M M^*\right)
    \\
    &=
    \Tr(A_\mR M M^*) - \Tr(A_\mP M M^*)
    \\
    &= 
    \Tr(A_\mR M M^* A_\mR) - \Tr(A_\mP M M^* A_\mP),
\end{aligned}
\]
using the circular symmetry and the linearity of the trace. Part $(iv)$ follows from the fact that projections can only have eigenvalues of $0$ and $1$, and from \cref{eq:energy-increment-matrices} when $M$ is the identity matrix (thus $A_\mR - A_\mP$ has a total of $||A_\mR||_{S^2}^2 - ||A_\mP||_{S^2}^2 = |\mR| - |\mP|$ singular values equal to $1$).
\end{proof}

\subsection{Schatten norms and related energies} \label{subsec:Schatten}

\begin{definition}[Schatten norms of matrices]
Given $q \in [1, \infty)$ and a matrix $M \in \C^{m \times n}$ with singular values $\sigma_1 \ge \sigma_2 \ge \ldots \ge \sigma_{\min(m,n)}$, we define its \emph{Schatten norms} by
\[
    ||M||_{S^q} := \left(\sum_{j=1}^{\min(m,n)} \sigma_j^q\right)^{1/q}
    \qquad\text{and}\qquad 
    ||M||_{S^\infty} := \max_{1 \le j \le \min(m,n)} \sigma_j.
\]
\end{definition}

\begin{lemma} \label{lem:Schatten-norms-matrices}
For $q, q_i, r \in [1, \infty]$ and any suitably-sized complex matrices $M, A_1, \ldots, A_k$, the following hold true.
\begin{itemize}
    \item[$(i)$.] \emph{(Norm)} $||\cdot||_{S^q}$ is a unitary-invariant norm on $\C^{m \times n}$, and $||M||_{S^q} = ||M^T||_{S^q} = ||M^*||_{S^q}$.
    \item[$(ii)$.] \emph{(Monotonicity and interpolation)} For $q_0 < q_1$, one has $||M||_{S^{q_0}} \ge ||M||_{S^{q_1}}$. Moreover, for $\theta \in [0, 1]$ and $\frac{1}{q_\theta} := \frac{1-\theta}{q_0} + \frac{\theta}{q_1}$, one has 
    \[
        ||M||_{S^{q_\theta}} \le ||M||_{S^{q_0}}^{1-\theta} \cdot  ||M||_{S^{q_1}}^{\theta}.
    \]
    \item[$(iii)$.] \emph{(Special cases)} $||M||_{S^\infty}$ is the spectral ($L^2$ operator) norm of $M$, often written $||M||$ or $||M||_2$, while $||M||_{S^2}$ is the $L^2$ norm on the entries of $M$ (also known as the Frobenius norm). More generally, one has $||M||_{S^{2k}}^{2k} = \Tr((MM^*)^k)$ for $k \in \Z_{\ge 1}$. In particular, if $M$ is the adjacency matrix of an undirected graph $G$ and $k \ge 2$, then $||M||_{S^{2k}}^{2k} = \Hom(C_{2k}, G)$.
    \item[$(iv)$.] \emph{(Doubling property)} One has
    \begin{equation} \label{eq:halfing}
        ||M||_{S^q} = ||MM^*||_{S^{q/2}}^{1/2}.
    \end{equation}
    \item[$(v)$.] \emph{(H\"older-type inequality)} If $\frac{1}{r} = \frac{1}{q_1} + \cdots + \frac{1}{q_k}$, one has $||A_1 \cdots A_k||_{S^r} \le ||A_1||_{S^{q_1}} \cdots ||A_k||_{S^{q_k}}$. In particular, given any diagonal matrix $I_A$ with entries in $\{0, 1\}$, and any partitions $\mP, \mR$ of a suitably-sized set such that $\mR$ is finer than or equal to $\mP$, one has
    \begin{equation} \label{eq:ignore-averaging}
        \max\left(||I_A M||_{S^r},\ ||A_{\mP}M||_{S^r},\ ||(A_{\mR} - A_\mP)M||_{S^r} \right) \le ||M||_{S^r}.
    \end{equation}
    \item[$(vi)$.] \emph{(Trace norm inequality)} When $q = 1$, $||M||_{S^1}$ (also known as the trace norm or nuclear norm) satisfies
    \[
        |\Tr(M)| \le ||M||_{S^1}.
    \]
\end{itemize}

\end{lemma}

\begin{proof}
For the fact that $||\cdot ||_{S^q}$ defines a norm, see \cite{bhatia2000cartesian,zhan2002matrix}. Otherwise, parts $(i)$, $(iii)$ and $(iv)$ follow from basic properties of singular values, and part $(ii)$ follows from the monotonicity and interpolation properties of $L^q$ norms (for the counting measure). 

The H\"older-type inequality in $(v)$ reduces to the case $k = 2$ (i.e., $||AB||_{S^r} \le ||A||_{S^p} ||B||_{S^q}$ when $\frac{1}{r} = \frac{1}{p} + \frac{1}{q}$); we refer the reader to \cite[Corollary 4.27]{zhan2002matrix} and \cite[\S III]{ball2002sharp} for this case. Then, \cref{eq:ignore-averaging} follows by taking $p = \infty$ and using that $||I_A||_{S^\infty}$, $||A_\mP||_{S^\infty}$, $||A_\mR - A_\mP||_{S^\infty}$ are either $0$ or $1$ (by \cref{lem:averaging-matrices}). 

Finally, let us give a short proof of the trace norm inequality in $(vi)$. Let $M = U\Sigma V^*$ be a singular value decomposition of $M$, where $\Sigma$ has diagonal entries $\sigma_i$. Then
\[
    \Tr(M) = \Tr(U\Sigma V^*) = \Tr(\Sigma V^* U) 
    =
    \sum_i \sigma_i w_i,
\]
where $w_i$ is the $(i,i)$th entry of $V^*U$, i.e., the inner product between $U$'s $i$th column and $V$'s $i$th row. But $|w_i| \le 1$ by Cauchy--Schwarz (since $U, V$ are unitary), and thus $|\Tr(M)| \le \sum_i \sigma_i = ||M||_{S^1}$.
\end{proof}

\begin{definition}[Schatten norms of functions]
Given $q \in [1, \infty]$, finite probability spaces $(X, \P_X)$, $(Y, \P_Y)$, and a function $f : X \times Y \to \C$, we define the $q$-\emph{Schatten norm} of $f$ by
\[
    ||f||_{S^q} := ||\Mat(f)||_{S^q},
\]
Note that this does not depend on the chosen ordering of the elements of $X$ and $Y$ (implicit in the notation $\Mat(f)$), since permutation matrices are unitary.
\end{definition}

\begin{lemma} \label{lem:Schatten-norms-functions}
Given finite probability spaces $(X, \P_X)$, $(Y, \P_Y)$, a function $f : X \times Y \to \C$, and $q, q_i \in [1, \infty]$, the following hold true.
\begin{itemize}
    \item[$(i)$.] $|| \cdot ||_{S^q}$ defines a norm on functions from $X \times Y$ to $\C$, which satisfies the same monotonicity and interpolation properties as in \cref{lem:Schatten-norms-matrices}. Moreover, $||f^T||_{S^q} = ||f||_{S^q}$.
    \item[$(ii)$.] When $q \in 2\Z \cup \{\infty\}$, one can rewrite $||f||_{S^q}$ as in \cref{eq:Schatten-2k} and \cref{eq:Schatten-infty}. In particular, one has $||f||_{S^2} = ||f||_{L^2}$, $||f||_{S^4} = ||f||\gow$, and $||f||_{S^\infty} \ge ||f||\sq$.
    \item[$(iii)$.] For any $A \subset X$ and $B \subset Y$, and any partitions $\mP$ of $X$ and $\mQ$ of $Y$, one has
    \[
        \max \left(||f \one_{A \times B}||_{S^q} ,\ ||\E(f \mid \mP \otimes \mQ)||_{S^q}\right) \le ||f||_{S^q}.
    \]
    Moreover, $\Mat(\E(f \mid \mP \otimes \mQ))$ has at most $\max(|\mP|, |\mQ|)$ nonzero singular values; in particular, the constant function $\one_{X \times Y}$ has $||\one_{X \times Y}||_{S^q} = 1$.
    \item[$(iv)$.] \emph{(Cycle counting inequality)}
    Let $k \in \Z_{\ge 2}$ and suppose that $\sum_{i=1}^k \frac{1}{q_i} \ge 1$.
    For $i \in \Z/k\Z$, let $(X_i, \P_i)$ be finite probability spaces and $f_i : X_i \times X_{i+1} \to \C$. Then one has
    \[
        \left\vert \E_{x_1, \ldots, x_k} \prod_{i=1}^k f_i(x_i, x_{i+1}) \right\vert 
        \le 
        \prod_{i=1}^k ||f_i||_{S^{q_i}}.
    \]
\end{itemize}
\end{lemma}

\begin{proof}
Parts $(i)$ and $(ii)$ follow from the first three parts of \cref{lem:Schatten-norms-matrices}, \cref{eq:k-cycle-count}, and the fact that the operator norm of an $m \times n$ complex matrix $M$ can be rewritten as
\[
    ||M||_{S^\infty} = \sup_{\substack{v \in \C^m, w \in \C^n \\ ||v||_2 = ||w||_2 = 1}} |v^T M w|,
\]
where we can make the change of variables $v_j := g(x_j) \sqrt{\P(x_j)}$ and $w_j := h(y_j) \sqrt{\P(y_j)}$ for some functions $g : X \to \C, h : Y \to \C$ with $L^2$ norms equal to $1$. The first part of $(iii)$ follows from \cref{lem:Schatten-norms-matrices}.(v), and the second part follows from the facts that $\Mat(\E(f \mid \mP \otimes \mQ)) = A_\mP \Mat(f) A_\mQ$, $\rank(A_\mP) = |\mP|$, and $\rank(A_\mQ) = |\mQ|$.

Finally, let us prove the cycle counting inequality in $(iv)$. Fixing an ordering of the elements of each $X_i$ and letting $F_{i} := \Mat(f_{i})$, we need to show (by \cref{eq:k-cycle-count}) that 
\[
    \left\vert \Tr\left( F_1 \cdots F_k \right)\right\vert \le \prod_{i=1}^k ||F_i||_{S^{q_i}}.
\]
For $i \in \Z/k\Z$, define $q_i' \in [q_i, \infty]$ by $q_i' := q_i \sum_{j=1}^k \frac{1}{q_j}$, so that $\sum_{i=1}^k \frac{1}{q_i'} = 1$. Then by the trace norm and H\"older-type inequalities in \cref{lem:Schatten-norms-matrices}, and the monotonicity of Schatten norms, we have
\[
    \left\vert \Tr\left( F_1 \cdots F_k \right)\right\vert \le \left\vert \left\vert F_1 \cdots F_k \right\vert\right\vert_{S^1} 
    \le 
    \prod_{i=1}^k ||F_i||_{S^{q_i'}} \le 
    \prod_{i=1}^k ||F_i||_{S^{q_i}},
\]
as we wanted.
\end{proof}

\begin{remark}
The cycle counting inequality in \cref{lem:Schatten-norms-functions} can be regarded as a generalization of the Gowers--Cauchy--Schwarz inequality of order $2$ (the latter is the special case $k = 4$, $q_1 = \cdots = q_4 = 1/4$). When $q_1 = \cdots = q_k \in 2\Z$, it also follows from Hatami's H\"older-type inequality \cite{hatami2009generalizations}.

In fact, when $q, q_1, \ldots, q_k$ are even integers, most (if not all) of the properties in \cref{lem:Schatten-norms-matrices,lem:Schatten-norms-functions} can be proven combinatorially, using repeated (and rather convoluted) applications of Cauchy--Schwarz, similar to those in \cite{conlon2017finite}. Following this idea, it is possible that some of the generalized Gowers norms considered by Hatami and then Conlon and Lee \cite{hatami2009generalizations,conlon2017finite} could function as a higher-order substitute for the Schatten norms, in an extension of our work to hypergraphs.


\end{remark}

\begin{open}
Obtain compatible versions of \cref{thm:regularity} and \cref{lem:counting} for hypergraphs, using combinatorial analogues of the Schatten norms.
\end{open}

We also mention a relationship between Schatten norms and $L^q$ norms of Fourier coefficients, connecting the graph-theoretic and Fourier-analytic methods in additive combinatorics.

\begin{lemma}[Fourier coefficients and singular values] \label{lem:Fourier-singular}
Let $G$ be a finite abelian group (equipped with the uniform probability measure) and $a, b \in \Z$ be such that $\gcd(|G|, ab) = 1$. Let $f : G \to \C$, and define $\tilde{f} : G \times G \to \C$ by $\tilde{f}(x, y) := f(ax+by)$. Then one has
\[
    ||\tilde{f}||_{S^q} = \frac{1}{|G|} ||\hat{f}||_{q}.
\]
\end{lemma}

\begin{proof}
Since the maps $x \mapsto ax$ and $y \mapsto by$ are automorphisms of $G$, and permutations of the rows and columns of a matrix preserve its singular values, we have
\[
    ||\tilde{f}||_{S^q}
    =
    \frac{1}{|G|}||M||_{S^q},
\]
where $M$ is the matrix with entries $M_{x,y} = f(x+y)$ (for some ordering of the elements of $G$). It thus suffices to show that the singular values of $M$ are precisely the absolute values of the Fourier coefficients of $f$. Indeed, the inverse Fourier transform identity implies that for any $\chi \in \hat{G}$, one has
\[
    f(x+y) = \frac{1}{|G|} \sum_{\chi} \chi(x) \hat{f}(\chi) \chi(y),
\]
which translates to the identity of $|G| \times |G|$ matrices (noting that $|G| = |\hat{G}|$)
\[
    M = U D U^T,
    \qquad\quad \text{where:}\qquad\quad  
    U_{x,\chi} = \chi(x),
    \qquad 
    D_{\chi_1, \chi_2} = \hat{f}(\chi_1) \one_{\chi_1 = \chi_2}.
\]
Moreover, the matrices $U$ and $U^T$ are unitary by standard properties of the Fourier transform. Hence $M M^* = U D D^* U^*$, which is a diagonalization of $MM^*$, and thus the singular values of $M$ are the square roots of the diagonal entries of $D D^*$, i.e., $|\hat{f}(\chi)|$ for $\chi \in \hat{G}$.
\end{proof}

We end this subsection by defining a family of energy functions related to the Schatten norms, which will be used in the proof of our regularity lemma (and may be of independent interest).

\begin{definition}[Generalized $L^2$ energy] \label{def:generalized-L2-energy}
Fix a finite probability space $(X, \P_X)$ and an ordering of its elements. For any complex matrix $F$ with $|X|$ rows, any integer $\ell \ge 1$, and any partition $\mP$ of $X$, we define the (left) $\ell$th $L^2$ energy of $\mP$ with respect to $F$ by
\[
    \mE_{F,\ell}(\mP) := ||A_\mP F_\ell||_{S^2}^2,
\]
where $F_\ell$ can be defined inductively by $F_1 := F$ and $F_{\ell+1} := F_\ell F_\ell^*$. In other words,
\[
    F_\ell := 
    \begin{cases} 
        F, & \ell = 1, \\
        (FF^*)^{2^{\ell-2}}, & \ell \ge 2.
    \end{cases}
\]
We also define the \emph{total} $\ell$th $L^2$ energy of $F$ by $\mE_{M,\ell} := ||F_\ell||_{S^2}^2$. Given another probability space $(Y, \P_Y)$ and a function $f : X \times Y \to \C$, we define $\mE_{f,\ell}(\mP) := \mE_{\Mat(f),\ell}(\mP)$ and $\mE_{f,\ell} := \mE_{\Mat(f),\ell}$.
\end{definition}

\begin{remark}
With the notation above, for $\ell \ge 1$ one has
\[
    \mE_{F,\ell} = \Tr \left(F_\ell F_\ell^* \right) =
    \Tr \left(F_{\ell+1} \right)
    = \Tr \left((FF^*)^{2^{\ell-1}} \right)
    =
    ||F||_{S^{2^{\ell}}}^{2^{\ell}},
\]
and thus $\mE_{f,\ell} = ||f||_{S^{2^\ell}}^{2^\ell}$ for any $f : X \times Y \to \C$. Also, $\mE_{f,\ell}(\mP)$ can only increase when we pass to a finer partition, due to \cref{eq:energy-increment-matrices}.
\end{remark}

\begin{lemma}[Energy increment] \label{lem:energy-increment}
Let $(X, \P_X)$, $(Y, \P_Y)$ be finite probability spaces, $f : X \times Y \to \C$, and $\mP = \mP_X \otimes \mP_Y$, $\mR = \mR_X \otimes \mR_Y$ be partitions of $X \times Y$ such that $\mR$ is finer than or equal to $\mP$. Then for all $\ell \ge 1$, one has
\[
\begin{aligned}
    ||\E(f \mid \mR) - \E(f \mid \mP)||_{S^{2^\ell}} 
    &\le 
    (\mE_{f,\ell}(\mR_X) - \mE_{f,\ell}(\mP_X))^{1/2^\ell}
    \\
    &+
    (\mE_{f^T,\ell}(\mR_Y) - \mE_{f^T,\ell}(\mP_Y))^{1/2^\ell},
\end{aligned}
\]
where $f^T : Y \times X \to \C$ is defined by $f^T(y, x) := f(x, y)$.
\end{lemma}

\begin{proof}
Fix orderings of the elements of $X$ and $Y$, and let $F := \Mat(f)$. By \cref{lem:averaging-matrices} and the triangle inequality, the quantity to bound becomes
\[
\begin{aligned}
    ||A_{\mR_X} F A_{\mR_Y} - A_{\mP_X} F A_{\mP_Y}||_{S^{2^\ell}}
    &\le 
    ||(A_{\mR_X} - A_{\mP_X})F A_{\mR_Y}||_{S^{2^\ell}}
    +
    ||A_{\mP_X} F (A_{\mR_Y} - A_{\mP_Y})||_{S^{2^\ell}}
    \\
    &\le 
    ||(A_{\mR_X} - A_{\mP_X})F||_{S^{2^\ell}}
    +
    ||F (A_{\mR_Y} - A_{\mP_Y})||_{S^{2^\ell}}
    \\
    &=
    ||(A_{\mR_X} - A_{\mP_X})F||_{S^{2^\ell}}
    +
    ||(A_{\mR_Y} - A_{\mP_Y})F^T||_{S^{2^\ell}}.
\end{aligned}
\]
By symmetry, it now suffices to show that
\[
    ||(A_{\mR_X} - A_{\mP_X})F||_{S^{2^\ell}} \le \left(\mE_{F,\ell}(\mR_X) - \mE_{F,\ell}(\mP_X) \right)^{1/2^\ell}.
\]
But by \cref{eq:halfing} and \cref{eq:ignore-averaging}, for $1 \le j \le \ell-1$ we have
\[
\begin{aligned}
    ||(A_{\mR_X} - A_{\mP_X}) F_j||_{S^{2^{\ell+1-j}}} &= ||(A_{\mR_X} - A_{\mP_Y}) F_j F_j^* (A_{\mR_X} - A_{\mP_X})||^{1/2}_{S^{2^{\ell-j}}}
    \\
    &\le ||(A_{\mR_X} - A_{\mP_X}) F_{j+1}||^{1/2}_{S^{2^{\ell-j}}}.
\end{aligned}
\]
By inducting on $j$, this implies that
\[
    ||(A_{\mR_X} - A_{\mP_X}) F||_{S^{2^\ell}} \le ||(A_{\mR_X} - A_{\mP_X}) F_{\ell}||^{1/2^{\ell-1}}_{S^{2}},
\]
and we can further write, by \cref{eq:energy-increment-matrices},
\[
\begin{aligned}
    ||(A_{\mR_X} - A_{\mP_X}) F_{\ell}||_{S^2}^2
    =
    ||A_{\mR_X} F_{\ell}||_{S^2}^2 - 
    ||A_{\mP_X} F_{\ell}||_{S^2}^2
    = \mE_{F,\ell}(\mR) - \mE_{F,\ell}(\mP).
\end{aligned}
\]
Putting the last two equations together, we obtain our conclusion.
\end{proof}

\section{The regularity lemma} \label{sec:regularity}


We start by presenting an abstract energy increment argument, the kind of which is typically employed in the proofs of regularity lemmas. Following Tao \cite{tao2005szemer}, we will in fact phrase this as an energy optimization argument, which depends on both an \emph{energy function} $\Energy$ and a \emph{complexity function} $\Complexity$. This discussion pertains to weak regularity lemmas; a strong regularity lemma would require an additional level of iteration.

\begin{lemma}[Abstract energy optimization argument] \label{lem:energy-optimization}
Let $X$ be a finite set and $\mathscr{P}$ be a set of partitions of $X$, which contains the discrete partition $\mD := \{\{x\} : x \in X\}$, the trivial partition $\mT := \{X\}$, and which is closed under taking common refinements (i.e., if $\mP, \mQ \in \mathscr{P}$, then $\mP \vee \mQ \in \mathscr{P}$). Consider two functions $\Energy, \Complexity : \mathscr{P} \to \R$ with the following properties:
\begin{itemize}
    \item[$(i)$] \emph{(Energy is nondecreasing)}. For any $\mP, \mP' \in \mathscr{P}$ such that $\mP'$ is finer than $\mP$, one has $\Energy(\mP) \le \Energy(\mP')$.
    \item[$(ii)$] \emph{(Complexity is subadditive)}. For any $\mP, \mQ \in \mathscr{P}$, one has $\Complexity(\mP \vee \mQ) \le \Complexity(\mP) + \Complexity(\mQ)$. Moreover, $\Complexity(\mT) = 0$.
\end{itemize}
Let $\eps > 0$. Then there exists a partition $\mP \in \mathscr{P}$ with $\Complexity(\mP) \le \frac{1}{\eps}$, such that for any $\mQ \in \mathscr{P}$, one has
\[
    \Energy(\mP \vee \mQ) - \Energy(\mP) \le \eps\ \Complexity(\mQ) \left(\Energy(\mD) - \Energy(\mT)\right).
\]
\end{lemma}
\begin{proof}
If $\Energy(\mD) = \Energy(\mT)$ (i.e., the energy function is constant), then we can easily pick $\mP = \mT$, for example. Otherwise, let $\mP \in \mathscr{P}$ be a partition which maximizes the real quantity
\[
    \Energy(\mP) - \eps\ \Complexity(\mP) \left(\Energy(\mD) - \Energy(\mT)\right).
\]
Then by comparing the optimizer $\mP$ to the trivial partition $\mT$, we find that
\[
    \Energy(\mT) \le \Energy(\mP) - \eps\ \Complexity(\mP) \left(\Energy(\mD) - \Energy(\mT)\right)
    \le 
    \Energy(\mD) - \eps\ \Complexity(\mP) \left(\Energy(\mD) - \Energy(\mT)\right),
\]
and thus $\Complexity(\mP) \le 1/\eps$. Moreover, given any $\mQ \in \mathscr{P}$, by comparing $\mP$ to $\mP \vee \mQ$ we find that
\[
    \Energy(\mP) - \eps\ \Complexity(\mP) \left(\Energy(\mD) - \Energy(\mT)\right)
    \ge 
    \Energy(\mP \vee \mQ) - \eps\ \Complexity(\mP \vee \mQ) \left(\Energy(\mD) - \Energy(\mT)\right),
\]
which, combined with the subadditivity of $\Complexity$, yields $\Energy(\mP \vee \mQ) - \Energy(\mP) \le \eps\ \Complexity(\mQ) \left(\Energy(\mD) - \Energy(\mT)\right)$. 
\end{proof}

\begin{remark}
Intuitively, \cref{lem:energy-optimization} finds a partition $\mP$ with a good balance of energy and complexity: it captures most of the energy needed for applications (in the sense that small refinements of $\mP$ only see a small increase in energy), while having bounded complexity. Here are some concrete objects that can fulfill the abstract roles of $\Complexity$, $\Energy$ and $\mathscr{P}$ from \cref{lem:energy-optimization}:
\begin{itemize}
    \item One can take $\Complexity(\mP)$ to be $\log |\mP|$, or the \emph{entropy} $\H(\mP)$ (by \cref{eq:entropy-inequalities}); Tao uses both in \cite{tao2005szemer}.
    \item If $\Omega$ is a probability space and $f : \Omega \to \C$ is an $L^2$ random variable, one can take $\Energy(\mP)$ to be the $L^2$ energy $\E[\E(f \mid \mP_1 \otimes \mP_2)^2]$. In this case, one can regard the optimizer $\E(f \mid \mP_1 \otimes \mP_2)$ as the best approximation to $f$ (of a certain type) with not too many different values.
    \item Alternatively, given a partition $\mP_0$ of $\Omega$ (e.g., induced by a discrete random variable $f$), one can take $\Energy(\mP)$ to be the \emph{mutual information} $\I(\mP_0 : \mP) = \H(\mP_0) + \H(\mP) - \H(\mP_0 \vee \mP)$; this is essentially what Tao does in \cite{tao2005szemer}. In this case, one can regard the optimizer $\mP$ as the best approximation to $\mP_0$ (of a certain type) with bounded complexity.
    \item For a logarithmic and \emph{relative} version of the energy $\mE(\mP)$, see \cref{sec:sharp-regularity}.
    \item If $\Omega = X \times Y$, one can let $\mathscr{P}$ consist of all partitions of the form $\mP \otimes \mQ$, where $\mP$ is a partition of $X$ and $\mQ$ is a partition of $Y$. This is often the case in the proof of regularity lemmas (see \cref{sec:sharp-regularity}), but we will actually end up using $\Omega = X$ and $\Omega = Y$ separately. 
\end{itemize}
\end{remark}

In our argument, we will use the full space $\mathscr{P}$ of partitions of a probability space $(X, \P_X)$, both the complexity functions $\log |\mP|$ and $\H(\mP)$, and the generalized $L^2$ energy functions from \cref{def:generalized-L2-energy} (which are monotonic by \cref{eq:energy-increment-matrices}). To be more precise, we will use a sum of finitely many such energy functions, which resembles a \emph{simultaneous} energy increment argument \cite{tao2007ergodic}.

Whether it is helpful to use $\log |\mP|$ or $\H(\mP)$ as the complexity function depends on the uniformity norm one wishes to bound. Indeed, while the cut norm of a function $f : X \times Y \to \C$ depends on the interaction between $f$ and subsets of $X$ and $Y$ (which correspond to partitions with at most two parts), the \emph{spectral norm} $S^\infty$ can be characterized in terms of the interaction between $f$ and partitions with bounded entropy; this is the content of the following lemma.

\begin{lemma}[Spectral norms and entropy] \label{lem:modify-spectral}
Given finite probability spaces $(X, \P_X)$, $(Y, \P_Y)$, and $f : X \times Y \to \C$, one has
\[
    ||f||_{S^\infty} \le \sup_{g,h}^* \left\vert \E_{x,y}[f(x,y)g(x)h(y)]\right\vert,
\]
where the supremum is over all functions $g : X \to \C$, $h : Y \to \C$ with $L^2$ norms at most equal to $3$, and values from $\{0\} \cup \{2^k, i2^k, -2^k, -i2^k : k \in \Z_{\ge 0}\}$. In particular, viewing such $g$ and $h$ as random variables, one has
\[
    \H(g) + \H(h) \ll 1.
\]
\end{lemma}

\begin{proof}
We leave this computation to \cref{sec:spectral-norms-entropy}.
\end{proof}


We are now ready to state and prove our sparse weak regularity lemma, which generalizes \cref{thm:regularity-intro}. For technical reasons, we state our result for a finite family of $k \ge 1$ functions rather than a single function $f : X \times Y \to \C$.

\begin{theorem}[Simultaneous weak regularity lemma with Schatten norms] \label{thm:regularity}
Let $\{(X_i, \P_i)\}_{1 \le i \le m}$ be finite probability spaces, $1 \le q_t < r_t \le \infty$, $i_t, j_t \in \{1, \ldots, m\}$ and $f_t : X_{i_t} \times X_{j_t} \to \C$ for $t \in \{1, \ldots, k\}$, and $\eps \in (0, 1)$.
Then there exist partitions $\mP_i, \mP_i'$ of $X_i$ such that
\begin{equation}\label{eq:cut-norm-bound}
    \max_i \log |\mP_i| \ll_{k,q} \frac{1}{\eps^{\alpha(q,\infty)}}, \quad\quad
    \max_t \left\vert \left\vert f_t - \E(f_t \mid \mP_{i_t} \otimes \mP_{j_t})
    \right\vert\right\vert\sq \le \eps \max_t ||f_t||_{S^{q_t}},
\end{equation}
respectively
\begin{equation}\label{eq:spec-norm-bound}
    \max_i \H(\mP'_i) \ll_{k,q,r} \frac{1}{\eps^{\alpha(q,r)}},
    \quad\quad
    \max_t \left\vert \left\vert f_t - \E(f_t \mid \mP'_{i_t} \otimes \mP'_{j_t})
    \right\vert\right\vert_{S^{r_t}} \le \eps \max_t ||f_t||_{S^{q_t}},
\end{equation}
where $q := \max_t q_t$, $r := q \min_{t} (r_t/q_t)$, and $\alpha(q,r) := 2^{\max(\lceil \log_2 q \rceil, 1)}/(1 - \frac{q}{r})$.
\end{theorem}

\begin{remark}
Ultimately, we will only use \cref{thm:regularity} for functions $f$ with values in $[0, \infty)$, but it is not much harder to treat all complex-valued functions. Also, our sparse regularity method will only use \cref{eq:spec-norm-bound}, but the cut norm bound in \cref{eq:cut-norm-bound} may be of independent interest to the reader (e.g., since it fits better with the counting lemmas of \cite{conlon2021regularity,conlon2021graphs}). 

While Schatten norms are stronger than cut norms, the drawback of using \cref{eq:spec-norm-bound} is that it only controls the entropies $\H(\mP'_i)$ (which could be much smaller than $\log |\mP'_i|$). However, this disadvantage can be removed if one is allowed to disregard a small number of elements of each $X_i$, which is the case for most applications.
\end{remark}

\begin{proof}[Proof of \cref{thm:regularity}]
We first note that it suffices to consider the case when all $r_t = \infty$, since Schatten norms can be interpolated. Indeed, assuming that \cref{eq:spec-norm-bound} holds when all $r_t = \infty$, one can apply it for $2^{1/(1-\frac{r}{q})}\eps^{1/(1-\frac{q}{r})}$ in place of $\eps$ (for any given $r \in (q, \infty)$); this yields partitions $\mP'_i$ with
\[
    \max_i \H(\mP'_i) \ll_q \frac{1}{2^{\alpha(q,\infty)/(1-\frac{r}{q})}\eps^{\alpha(q,\infty)/(1 - \frac{q}{r})}}
    \ll_{q,r}
    \frac{1}{\eps^{\alpha(q,r)}},
\]
such that
\[
    \max_t ||f_t - \E(f_t \mid \mP'_{i_t} \otimes \mP'_{j_t})||_{S^\infty} \le 2^{1/(1-\frac{r}{q})} \eps^{1/(1-\frac{q}{r})} \max_t ||f_t||_{S^{q_t}}.
\]
But then, given any values $r_t \in (q_t, \infty]$ and setting $r := q \min_t(r_t/q_t)$, \cref{lem:Schatten-norms-functions} implies that if $M := \max_t ||f_t||_{S^{q_t}}$, then for each $t \in \{1, \ldots, k\}$ one has
\[
\begin{aligned}
    &||f_t - \E(f_t \mid \mP'_{i_t} \otimes \mP'_{j_t})||_{S^{r_t}} 
    \\
    &\le
    ||f_t - \E(f_t \mid \mP'_{i_t} \otimes \mP'_{j_t})||_{S^{q_t}}^{q_t/r_t} \cdot ||f_t - \E(f_t \mid \mP'_{i_t} \otimes \mP'_{j_t})||_{S^\infty}^{1 - q_t/r_t}
    \\
    &\le 
    \left(||f_t||_{S^{q_t}} + ||\E(f_t \mid \mP'_{i_t} \otimes \mP'_{j_t})||_{S^{q_t}}\right)^{q_t/r_t} \cdot \left(2^{1/(1-\frac{r}{q})}\eps^{1/(1-\frac{q}{r})} M\right)^{1 - q_t/r_t}
    \\
    &\le 
    \left(2M\right)^{q_t/r_t} \cdot 2^{-q_t/r_t}\eps M^{1 - q_t/r_t}
    \\
    &=
    \eps M.
\end{aligned}
\]

After this reduction, it remains to find partitions $\mP_i$ which satisfy \cref{eq:cut-norm-bound}, and $\mP_i'$ which satisfy \cref{eq:spec-norm-bound} for $r = r_t = \infty$. For now, fix some partitions $\mP_i$ and $\mP_i'$ of $X_i$, which will be chosen later via the abstract energy optimization argument in \cref{lem:energy-optimization}; it will take some work to explain how exactly this optimization argument will be applied. By \cref{eq:cut} and \cref{lem:modify-spectral}, we can express the uniformity norms of interest (for $t \in \{1, \ldots, k\}$) as
\begin{gather*}
    ||f_t - \E(f_t \mid \mP_{i_t} \otimes \mP_{j_t})||\sq = \sup_{\substack{g_1 : X_{i_t} \to \{0,1\} \\ g_2 : X_{j_t} \to \{0,1\}}} |\E_{x_1,x_2}[(f_t - \E(f_t \mid \mP_{i_t} \otimes \mP_{j_t}))(x_1,x_2)g_1(x_1)g_2(x_2)]|,
    \\
    ||f_t - \E(f_t \mid \mP'_{i_t} \otimes \mP'_{j_t})||_{S^\infty} \le \sup_{\substack{h_1 : X_{i_t} \to \C \\ h_2 : X_{j_t} \to \C}}^* \left\vert \E_{x_1,x_2}[(f_t - \E(f_t \mid \mP'_{i_t} \otimes \mP'_{j_t}))(x_1,x_2)h_1(x_1)h_2(x_2)]\right\vert,
\end{gather*}
where $x_1$ is sampled from $(X_{i_t}, \P_{i_t})$, $x_2$ is sampled from $(X_{j_t}, \P_{j_t})$, and $\sup_{h_1,h_2}^*$ has the same meaning as in \cref{lem:modify-spectral}.
In particular, all functions involved in both suprema have $||g_u||_{L^2} \ll 1$ and $||h_u||_{L^2} \ll 1$, while each $g_u$ in the first supremum takes at most two values, and each $h_u$ in the second supremum has $\H(h_u) \ll 1$ (for $u \in \{1, 2\}$). Taking a maximum over $t$ and suppressing dependencies on $x_1, x_2$ for notational simplicity, we get
\begin{gather*}
    \max_{t \in \{1, \ldots, k\}} ||f_t - \E(f_t \mid \mP_{i_t} \otimes \mP_{j_t})||\sq = \sup_{\substack{g^t_1 : X_{i_t} \to \{0,1\} \\ g^t_2 : X_{j_t} \to \{0,1\} \\ t \in \{1, \ldots, k\}}} |\E[(f_t - \E(f_t \mid \mP_{i_t} \otimes \mP_{j_t}))g^t_1 g^t_2]|,
    \\
    \max_{t \in \{1, \ldots, k\}} ||f_t - \E(f_t \mid \mP'_{i_t} \otimes \mP'_{j_t})||_{S^\infty} \le \sup_{\substack{h^t_1 : X_{i_t} \to \C \\ h^t_2 : X_{j_t} \to \C \\ t \in \{1, \ldots, k\}}}^* \left\vert \E[(f_t - \E(f_t \mid \mP'_{i_t} \otimes \mP'_{j_t}))h^t_1 h^t_2]\right\vert,
\end{gather*}
Thus it suffices to bound the right-hand sides for any specific choice of the tuples of functions $\{g^t_u\}_{t,u}$, respectively $\{h^t_u\}_{t,u}$. Fix such tuples of functions, and consider the partitions of $X_i$ that they collectively induce:
\[
    \mQ_i := \bigvee_{\substack{t \in \{1, \ldots, k\} \\ i_t = i}} \Part(g^t_1) \vee 
    \bigvee_{\substack{t \in \{1, \ldots, k\} \\ j_t = i}} \Part(g^t_2),
    \qquad\quad
    \mQ_i' := \bigvee_{\substack{t \in \{1, \ldots, k\} \\ i_t = i}} \Part(h^t_1) \vee 
    \bigvee_{\substack{t \in \{1, \ldots, k\} \\ j_t = i}} \Part(h^t_2).
\] 
Due to the nature of the functions $g^t_u$ and $h^t_u$, these partitions satisfy the the complexity bounds
\begin{equation}\label{eq:complexity-bounds}
    \log |\mQ_i| \ll_k 1 \qquad\quad \text{and} \qquad\quad 
    \H(\mQ_i') \ll_k 1.
\end{equation}
For notational convenience, we also denote
\begin{equation}\label{eq:joint-partitions}
    \mR_i := \mP_i \vee \mQ_i \qquad\quad \text{and} \qquad\quad 
    \mR_i' := \mP_i' \vee \mQ_i',
\end{equation}
which begins to resemble the setting of our abstract energy optimization argument (\cref{lem:energy-optimization}). Now for $t \in \{1, \ldots, k\}$, using \cref{eq:Schatten-infty} and the fact that the function $(x_1, x_2) \mapsto g_1^{i,j}(x_1) g_2^{i,j}(x_2)$ is $(\mR_{i_t} \otimes \mR_{j_t})$-measurable, we can write
\[
\begin{aligned}
    \left\vert \E[(f_t - \E(f_t \mid \mP_{i_t} \otimes \mP_{j_t}))g^t_1 g^t_2]\right\vert
    &=
    \left\vert \E[(\E(f_t \mid \mR_{i_t} \otimes \mR_{j_t}) - \E(f_t \mid \mP_{i_t} \otimes \mP_{j_t})) g^t_1 g^t_2]\right\vert
    \\
    &\le
    ||\E(f_t \mid \mR_{i_t} \otimes \mR_{j_t}) - \E(f_t \mid \mP_{i_t} \otimes \mP_{j_t})||_{S^\infty} \cdot ||g_1^t||_{L^2} \cdot ||g_2^t||_{L^2} 
    \\
    &\ll
    ||\E(f_t \mid \mR_{i_t} \otimes \mR_{j_t}) - \E(f_t \mid \mP_{i_t} \otimes \mP_{j_t})||_{S^\infty},
\end{aligned}
\]
and similarly,
\[
    \left\vert \E[(f_t - \E(f_t \mid \mP'_{i_t} \otimes \mP'_{j_t}))h^t_1 h^t_2]\right\vert \ll
    ||\E(f_t \mid \mR'_{i_t} \otimes \mR'_{j_t}) - \E(f_t \mid \mP'_{i_t} \otimes \mP'_{j_t})||_{S^\infty}.
\]
Putting things together, to obtain the desired uniformity norm bounds in \cref{eq:cut-norm-bound,eq:spec-norm-bound}, it suffices to have
\begin{gather}
    \max_t ||\E(f_t \mid \mR_{i_t} \otimes \mR_{j_t}) - \E(f_t \mid \mP_{i_t} \otimes \mP_{j_t})||_{S^\infty} \le \eps \max_t ||f_t||_{S^q}, \label{eq:first-supremum}
    \\
    \max_t ||\E(f_t \mid \mR'_{i_t} \otimes \mR'_{j_t}) - \E(f_t \mid \mP'_{i_t} \otimes \mP'_{j_t})||_{S^\infty} \le \eps \max_t ||f_t||_{S^q}.
    \label{eq:second-supremum}
\end{gather}

We will first present the argument for \cref{eq:first-supremum}, and then mention what changes for \cref{eq:second-supremum}. By \cref{lem:energy-increment} for $\ell := \max(\lceil \log_2 q \rceil, 1) \ge \log_2 q$, for each $t$ we have
\begin{equation} \label{eq:final-energy-bound}
\begin{aligned}
    ||\E(f_t \mid \mR_{i_t} \otimes \mR_{j_t}) - \E(f_t \mid \mP_{i_t} \otimes \mP_{j_t}))||_{S^\infty}
    &\le 
    (\mE_{f_t,\ell}(\mR_{i_t}) - \mE_{f_t,\ell}(\mP_{i_t}))^{1/2^\ell} 
    \\
    &+ 
    (\mE_{f_t^T,\ell}(\mR_{j_t}) - \mE_{f_t^T,\ell}(\mP_{j_t}))^{1/2^\ell},
\end{aligned}
\end{equation}
where all energies are bounded by $\mE_{f_t,\ell} = ||f_t||_{S^{2^\ell}}^{2^\ell} \le ||f_t||_{S^q}^{2^\ell} \le ||f_t||_{S^{q_t}}^{2^\ell}$. We are finally ready to choose $\mP_i$ by applying the energy optimization argument in \cref{lem:energy-optimization}, once for each $i \in \{1, \ldots, m\}$, to the set $\Omega = X_i$, in the search space $\mathscr{P}$ of all partitions of $X_i$, with the complexity function $\log |\cdot|$, $c\eps^{2^\ell}$ in place of $\eps$ (for some $c > 0$ to be chosen shortly), and with the energy function
\[
    \mE_i(\mP_i) := \sum_{\substack{t \in \{1, \ldots, k\} \\ i_t = i}} \mE_{f_t,\ell}(\mP_i) + \sum_{\substack{t \in \{1, \ldots, k\} \\ j_t = i}} \mE_{f_t^T,\ell}(\mP_i).
\]
Recalling the notation $M := \max_t ||f_t||_{S^{q_t}}$, note that the energy function $\mE_i$ is nonnegative and bounded above by
\[
    \sum_{\substack{t \in \{1, \ldots, k\} \\ i_t = i}} \mE_{f_t,\ell} + 
    \sum_{\substack{t \in \{1, \ldots, k\} \\ j_t = i}} \mE_{f_t^T,\ell} 
    \le 
    \sum_{\substack{t \in \{1, \ldots, k\} \\ i_t = i}} ||f_t||_{S^{q_t}}^{2^\ell} + 
    \sum_{\substack{t \in \{1, \ldots, k\} \\ j_t = i}} ||f_t||_{S^{q_t}}^{2^\ell} 
    \le 2k 
    M^{2^\ell}.
\]
Thus the resulting $\mP_i$ satisfy 
\[
    \log |\mP_i| \ll \frac{1}{c\eps^{2^\ell}} \ll_c \frac{1}{\eps^{\alpha(q,\infty)}}
\]
and, due to the conclusion of \cref{lem:energy-optimization}, \cref{eq:joint-partitions} and \cref{eq:complexity-bounds},
\[
    \mE_i(\mR_i) - \mE_i(\mP_i)
    \ll_k c\eps^{2^\ell} \cdot 2k M^{2^\ell} \cdot  \log |\mQ_i|
    \ll_k 
    c (\eps M)^{2^\ell}.
\]
Choosing $c = c(k,q) > 0$ small enough in terms of $k$ and $\ell = \ell(q)$, we conclude that
\[
    \sum_{\substack{t \in \{1, \ldots, k\} \\ i_t = i}} (\mE_{f_t,\ell}(\mR_i) -
    \mE_{f_t,\ell}(\mP_i)) 
    + 
    \sum_{\substack{t \in \{1, \ldots, k\} \\ j_t = i}} (\mE_{f_t^T,\ell}(\mR_i) -
    \mE_{f_t^T,\ell}(\mP_i))
    \le \left(\frac{\eps M}{2}\right)^{2^\ell},
\]
where each term in the two sums is nonnegative.
Putting this together with \cref{eq:final-energy-bound}, we obtain
\[
    ||\E(f_t \mid \mR_{i_t} \otimes \mR_{j_t}) - \E(f_t \mid \mP_{i_t} \otimes \mP_{j_t}))||_{S^\infty} \le 
    2^{-1} \eps M + 2^{-1} \eps M
    =
    \eps M,
\]
and taking a maximum over $t \in \{1, \ldots, k\}$ yields the desired bound in \cref{eq:first-supremum}. 

To establish \cref{eq:second-supremum}, we choose the partitions $\mP_i'$ via \cref{lem:energy-optimization} in the same way, except that we use the complexity function $\H(\cdot)$ instead of $\log |\cdot|$. This yields
\[
    \H(\mP_i') \ll_c \frac{1}{\eps^{\alpha(q,\infty)}}
    \qquad \quad 
    \text{and} 
    \qquad \quad 
    \mE_i(\mR_i') - \mE_i(\mP_i')
    \ll_k c\eps^{2^\ell} \cdot 2k M^{2^\ell} \cdot  \H(\mQ_i'),
\]
and choosing a suitable constant $c = c(k, q)$ is now possible due to the entropy bound in  \cref{eq:complexity-bounds}.
\end{proof}

\begin{remark}
One can also give a constructive proof of a statement like \cref{thm:regularity} using singular value decomposition (related to the spectral methods of Frieze and Kannan \cite{frieze1999simple,frieze1999quick}),
at the expense of inserting cutoffs of the form $\one_{R_{i}^c \times R_{j}^c}$ in the uniformity norms, for some small exceptional sets $R_i$. Such cutoffs are compatible with our methods, and can eliminate the need to use entropy as a measure of complexity. We chose the method of proof above partly since we find it cleaner, and partly since the generalized $L^2$ energies from \cref{def:generalized-L2-energy} might find other applications (in particular, they may have higher-order combinatorial analogues, applicable in the setting of hypergraphs).
\end{remark}

\section{The counting lemma} \label{sec:counting}

For comparison purposes, we first state a version of the $k$-cycle counting lemma for dense weighted graphs (see a similar version for triangles, e.g., in \cite[Lemma 3.4]{conlon2021regularity}).

\begin{knownlemma}[Dense cycle counting lemma]
Let $k \in \Z_{\ge 3}$ and $\eps, C > 0$. Let $(X_i, \P_i)$ be finite probability spaces and $f_i, \tilde{f}_i : X_i \times X_{i+1} \to [0, \infty)$ be functions satisfying $||f_i||_{L^\infty} \le C$, $||\tilde{f}_i||_{L^\infty} \le C$, and $||f_i - \tilde{f}_i||\sq \le \eps$, for $i \in \Z/k\Z$. Then
\[
    \left\vert \E_{x_1, \ldots, x_k} \prod_{i=1}^k f_{i}(x_i,x_{i+1})
    -
    \E_{x_1, \ldots, x_k} \prod_{i=1}^k \tilde{f}_{i}(x_i,x_{i+1}) 
    \right\vert
    \le k C^{k-1}
    \eps.
\]
\end{knownlemma}

Like in the case of regularity lemmas, a \emph{sparse} counting lemma should not require $L^\infty$ bounds on the functions $f_i$. The counting lemmas of \cite{conlon2014green,conlon2021regularity,conlon2021graphs}, for instance, assume that $\tilde{f}_i$ remain $L^\infty$-bounded, but have weaker assumptions about $f_i$. In our sparse counting lemma below, we eliminate all of the $L^\infty$ constraints, by replacing them with appropriate Schatten norm bounds. This weakening of the \emph{density norm} bounds comes at the expense of stronger \emph{uniformity norm} bounds, as we also replace the cut norm by the appropriate Schatten norms.

\begin{lemma}[Sparse cycle counting lemma with Schatten norms] \label{lem:counting}
Let $k \in \Z_{\ge 2}$, $q_1, \ldots, q_k \in [1, \infty)$, and $r_1, \ldots, r_k \in (1, \infty]$ be such that for each $i \in \Z/k\Z$,
\[
    \frac{1}{r_i} + \sum_{j \neq i} \frac{1}{q_j} \ge 1.
\]
Also, let $\eps, C > 0$, $(X_i, \P_i)$ be finite probability spaces, and $f_{i}, \tilde{f}_{i} : X_i \times X_{i+1} \to \C$ be functions satisfying $||f_{i}||_{S^{q_i}} \le C$, $||\tilde{f}_{i}||_{S^{q_i}} \le C$, and $||f_{i} - \tilde{f}_{i}||_{S^{r_i}} \le \eps$, for $i \in \Z/k\Z$. Then
\[
    \left\vert \E_{x_1, \ldots, x_k} \prod_{i=1}^k f_{i}(x_i,x_{i+1})
    -
    \E_{x_1, \ldots, x_k} \prod_{i=1}^k \tilde{f}_{i}(x_i,x_{i+1}) 
    \right\vert
    \le k C^{k-1}
    \eps.
\]
\end{lemma}

\begin{proof}
Replacing $f_{i}$ with $\tilde{f}_{i}$ one $i$ at a time, it is enough to show (by the circular symmetry of the indices and the equivalent conditions on $f_{i}$ and $\tilde{f}_{i}$) that
\[
    \left\vert \E_{x_1, \ldots, x_k} (f_{1} - \tilde{f}_{1})(x_1,x_2) \prod_{i=2}^k f_{i}(x_i,x_{i+1}) \right\vert \le C^{k-1} \eps.
\]
But is an immediate consequence of the cycle counting inequality from \cref{lem:Schatten-norms-functions}, our assumed Schatten norm bounds, and the fact that $\frac{1}{r_i} + \sum_{j \neq i} \frac{1}{q_j} \ge 1$.
\end{proof}

\begin{remark}
The counting lemmas considered in \cite{conlon2021regularity,conlon2021graphs} are one-sided, in the sense that they only provide lower bounds for expressions like $\E_{x_1, \ldots, x_k} \prod_{i=1}^k f_{i}(x_i,x_{i+1})$. Conlon, Fox, Sudakov and Zhao also asked \cite{conlon2021graphs} about potential two-sided results: specifically, for which (connected) subgraphs $F$ is there a two-sided $F$-counting lemma, in graphs $G$ with no $4$-cycles? \cref{lem:counting} answers this question when $F$ is a cycle or a path, in the broader setting of graphs $G$ with not too many $4$-cycles; the only drawback is that it uses the slightly stronger $S^\infty$ uniformity norm instead of the cut norm. The reader should compare the corollary below with the notion of \emph{countable} graphs from \cite{conlon2021graphs}.
\end{remark}

\begin{corollary}[Two-sided counting lemma in graphs with not too many $C_4$'s] \label{cor:two-sided-counting}
For any graph $F \in \{C_k : k \ge 5\} \cup \{P_k : k \ge 1\}$ and any $\epsilon, C > 0$, there exists $\delta > 0$ such that the following holds. Let $G$ be an $n$-vertex graph with $\Inj(C_4, G) \le Cn^2$, and $H \in [0, \infty)^{V(G) \times V(G)}$ be a symmetric matrix with $\Tr(H^4) \le Cn^4$, such that
\[
    \left\vert \frac{\sum_{x,y} \one_{G}(x,y) a(x)b(y)}{n^{3/2}} - \frac{\sum_{x,y} H(x,y)a(x)b(y)}{n^2} \right\vert \le \delta
    \qquad
    \forall a, b : V(G) \to \C,\  ||a||_{2}, ||b||_{2} \le \sqrt{n}.
\]
Then for every family $(A_v)_{v \in V(F)}$ of subsets $A_v \subset V(G)$, one has
\[
    \left\vert 
    \frac{\sum_{x_v \in A_v, \forall v \in V(F)} \prod_{\{u, v\} \in E(F)} \one_{G}(x_u,x_v)}{n^{|V(F)| - |E(F)|/2}}
    -
    \frac{\sum_{x_v \in A_v, \forall v \in V(F)} \prod_{\{u, v\} \in E(F)} H(x_u,x_v)}{n^{|V(F)|}}
    \right\vert
    \le \epsilon.
\]
\end{corollary}

\begin{proof}
Equip $V(G)$ with the uniform probability measure, and define $f, \tilde{f} : V(G) \times V(G) \to [0, \infty)$ by $f := \one_{G}/p$ and $\tilde{f} := H$, for $p = n^{-1/2}$ (one can in fact use any $p \ge n^{-1/2}$ to obtain a slightly more general result). A short computation of Conlon, Fox, Sudakov and Zhao \cite[p.\ 2389, below $(11)$]{conlon2021regularity} shows that for any $n$-vertex graph $G$, one has
\begin{equation} \label{eq:4-cycle-counts}
    \Hom(C_4, G) 
    \ll n^2 + \Inj(C_4, G),
\end{equation}
with an absolute implied constant. Using \cref{eq:2k-cycles} and our assumption that $\Inj(C_4, G) \le Cn^4$, this yields
\[
    ||f||_{S^4}^4 = \frac{\Hom(C_4, G)}{p^4 n^4}
    \le \frac{C' n^2}{p^4 n^4} = C',
\]
for some constant $C' > 0$ (depending only on $C$). Similarly, $\Tr(H^4) \le Cn^4$ translates to $||\tilde{f}||_{S^4} \le C^{1/4}$, and the assumed $\delta$-bound involving functions $a, b : V(G) \to \C$ simply states that
\[
    ||f - \tilde{f}||_{S^\infty} \le \delta,
\]
by \cref{eq:Schatten-infty} (in \cite{conlon2021graphs}, $a$ and $b$ are required to be indicator functions of subsets of $V(G)$, which amounts to the less restrictive cut norm bound $||f - \tilde{f}||\sq \le \delta$). Finally, the desired conclusion translates to
\begin{equation} \label{eq:counting-conclusion-1}
    \left\vert 
    \E_{x_1, \ldots, x_k} \prod_{v \in \Z/k\Z} (f \one_{A_v \times A_{v+1}})(x_v,x_{v+1})
    -
    \E_{x_1, \ldots, x_k} \prod_{v \in \Z/k\Z} (\tilde{f} \one_{A_v \times A_{v+1}})(x_v,x_{v+1})
    \right\vert
    \le \epsilon,
\end{equation}
when $F = C_k$ is a $k$-cycle with $k \ge 5$ (and $V(F) = \Z/k\Z$), respectively
\begin{equation} \label{eq:counting-conclusion-2}
    \left\vert 
    \E_{x_1, \ldots, x_k} \prod_{v = 1}^{k-1} (f \one_{A_v \times A_{v+1}})(x_v,x_{v+1})
    -
    \E_{x_1, \ldots, x_k} \prod_{v = 1}^{k-1} (\tilde{f} \one_{A_v \times A_{v+1}})(x_v,x_{v+1})
    \right\vert
    \le \epsilon,
\end{equation}
when $F = P_{k-1}$ is a $(k-1)$-path with $k \ge 2$ (and vertex set $V(F) = \{1, \ldots, k\}$). Letting $f_v := f \one_{A_v \times A_{v+1}}$ and $\tilde{f}_v := \tilde{f} \one_{A_v \times A_{v+1}}$, we still have $||f_v||_{S^4}, ||\tilde{f}_v||_{S^4} \ll_C 1$ and $||f_v - \tilde{f}_v||_{S^\infty} \le \delta$ for each $v$, by \cref{lem:Schatten-norms-functions}. 

Thus \cref{eq:counting-conclusion-1} follows immediately from \cref{lem:counting}, by picking a small enough $\delta$ and noting that $\frac{1}{\infty} + \sum_{v=1}^{k-1} \frac{1}{4} \ge 1$. Similarly, \cref{eq:counting-conclusion-2} follows from \cref{lem:counting}, by inserting constant functions $f_k = \tilde{f}_k := \one_{V(G)}$ with $||f_k||_{S^1} = ||\tilde{f}_k||_{S^1} = 1$ (by \cref{lem:Schatten-norms-functions}), and noting that $\frac{1}{\infty} + \sum_{v=1}^{k-2} \frac{1}{4} + 1 \ge 1$. 
\end{proof}

Our next lemma is a bit technical, but it will ultimately allow us to consider only proper $k$-cycles (as opposed to homomorphic copies of $k$-cycles) in results about non-$k$-partite graphs such as \cref{cor:graphs-few-4-cycles}. This is done by following an idea from \cite{conlon2021regularity} to carefully partition the space $X$ into $k$ parts $X^{(1)}, \ldots, X^{(k)}$, and to only consider the tuples $(x_1, \ldots, x_k) \in X^{(1)} \times \cdots \times X^{(k)}$. 

\begin{lemma}[Restricting to distinct vertices] \label{lem:distinct-vertices}
Let $k \in \Z_{\ge 2}$, $(X, \P)$ be a finite space equipped with uniform probability, and $f : X^k \to [0,\infty)$ be a $(\bigotimes_{i=1}^k \mP)$-measurable function, for some partition $\mP$ of $X$. If $R$ denotes the union of all parts of $\mP$ with fewer than $k(k+1)$ elements, then there exists a partition of $X$ into $k$ disjoint parts $X^{(1)}, \ldots, X^{(k)}$ such that
\[
    \E_{x_1, \ldots, x_k} f(x_1,\ldots,x_k) \prod_{i=1}^k \one_{R^c}(x_i)
    \ll_k \E_{x_1, \ldots, x_k} f(x_1,\ldots,x_k) \prod_{i=1}^k\one_{X^{(i)}}(x_i).
\]
\end{lemma}

\begin{proof}
We split each part $A \in \mP$ arbitrarily into $k$ (possibly empty) parts of sizes differing by at most $1$, labeled by $A^{(i)}$ for $i \in \Z/k\Z$, and let $X^{(i)}$ be the union of all such $A^{(i)}$.

Now let $A_1, \ldots, A_k$ denote some (not necessarily distinct) parts of $\mP$, which have at least $k(k+1)$ vertices (i.e., $A_1, \ldots, A_k \subset R^c$). Consider the expectation
\[
    \E_{x_1, \ldots, x_k} f(x_1,\ldots,x_k) \prod_{i=1}^k \one_{X^{(i)} \cap A_i}(x_i),
\]
where $x_1, \ldots, x_k$ still range over $X$, although only the values $(x_1, \ldots, x_k) \in A_1 \times \cdots \times A_k$ are relevant. Note that for these values, $f(x_1,\ldots,x_k)$ is equal to some nonnegative constant $c(A_1,\ldots,A_k)$. Moreover, by the construction of $X^{(i)}$, the intersection $X^{(i)} \cap A_i$ contains at least 
\[
    \left\lfloor \frac{|A_i|}{k} \right\rfloor \ge \frac{|A_i| - k}{k} \ge \frac{|A_i|}{k+1}
\]
vertices, and thus $\E_{x_i} \one_{X^{(i)} \cap A_i}(x_i) \ge \frac{1}{k+1} \E_{x_i} \one_{A_i}(x_i)$ (since we are using the uniform probability distribution on $X$). It follows that
\[
\begin{aligned}
    \E_{x_1,\ldots,x_k} f(x_1,\ldots,x_k) \prod_{i=1}^k \one_{X^{(i)} \cap A_i}(x_i)
    &= 
    c(A_1, \ldots, A_k) \E_{x_1,\ldots,x_k} \prod_{i=1}^k \one_{X^{(i)} \cap A_i}(x_i)
    \\ 
    &\ge 
    \frac{1}{(k+1)^k} c(A_1, \ldots, A_k) \E_{x_1,\ldots,x_k} \prod_{i=1}^k \one_{A_i}(x_i)
    \\
    &=
    \frac{1}{(k+1)^k} \E_{x_1,\ldots,x_k} f(x_1,\ldots,x_k) \prod_{i=1}^k \one_{A_i}(x_i),
\end{aligned}
\]
and summing over all choices of $A_1 \times \cdots \times A_k$, we conclude that
\[
\begin{aligned}
    \E_{x_1,\ldots,x_k} f(x_1,\ldots,x_k) \prod_{i=1}^k \one_{R^c}(x_i)
    &\ll_k 
    \E_{x_1,\ldots,x_k} f(x_1,\ldots,x_k) \prod_{i=1}^k \one_{X^{(i)} \cap R^c}(x_i)
    \\
    &\le 
    \E_{x_1,\ldots,x_k} f(x_1,\ldots,x_k) \prod_{i=1}^k \one_{X^{(i)}}(x_i),
\end{aligned}
\]
using the nonnegativity of $f$ once again.
\end{proof}

We can now present the natural consequence of our regularity and counting lemmas (with no transference-step machinery or assumptions involved), which reduces a general $k$-cycle count in a weighted $k$-partite graph to an analogous count where the weight functions are fairly structured. When all vertex sets are equal, \cref{lem:distinct-vertices} also allows us to consider strictly proper $k$-cycles (with distinct vertices) in the upper bound.

\begin{proposition}[Reduction to low complexity] \label{prop:reduction-to-structure}
For any $k \in \Z_{\ge 2}$ and $q_1, \ldots, q_k \in [1, \infty)$ with $\sum_i \frac{1}{q_i} > 1$, and any $\eps, C > 0$, there exists $\tau > 0$ such that the following hold. Let $(X_i, \P_i)$ be finite probability spaces, and $f_{i} : X_i \times X_{i+1} \to \C$ satisfy $||f_{i}||_{S^{q_i}} \le C$, for $i \in \Z/k\Z$.
\begin{itemize}
    \item[$(i)$.] There exist partitions $\mP_i'$ of $X_i$ with $\H(\mP_i') = O_{\eps,k,q_1, \ldots, q_k,C}(1)$, such that
    \[
        \left\vert \E_{x_1, \ldots, x_k} \prod_{i=1}^k f_{i}(x_i,x_{i+1})
        -
        \E_{x_1, \ldots, x_k} \prod_{i=1}^k \E(f_{i} \mid \mP'_i \otimes \mP'_{i+1})(x_i,x_{i+1}) 
        \right\vert
        \le
        \eps.
    \]
    \item[$(ii)$.] If all $f_{i}$ take values in $[0, \infty)$, then there exist partitions $\mP_i = \{R_i, A_{i,1}, A_{i,2}, \ldots\}$ of $X_i$, with $\P_i(R_i) \le \eps^2$ and $\P(A_{i,j}) \ge \tau$ for all $i, j$, such that
    \[
        \E_{x_1, \ldots, x_k} \prod_{i=1}^k \E(f_{i} \mid \mP_i \otimes \mP_{i+1})(x_i, x_{i+1})\one_{R_i^c}(x_i) \le \eps + \E_{x_1, \ldots, x_k} \prod_{i=1}^k f_{i}(x_i,x_{i+1}).
    \]
    \item[$(iii)$.] If additionally all spaces $(X_i, \P_i) = (X, \P)$ are equal and $\P$ is the uniform probability distribution, and $|X|$ is sufficiently large in terms of $\eps,k,q_1, \ldots, q_k, C$, then in part (ii) one can take all $\mP_i = \mP$ and $R_i = R$ to be equal, and also conclude the stronger bound
    \[
        \E_{x_1, \ldots, x_k} \prod_{i=1}^k \E(f_{i} \mid \mP \otimes \mP)(x_i,x_{i+1}) \one_{R^c}(x_i) 
        \ll_k
        \eps + \E_{x_1, \ldots, x_k} \one_{x_1,\ldots,x_k \textnormal{ distinct}}\prod_{i=1}^k f_{i}(x_i,x_{i+1}).
    \]
\end{itemize}
\end{proposition}

\begin{remark}
Parts $(ii)$ and $(iii)$ of \cref{prop:reduction-to-structure} yield partitions $\mP_i$ whose parts all have probabilities at least $\tau$ (except possibly for one part $R_i$), which automatically implies the cardinality bound $|\mP_i| \le \frac{1}{\tau} + 1 = O_{\eps,k,q_1, \ldots, q_k,C}(1)$. This is already better than the entropy bound $\H(\mP'_i) = O_{\eps,k,q_1, \ldots, q_k,C}(1)$ from part $(i)$, but the more explicit restriction on the structure of $\mP_i$ will be useful in \cref{sec:transference}.
\end{remark}

\begin{proof}
Assume without loss of generality that $\eps/(kC^k) < 1$ (by making $\eps$ smaller if necessary).
To prove $(i)$, apply \cref{thm:regularity} for the functions $f_{i}$, and with
\begin{equation} \label{eq:r-choice}
\begin{aligned}
    r_i := q_i + q_i^2\left(\sum_j \frac{1}{q_j}-1\right) > q_i
    \qquad 
    &\Rightarrow 
    \qquad 
    \frac{1}{q_i} - \frac{1}{r_i} \le \frac{r_i - q_i}{q_i^2} = \sum_j \frac{1}{q_j}-1
    \\
    &\Rightarrow 
    \qquad 
    \frac{1}{r_i} + \sum_{j \neq i} \frac{1}{q_j} \ge 1.
\end{aligned}
\end{equation}
This produces partitions $\mP_i'$ of $X_i$ such that $\H(\mP_i') = O_{\eps,k,q_1, \ldots, q_k,C}(1)$ and 
\[
    \max_i ||f_{i} - \E(f_{i} \mid \mP_i' \otimes \mP_{i+1}')||_{S^{r_i}} \le \frac{\eps}{kC^k} \max_i ||f_{i}||_{S^{q_i}} \le \frac{\eps}{kC^{k-1}}.
\]
The conclusion now follows from our counting lemma (\cref{lem:counting}).

For $(ii)$, consider the partitions $\mP_i'$ from $(i)$, and let $\mP_i$ be obtained from $\mP'_i$ by combining all parts with probability less than $\tau$ into one part $R_i$ (for some parameter $\tau \in (0, 1)$ to be chosen shortly). Then, if $p_1, p_2, \ldots$ denote the probabilities of all parts of $\mP'_i$, we have 
\begin{equation} \label{eq:entropy-bound}
\begin{aligned}
    \P_i(R_i) = \sum_{p_n < \tau} p_n 
    &=
    \frac{1}{\log (1/\tau)} \sum_{p_n < \tau} p_n \log \frac{1}{\tau}
    \\
    &\le 
    \frac{1}{\log (1/\tau)} \sum_{n} p_n \log \frac{1}{p_n}
    =
    \frac{\H(\mP_i')}{\log(1/\tau)} = O_{\eps,k,q,C}\left(\frac{1}{\log(1/\tau)} \right).
\end{aligned}
\end{equation}
We can thus pick $\tau = \tau(\eps,k,q_1, \ldots, q_k,C)$ small enough so that $\P_i(R_{i}) \le \eps^2$ (uniformly in $i$). Since the functions $\E(f_{i} \mid \mP_i' \otimes \mP_{i+1}')$ and $\E(f_{i} \mid \mP_i \otimes \mP_{i+1})$ are equal inside $R_i^c \times R_{i+1}^c$, we conclude (using $(i)$) that
\[
\begin{aligned}
    \E_{x_1, \ldots, x_k} \prod_{i=1}^k \E(f_{i} \mid \mP_i \otimes \mP_{i+1})\one_{R_{i}^c}(x_i) 
    &\le
    \E_{x_1, \ldots, x_k} \prod_{i=1}^k \E(f_{i} \mid \mP_i' \otimes \mP_{i+1}')(x_i,x_{i+1}) 
    \\
    &\le \eps + \E_{x_1, \ldots, x_k} \prod_{i=1}^k f_{i}(x_i,x_{i+1}).
\end{aligned}
\]

Finally, to prove $(iii)$, we apply \cref{thm:regularity} as before, except that now all vertex sets are equal (so $m = 1$). This produces a partition $\mP'$ of $X$ such that $\H(\mP') = O_{\eps,k,q_1, \ldots, q_k,C}(1)$ and
\[
    \max_i ||f_{i} - \E(f_{i} \mid \mP' \otimes \mP')||_{S^{r_i}} \le \frac{\eps}{kC^k} \max_i ||f_{i}||_{S^{q_i}} \le \frac{\eps}{kC^{k-1}}.
\]
We repeat the construction of $\mP_i$ and $R_i$ from part $(ii)$ (except that now there is only one partition), to obtain a partition $\mP$ with $|\mP| = O_{\eps,k,q_1, \ldots, q_k,C}(1)$, and a $\mP$-measurable set $R$ with $\P(R) \le \eps^2$, which is the union of all parts of $\mP'$ with probability less than $\tau = \tau(\eps,k,q_1,\ldots,q_k,C)$.

Now assume that $|X|$ is large enough so that $k(k+1)/|X| < \tau$, and apply \cref{lem:distinct-vertices} for the partition $\mP'$ and the $(\bigotimes_{i=1}^k \mP')$-measurable function 
\[
    f(x_1,\ldots,x_k) := \prod_{i=1}^k\E(f_{i} \mid \mP' \otimes \mP')(x_i, x_{i+1}),
\]
to obtain a partition of $X$ into $k$ disjoint parts $X^{(1)}, \ldots, X^{(k)}$, such that
\[
    \E_{x_1,\ldots,x_k} \prod_{i=1}^k \E(f_{i} \mid \mP' \otimes \mP')(x_i, x_{i+1}) \one_{R^c}(x_i) \ll_k 
    \E_{x_1,\ldots,x_k} \prod_{i=1}^k
    \E(f_{i} \mid \mP' \otimes \mP')(x_i, x_{i+1}) \one_{X^{(i)}}(x_i).
\]
Denoting
\[
\begin{aligned}
    g_{i}(x_i, x_{i+1}) &:= f_{i}(x_i,x_{i+1}) \one_{X^{(i)}}(x_i), \\
    \tilde{g}_{i}(x_i, x_{i+1}) &:= \E(f_{i} \mid \mP' \otimes \mP')(x_i,x_{i+1}) \one_{X^{(i)}}(x_i),
\end{aligned} 
\]
\cref{lem:Schatten-norms-functions} implies that
\[
    ||g_{i} - \tilde{g}_{i}||_{S^{r_i}} \le ||f_{i} - \E(f_{i} \mid \mP' \otimes \mP')||_{S^{r_i}} \le \frac{\eps}{kC^{k-1}}.
\]
We can thus apply our counting lemma, \cref{lem:counting}, for the spaces $X_i = X$ and the functions $g_{i}, \tilde{g}_{i}$, to obtain
\[
    \left\vert \E_{x_1,\ldots,x_k} \prod_{i=1}^k g_{i}(x_i,x_{i+1}) - \E_{x_1,\ldots,x_k} \prod_{i=1}^k \tilde{g}_{i}(x_i,x_{i+1})\right\vert \le \eps.
\]
The point of splitting $X$ into $k$ parts and inserting the indicator functions $\one_{X^{(i)}}(x_i)$ was to force the product $\prod_{i=1}^k g_{i}(x_i,x_{i+1})$ to vanish when $x_1, \ldots, x_k$ are not pairwise distinct; it follows that 
\[
    \E_{x_1,\ldots,x_k} \prod_{i=1}^k g_{i}(x_i,x_{i+1}) \le 
    \E_{x_1, \ldots, x_k} \one_{x_1,\ldots,x_k \text{ distinct}}\prod_{i=1}^k f_{i}(x_i,x_{i+1}),
\]
and thus putting things together, we have
\[
\begin{aligned}
    \E_{x_1,\ldots,x_k} \prod_{i=1}^k \E(f_{i} \mid \mP' \otimes \mP')(x_i, x_{i+1}) \one_{R^c}(x_i) 
    &\ll_k
    \E_{x_1,\ldots,x_k} \prod_{i=1}^k
    \tilde{g}_{i}(x_i, x_{i+1}) 
    \\
    &\le
    \eps +
    \E_{x_1,\ldots,x_k} \prod_{i=1}^k
    g_{i}(x_i, x_{i+1}) 
    \\
    &\le
    \eps + \E_{x_1, \ldots, x_k} \one_{x_1,\ldots,x_k \textnormal{ distinct}}\prod_{i=1}^k f_{i}(x_i,x_{i+1}).
\end{aligned}
\]
But the functions $\E(f_{i} \mid \mP' \otimes \mP')$ and $\E(f_{i} \mid \mP \otimes \mP)$ are equal inside $R^c \times R^c$, so the same conclusion holds for $\mP$ instead of $\mP'$.
\end{proof}

\section{Transference-free results} \label{sec:transference-free}

When counting paths or even-length cycles, \cref{prop:reduction-to-structure} already has some interesting direct applications, with no transference step (and thus no dense removal lemma or tower-exponential bounds) required. The main relevant result here is the following.

\begin{theorem}[Expected number of paths and even-length cycles] \label{thm:expected-cycles-paths}
Fix $\epsilon, C > 0$ and integers $j \ge 1$, $k > \ell \ge 2$. Then for all $n \ge n_0(\epsilon, C, j, k, \ell)$ and all $p > 0$, any $n$-vertex graph $G$ satisfying
\[
    \Hom(P_1, G) \ge \epsilon p n^2 \qquad \qquad 
    \text{and} \qquad\qquad 
    \Hom(C_{2\ell}, G) \le C p^{2\ell} n^{2\ell}
\]
also satisfies
\[
    \Inj(P_{j}, G) \asymp p^{j} n^{j+1} 
    \qquad\qquad \text{and} 
    \qquad\qquad \Inj(C_{2k},G) \asymp p^{2k} n^{2k}.
\]
\end{theorem}

\begin{remark}
To clarify our notation, the implied constants may depend on the fixed parameters, but not on $n$ and $p$. It is worth noting that a random (Erd\"os--R\'enyi) graph with $\asymp pn^2$ edges has an expected number of $\asymp p^j n^{j+1}$ $P_j$'s and $\asymp p^{2k} n^{2k}$ $C_{2k}$'s, and the conclusion of \cref{thm:expected-cycles-paths} confirms these predictions under relatively mild assumptions. What makes this result nontrivial is the fact that we are counting \emph{injective} copies of $P_j$ and $C_{2k}$; if we counted homomorphic copies instead, the result would be a relatively straight-forward corollary of the fact that paths and even-length cycles are bipartite graphs with Sidorenko's property \cite{sidorenko1991inequalities}.
\end{remark}

We work our way to proving \cref{thm:expected-cycles-paths}, starting with a couple of propositions for weighted graphs.

\begin{proposition}[Weighted graphs with few paths] \label{prop:few-paths}
Let $k \in \Z_{\ge 1}$, $q \in [1, \infty)$, $\eps, C > 0$, and $X$ be a large enough finite set (in terms of $\eps, k, q, C$), equipped with uniform probability. If $f : X \times X \to [0, \infty)$ is symmetric (i.e., $f(x, y) = f(y, x), \forall x, y \in X$), $||f||_{S^q} \le C$, and
\[
    \E_{x_1,\ldots,x_{k+1}} \one_{x_1,\ldots,x_{k+1} \textnormal{ distinct}} \prod_{i=1}^k f(x_i, x_{i+1}) \le \eps,
\]
then $||f||_{L^1} \ll_k \eps^{1/k} + C\eps$.
\end{proposition}

\begin{proof}
Apply part $(iii)$ of \cref{prop:reduction-to-structure} with $k+1$ in place of $k$, $q_1 = \cdots = q_{k} = q$ and $q_{k+1} = 1$, and the functions $f_{1} = \cdots = f_{k} = f$ and $f_{k+1} = \one_{X_k \times X_1}$ (recall that $||f_{k+1}||_{S^1} = 1$ by \cref{lem:Schatten-norms-functions}). This produces a partition $\mP$ of $X$ and a $\mP$-measurable subset $R \subset X$, such that $\P(R) \le \eps^2$ and
\[
    \E_{x_1,\ldots,x_{k+1}} \prod_{i=1}^k \E(f \mid \mP \otimes \mP)(x_i, x_{i+1}) \one_{R^c}(x_i) \ll_k \eps.
\]
Letting $g := \E(f \mid \mP \otimes \mP) \one_{R^c \times R^c}$, we thus have $\E_{x_1,\ldots,x_{k+1}} \prod_{i=1}^k g(x_i, x_{i+1}) \ll_k \eps$. But since $g$ is symmetric and nonnegative, an older result of Blakley and Roy \cite{blakley1965holder} implies that
\[
    \E_{x_1,\ldots,x_{k+1}}
    \prod_{i=1}^k g(x_i, x_{i+1}) 
    \ge 
    \left( \E_{x_1,x_2} g(x_1, x_2) \right)^k.
\]
(Blakley and Roy state their result in the form $\langle u, G^ku \rangle \ge \langle u, Gu \rangle^k$, for any nonnegative symmetric $n \times n$ matrix $G$, and any $u \in \R^n$ with $||u||_{2} = 1$; we can recover the inequality above by letting $n := |X|$, $G := \Mat(g)$, and $u_x := \sqrt{1/|X|}$ for all $x \in X$. This can also be interpreted as the statement that paths satisfy Sidorenko's conjecture \cite{sidorenko1991inequalities}.) 

We thus get $||g||_{L^1} \ll_k \eps^{1/k}$. But $R^c \times R^c$ is $(\mP \otimes \mP)$-measurable, so 
\[
\begin{aligned}
    \eps^{1/k} \gg_k ||g||_{L^1}
    &= ||f \one_{R^c \times R^c}||_{L^1} 
    \\
    &\ge ||f||_{L^1} - 2||f \one_{R \times X}||_{L^1}. 
\end{aligned}
\]
Recalling from \cref{eq:Schatten-infty} that 
\[
    ||f \one_{R \times X}||_{L^1} \le ||f||_{S^\infty} ||\one_R||_{L^2} ||\one_X||_{L^2} 
    \le ||f||_{S^q} \sqrt{\P(R)} \le C\eps,
\]
we reach the desired conclusion, $||f||_{L^1} \ll_k \eps^{1/k} + C\eps$.
\end{proof}

\begin{proposition}[Weighted graphs with few even cycles] \label{prop:few-even-cycles}
Let $k$ be an even positive integer, $q \in [1, k)$, $\eps, C > 0$, and $X$ be a large enough finite set (in terms of $\eps, k, q$), equipped with uniform probability. If $f : X \times X \to [0, \infty)$ is symmetric, $||f||_{S^q} \le C$, and (working with indices modulo $k$)
\[
    \E_{x_1,\ldots,x_k} \one_{x_1,\ldots,x_k \textnormal{ distinct}} \prod_{i \in \Z/k\Z} f(x_i, x_{i+1}) \le \eps,
\]
then $||f||_{L^1} \ll_k \eps^{1/k} + C\eps$.
\end{proposition}

\begin{proof}
Apply part $(iii)$ of \cref{prop:reduction-to-structure} for the functions $f_{i} = f$, to obtain a partition $\mP$ of $X$ and a $\mP$-measurable subset $R \subset X$, such that $\P(R) \le \eps^2$ and
\[
    \E_{x_1,\ldots,x_k} \prod_{i \in \Z/k\Z} \E(f \mid \mP \otimes \mP)(x_i, x_{i+1}) \one_{R^c}(x_i) \ll_k \eps.
\]
Let $g := \E(f \mid \mP \otimes \mP) \one_{R^c \times R^c}$. Since $k$ is an even integer and $g$ is real symmetric, \cref{eq:Schatten-2k} implies that the left-hand side above is precisely $||g||_{S^k}^k$. Since $g$ is nonnegative, we further have
\[
    ||g||_{L^1} = ||g||\sq \le ||g||_{S^\infty} \le ||g||_{S^k} \ll_k \eps^{1/k}.
\]
(The bound $||g||_{L^1} \le ||g||_{S^k}$ is essentially the statement that even-length cycles satisfy Sidorenko's conjecture \cite{sidorenko1991inequalities}.)

But $R^c \times R^c$ is $(\mP \otimes \mP)$-measurable, so $||g||_{L^1} = ||f \one_{R^c \times R^c}||_{L^1}$. As in the previous proof, this further implies that $||f||_{L^1} \ll_k \eps^{1/k} + C\eps$ since $\P(R) \le \eps^2$.
\end{proof}

We take a final intermediate step to proving \cref{thm:expected-cycles-paths} through the following result, which resembles a removal lemma.

\begin{proposition}
\label{prop:graphs-transference-free}
For any integers $j \ge 1$ and $k > \ell \ge 2$, and any $\epsilon, C > 0$, there exist $n_0, \delta > 0$ such that the following holds. For any $n \ge n_0$ and $p > 0$, if an $n$-vertex graph $G$ with $\Hom(C_{2\ell},G) \le C p^{2\ell} n^{2\ell}$ has
\[
    \Inj(P_{j}, G) \le \delta p^{j} n^{j+1}
    \qquad\qquad \text{or} \qquad\qquad 
    \Inj(C_{2k}, G) \le \delta p^{2k} n^{2k},
\]
then $G$ has fewer than $\epsilon pn^2$ edges.
\end{proposition}

\begin{proof}
Let $q := 2\ell$, $X := V(G)$ (equipped with uniform probability) and $f := \one_{G}/p : X \times X \to [0, \infty)$. Then by \cref{eq:2k-cycles},
\[
    ||f||_{S^q}^q =  ||\one_{G} / p||_{S^{2\ell}}^{2\ell} = \frac{\Hom(C_{2\ell}, G)}{p^{2\ell}n^{2\ell}} \le C,
\]
so $||f||_{S^q} \le C^{1/q}$, and we aim to show that $||f||_{L^1} \le \epsilon$. But our second assumption about $G$ states that
\[
    \E_{x_1,\ldots,x_{j+1}} \one_{x_1,\ldots,x_{j+1} \textnormal{ distinct}} \prod_{i=1}^{j} f(x_i, x_{i+1}) \le \delta,
\]
or
\[
    \E_{x_1,\ldots,x_{2k}} \one_{x_1,\ldots,x_{2k} \text{ distinct}} \prod_{i \in \Z/2k\Z} f(x_i, x_{i+1}) \le \delta,
\]
for a value of $\delta$ that we have yet to choose. \cref{prop:few-paths,prop:few-even-cycles} now imply that for all large enough $n$,
\[
    ||f||_{L^1} \ll_j \delta^{1/j} + C^{1/q}\delta
    \qquad\qquad 
    \text{or}
    \qquad\qquad 
    ||f||_{L^1} \ll_k \delta^{1/2k} + C^{1/q}\delta,
\]
and we can choose $\delta$ small enough so that the upper bounds above are less than $\epsilon$.
\end{proof}

\begin{proof}[Proof of \cref{thm:expected-cycles-paths}]
Apply \cref{prop:graphs-transference-free} for the same values of $j, k, \ell, \epsilon, C$, to obtain some $n_0, \delta > 0$ such that the following contrapositive statement holds. For any $n \ge n_0$ and $p > 0$, and any $n$-vertex graph $G$ with $\ge \epsilon p n^2$ edges and $\Hom(C_{2\ell}, G) \le C p^{2\ell} n^{2\ell}$, one has
\[
    \Inj(P_{j}, G) > \delta p^{j} n^{j+1}
    \qquad\qquad 
    \text{and} 
    \qquad\qquad 
    \Inj(C_{2k}, G) > \delta p^{2k} n^{2k}.
\]
We wish to find constants $\alpha, \beta$ (depending on $j, k, \ell, \epsilon, C$) such that
\[
    \alpha p^j n^{j+1} \le \Inj(P_{j}, G) \le \beta p^{j} n^{j+1}
    \qquad\quad \text{and} \qquad\quad 
    \alpha p^{2k} n^{2k} \le \Inj(C_{2k}, G) \le \beta p^{2k} n^{2k},
\]
and we can already take $\alpha := 1/\delta$ for the desired lower bounds. For the upper bounds, first note that by \cref{eq:2k-cycles}, the condition $\Hom(C_{2\ell}, G) \le Cp^{2\ell} n^{2\ell}$ means $||\one_{G}/p||_{S^{2\ell}} \le C^{1/{2\ell}}$. Since $||\one_{G}/p||_{S^{2k}} \le ||\one_{G}/p||_{S^{2\ell}}$, we get
\[
    \Inj(C_{2k}, G) \le \Hom(C_{2k}, G)
    =
    n^{2k} ||\one_{G}||_{S^{2k}}^{2k}
    \le 
    C^{k/\ell} p^{2k} n^{2k}.
\]
Similarly, after fixing an ordering of the vertices of $G$ (and equipping $V(G)$ with the uniform probability measure), one has
\[
    \Inj(P_{j},G) \le \Hom(P_{j}, G) = 
    \sum_{v_1, \ldots, v_{j+1} \in V(G)} \prod_{i=1}^k \one_{G}(v_i, v_{i+1})
    =
    n^{j+1} \langle u, \Mat(\one_{G})^{j} u \rangle,
\]
where $u \in \R^n$ is the constant vector with each entry equal to $1/\sqrt{n}$ (and thus $||u||_{2} = 1$). We now have
\[
    \langle u, \Mat(\one_{G})^{j} u\rangle
    \le 
    \left\vert \left\vert \Mat(\one_{G})^{j} \right\vert \right\vert_{S^\infty}
    =
    \left\vert \left\vert \Mat(\one_{G}) \right\vert \right\vert_{S^\infty}^{j} = 
    ||\one_{G}||_{S^\infty}^{j},
\]
since the singular values of a self-adjoint matrix are just the absolute values of its eigenvalues (thus the singular values of $\Mat(\one_G)^j$ are the $j$th powers of the singular values of $\Mat(\one_G)$). Putting things together, we get
\[
    \Inj(P_j,G) \le n^{j+1} ||\one_{G}||_{S^\infty}^{j}
    \le 
    n^{j+1}
    ||\one_{G}||_{S^{2\ell}}^{j}
    \le 
    C^{j/2\ell} p^{j} n^{j+1}.
\]
So we can take $\beta := \max(C^{k/\ell}, C^{j/2\ell})$, and our proof is complete.
\end{proof}

Next, we prove the first two parts \cref{cor:graphs-few-4-cycles}, and leave the last part (which requires a transference step) to \cref{sec:transference}.

\begin{theorem}[Rephrasing of parts $(i)$, $(ii)$ of \cref{cor:graphs-few-4-cycles}]
For any integers $j \ge 1$ and $k \ge 3$, and any $\epsilon, C > 0$, there exist $n_0, \delta > 0$ such that the following holds. For any $n \ge n_0$ and $p \ge n^{-1/2}$, if an $n$-vertex graph $G$ with $\Inj(C_4,G) \le C p^{4} n^{4}$ has
\[
    \Inj(P_{j}, G) \le \delta p^{j} n^{j+1}
    \qquad\qquad \text{or} \qquad\qquad 
    \Inj(C_{2k}, G) \le \delta p^{2k} n^{2k},
\]
then $G$ has fewer than $\epsilon pn^2$ edges.
\end{theorem}

\begin{proof}
This follows directly from \cref{prop:graphs-transference-free} by taking $\ell = 2$, provided that we can use the additional lower bound $p \ge n^{-1/2}$ to upgrade the upper bound $\Inj(C_4, G) \le Cp^4 n^4$ to 
\[
    \Hom(C_4, G) \le C'p^4 n^4,
\]
for another constant $C'$ (depending on $C$). But we already did this in the proof of \cref{cor:two-sided-counting}, using \cref{eq:4-cycle-counts}.
\end{proof}

\section{The transference step} \label{sec:transference}

In this section, we complete the transference step and prove our main removal lemmas for $k$-cycles (and $k$-paths) in weighted sparse graphs. We will ``transfer'' to the following weighted version of the $k$-cycle removal lemma for dense graphs (i.e., assuming $L^\infty$ bounds); this is a fairly classical result, which can be proven along the same lines as \cite[Theorem 4.1]{conlon2021regularity} (see also \cite{tao2007ergodic}).

\begin{knowntheorem}[Dense cycle removal] \label{thm:dense-k-cycle-removal}
For any $k \in \Z_{\ge 3}$ and $\epsilon > 0$, there exists $\delta > 0$ such that the following holds. Let $(X_i, \Sigma_i, \P_i)$ be any probability spaces and $f_{i} : X_i \times X_{i+1} \to [0, 1]$ be measurable functions, for $i \in \Z/k\Z$. If
\begin{equation} \label{eq:few-k-cycles-0}
    \E_{x_1,\ldots,x_k} \prod_{i=1}^k f_{i}(x_i, x_{i+1}) \le \delta,
\end{equation}
then there exist measurable subsets $E_{i} \subset X_i \times X_{i+1}$ with $||f_{i} \one_{E_{i}^c}||_{L^1} \le \epsilon$ for each $i$, such that
\[
    \E_{x_1,\ldots, x_k} \prod_{i=1}^k \one_{E_{i}}(x_i, x_{i+1}) = 0.
\]
\end{knowntheorem}

\begin{remark}
The conclusion of this theorem (and similar ones) is sometimes \cite{tao2007ergodic} phrased with the vanishing of $\prod_{i=1}^k (f_i\one_{E_i})(x_i, x_{i+1})$ rather than $\prod_{i=1}^k \one_{E_i}(x_i, x_{i+1})$. To recover our statement from this one, we can simply intersect $E_i$ with the support of $f_i$.
\end{remark}

\begin{remark}
Thinking of $f_{i}$ as the weight functions of a $k$-partite graph with few $k$-cycles, $\supp(f_i) \setminus E_{i}$ represent the (small) sets of edges that one can remove in order to obtain a $k$-cycle free graph. If one is interested in general weighted graphs, one may take $X_1 = \cdots = X_k = X$, $f_{1} = \cdots = f_{k,1} = f$,
\[
    E := \bigcap_{i=1}^k E_{i},
\]
and bound $||f \one_{E^c}||_{L^1}$ by the triangle inequality.
If one is further interested in undirected graphs (so $f$ is also symmetric), one can always modify $E$ by considering the minimal symmetric set containing it. Of course, one issue with passing to non-$k$-partite graphs this way is that due to vertex collisions, \cref{eq:few-k-cycles-0} now assumes a bound for homomorphic copies of $k$-cycles rather than proper $k$-cycles, which matters for $k > 3$ (unlike in \cref{thm:dense-triangle-removal}). Luckily, this issue can be ammended through the last part of \cref{prop:reduction-to-structure} (the proof of which used \cref{lem:distinct-vertices}).
\end{remark}

\begin{theorem}[Sparse cycle removal] \label{thm:sparse-k-cycle-removal}
For any $k \in \Z_{\ge 3}$ and $q_1, \ldots, q_k \in [1, \infty)$ with $\sum_i \frac{1}{q_i} > 1$, and any $\epsilon, K, C > 0$, there exist $\delta, \tau > 0$ such that the following holds.
Let $(X_i, \P_i)$ be finite probability spaces and $f_{i} : X_i \times X_{i+1} \to [0, \infty)$, for $i \in \Z/k\Z$. Assume that for each $i$,
\begin{itemize}
    \item[$(i)$] \emph{(Bounded Schatten norm).} One has $||f_{i}||_{S^{q_i}} \le C$, and
    \item[$(ii)$] \emph{(Dense pairs condition with large parts).} For any partitions $\mP_i, \mP_{i+1}$ of $X_i, X_{i+1}$ into parts with probabilities $\ge \tau$, one has
    \begin{equation} \label{eq:dense-pairs}
        ||f_{i}\one_{\E(f_{i} \mid \mP_i \otimes \mP_{i+1}) > K} ||_{L^1} \le \epsilon/4.
    \end{equation}
\end{itemize}
If additionally
\begin{equation} \label{eq:few-k-cycles-1}
    \E_{x_1, \ldots, x_k} \prod_{i=1}^k f_{i}(x_i,x_{i+1}) \le \delta,
\end{equation}
then there exist subsets $E_{i} \subset X_i \times X_{i+1}$ with $||f_{i} \one_{E_{i}^c}||_{L^1} \le \epsilon$ for each $i$, such that
\[
    \E_{x_1, \ldots, x_k} \prod_{i=1}^k \one_{E_{i}}(x_i,x_{i+1}) = 0.
\]
Moreover, if all $(X_i, \P_i) = (X, \P)$ are equal and $\P$ is the uniform distribution, and $|X|$ is sufficiently large in terms of $k,q_1,\ldots,q_k,\epsilon,K,C$, then in the dense pairs condition one can assume that $\mP_i = \mP_{i+1}$, and one can replace \cref{eq:few-k-cycles-1} with $\E_{x_1, \ldots, x_k}\one_{x_1,\ldots,x_k\textnormal{ distinct}} \prod_{i=1}^k f_{i}(x_i,x_{i+1})  \le \delta$. If additionally all $f_i = f$ are equal, then one can also replace \cref{eq:dense-pairs} with $||f\one_{\E(f \mid \mP \otimes \mP) > K} \one_{\diag(\mP)^c}||_{L^1} \le \epsilon/8$.
\end{theorem}

\begin{remark}
We recall that for a partition $\mP$ of a set $X$, we denoted $\diag(\mP) := \bigcup_{A \in \mP} (A \times A) \subset X \times X$ in \cref{subsec:partitions}. The fact that we can ignore the contribution of the set $\diag(\mP)$ to the average in \cref{eq:dense-pairs} when all $f_i$'s are equal will be relevant for the last part of \cref{cor:graphs-few-4-cycles}.
\end{remark}

\begin{remark}
One can obtain a sparse $k$-path removal lemma similar to \cref{thm:sparse-k-cycle-removal}, by using the appropriate dense removal lemma, and applying our counting lemma (\cref{lem:counting}) with $f_1 = \tilde{f}_1 = \one_{X_1 \times X_2}$ and $q_1 = 1$; the condition $\sum_i \frac{1}{q_i} > 1$ can be discarded in that case. The resulting removal lemma is superseded by \cref{prop:few-paths} when applied to a single, undirected graph (which is why we leave it to the reader), but it has new content for $k$-partite and/or directed graphs. 
\end{remark}

   
\begin{proof}
Denote $q := (q_1, \ldots, q_k)$; we will choose a parameter $\delta = \delta(\epsilon, K, k, C) > 0$ later in the proof. For now, we fix $\delta$ and apply part $(ii)$ of \cref{prop:reduction-to-structure} with $\delta$ in place of $\eps$, to obtain a parameter $\tau = \tau(\delta, k, q, C) \in (0, 1)$ and partitions $\mP_i$ of $X_i$ whose parts have probabilities $\ge \tau$ (except possibly for one part $R_i \in \mP_i$ with $\P_i(R_i) \le \delta^2$), such that
\[
    \E_{x_1, \ldots, x_k} \prod_{i=1}^k \E(f_{i} \mid \mP_i \otimes \mP_{i+1})(x_i, x_{i+1})\one_{R_i^c}(x_i) 
    \le 
    \delta + \E_{x_1, \ldots, x_k} \prod_{i=1}^k f_{i}(x_i,x_{i+1}).
\]
By \cref{eq:few-k-cycles-1}, the right-hand side is at most $2\delta$.
Now consider, for each $i$, the $(\mP_i \otimes \mP_{i+1})$-measurable functions
\[ 
\begin{aligned}
    \tilde{f}_i := \E(f_{i} \mid \mP_i \otimes \mP_{i+1}) \qquad\quad  \text{and}  \qquad\quad  
    g_i := \frac{\tilde{f}_i \one_{\tilde{f}_i \le 4K}}{4K} \one_{R_i^c \times R_{i+1}^c}.
\end{aligned}
\] 
It follows that for some constant $c_{K,k} > 0$ depending on $K$ and $k$, we have
\begin{equation}
    \E_{x_1, \ldots, x_k} \prod_{i=1}^k g_i(x_i, x_{i+1}) \le c_{K,k} \delta.
\end{equation}
Note that $0 \le g_i \le 1$ everywhere, so we can apply the dense removal lemma (\cref{thm:dense-k-cycle-removal}) for the probability spaces $(X_i, \Sigma(\mP_i), \P_i)$ and the functions $g_i$. We finally pick $\delta = \delta(\epsilon, K, k)$ small enough, as given by \cref{thm:dense-k-cycle-removal}, to find $(\mP_{i} \otimes \mP_{i+1})$-measurable subsets $E_{i} \subset X_i \times X_{i+1}$ such that
\[
    \prod_{i=1}^k \one_{E_{i}}(x_i, x_{i+1}) = 0 \qquad\quad \text{and} \qquad\quad  
     ||g_{i} \one_{E_{i}^c}||_{L^1} \le \frac{\epsilon}{16K},
\]
for all $x_1, \ldots, x_k$, and for all $i$ respectively.
We also make $\delta$ smaller if necessary, in terms of $\epsilon$ and $C$, to guarantee that 
\[
    \delta ||f_i||_{S^\infty} \le \frac{\epsilon}{4}.
\]
(This is possible by the assumed Schatten norm bounds, recalling that $||f_i||_{S^{\infty}} \le ||f_i||_{S^{q_i}} \le C$.)

It remains to show that $||f_{i}\one_{E_{i}^c}||_{L^1} \le \epsilon$. To this end, we split
\begin{equation} \label{eq:edges-decomposition}
\begin{aligned}
    ||f_{i}\one_{E_{i}^c}||_{L^1} 
    &\le
    ||f_i\one_{(R_i^c \times R_{i+1}^c)^c}||_{L^1}
    +
    ||f_i\one_{R_i^c \times R_{i+1}^c}\one_{E_{i}^c}||_{L^1}.
\end{aligned}
\end{equation}
We use \cref{eq:Schatten-infty} to bound the first term:
\begin{equation} \label{eq:edges-first-term}
\begin{aligned}
    ||f_i\one_{(R_i^c \times R_{i+1}^c)^c}||_{L^1}
    &\le 
    ||f_i\one_{R_i \times X_{i+1}}||_{L^1}
    +
    ||f_i\one_{X_i \times R_{i+1}}||_{L^1}
    \\
    &\le ||f_i||_{S^\infty} \left(||\one_{R_i}||_{L^2} + ||\one_{R_{i+1}}||_{L^2}  \right)
    \\
    &\le 
    ||f_i||_{S^\infty} \left(\sqrt{\P_i(R_i)} + \sqrt{\P_{i+1}(R_{i+1})} \right)
    \le 
    2 ||f_i||_{S^{\infty}} \delta
    \le \frac{\epsilon}{2}.
\end{aligned}
\end{equation}
To handle the second term in \cref{eq:edges-decomposition}, note that since $\{\tilde{f}_i \le K\}$, $E_i$ and and $R_i^c \times R_{i+1}^c$ are $(\mP_i \otimes \mP_{i+1})$-measurable sets, we have
\begin{equation} \label{eq:edges-second-term}
\begin{aligned}
    ||f_i\one_{R_i^c \times R_{i+1}^c}\one_{E_{i}^c}||_{L^1} &\le 
    ||f_i\one_{\tilde{f}_i >4K} \one_{R_i^c \times R_{i+1}^c}||_{L^1}
    +
    ||f_i\one_{\tilde{f}_i \le 4K} \one_{R_i^c \times R_{i+1}^c} \one_{E_i^c}||_{L^1}
    \\
    &= 
    ||f_i\one_{\tilde{f}_i >4K} \one_{R_i^c \times R_{i+1}^c}||_{L^1}
    +
    ||\tilde{f}_i\one_{\tilde{f}_i \le 4K} \one_{R_i^c \times R_{i+1}^c} \one_{E_i^c}||_{L^1}
    \\
    &=
    ||f_i \one_{\tilde{f}_i >4K} \one_{R_i^c \times R_{i+1}^c}||_{L^1}
    +
    ||4K g_i \one_{E_i^c}||_{L^1}
    \\
    &\le 
    ||f_i \one_{\tilde{f}_i >4K} \one_{R_i^c \times R_{i+1}^c}||_{L^1} + \frac{\epsilon}{4}.
\end{aligned}
\end{equation}
We are left to handle the term $||\tilde{f}_i\one_{\tilde{f}_i >4K} \one_{R_i^c \times R_{i+1}^c}||_{L^1}$, and we intend to use the dense pairs condition, recalling that $\tilde{f}_i = \E(f_i \mid \mP_i \otimes \mP_{i+1})$. However, the partitions $\mP_j$ (for $1 \le j \le k$) are not directly suitable for this condition, since one part of $\mP_j$ (which is $R_j$) may have probability $< \tau$. Hence if $\P_j(R_j) < \tau$, we define $\mP'_j$ by combining $R_j$ with an arbitrary other part $B_j$ of $\mP_j$ (which has $\P_j(B_j) \ge \tau$); otherwise we let $\mP'_j := \mP_j$. Denoting $f'_i := \E(f_i \mid \mP'_i \otimes \mP_{i+1}')$, we now claim that
\begin{equation} \label{eq:fs-bound}
    \tilde{f}_i \one_{R_i^c \times R_{i+1}^c} \le 4 f'_i.
\end{equation}
Both sides are $(\mP_i \otimes \mP_{i+1})$-measurable, so we must check the inequality in every part $A_i \times A_{i+1}$ of $\mP_i \otimes \mP_{i+1}$. If $A_i = R_i$ or $A_{i+1} = R_{i+1}$, then the left-hand side vanishes. If $\P_i(R_i) < \tau$ and $A_i = B_i$, then we have 
\[
    \E[f_i \mid B_i \times A_{i+1}] = \frac{\E[f_i \one_{B_i \times A_{i+1}}]}{\P_i(B_i)\P_{i+1}(A_{i+1})}\le 
    2\frac{\E[f_i \one_{(B_i \cup R_i) \times A_{i+1}}]}{\P_i(B_i \cup R_i)\P_{i+1}(A_{i+1})}
    =
    2\E[f_i \mid (B_i \cup R_i) \times A_{i+1}],
\]
since $\P_i(B_i) \ge \tau > \P_i(R_i)$. One can repeat this reasoning if $\P_{i+1}(R_{i+1}) < \tau$ and $A_{i+1} = B_{i+1}$, showing that $\tilde{f}_i \one_{R_i^c \times R_{i+1}^c} \le 4f'_i$ whenever $A_i \in \{B_i, R_i\}$ or $A_{i+1} \in \{B_{i+1}, R_{i+1}\}$. In all other cases, the values of $\tilde{f}_i$ and $f'_i$ are equal. Thus \cref{eq:fs-bound} holds, implying that $f'_i > K$ whenever $\tilde{f}_i > 4K$ and $\one_{R_i^c \times R_{i+1}^c} \neq 0$. Hence we have the pointwise bound
\[
    f_i\one_{\tilde{f}_i >4K} \one_{R_i^c \times R_{i+1}^c} 
    \le f_i \one_{f'_i >K}
\]
everywhere, which implies
\begin{equation} \label{eq:edges-third-term}
\begin{aligned}
    ||f_i\one_{\tilde{f}_i >4K} \one_{R_i^c \times R_{i+1}^c}||_{L^1} &\le ||f_i \one_{f'_i >K}||_{L^1} \le \epsilon/4,
\end{aligned}
\end{equation}
by our dense pairs condition (note that each part of $\mP'_j$ has probability $\ge \tau$).

Combining \cref{eq:edges-decomposition,eq:edges-first-term,eq:edges-second-term,eq:edges-third-term} completes our proof in the general case. The statement about equal probability spaces $(X_i, \P_i) = (X, \P)$ (where $\P$ is the uniform probability distribution on $X$) follows analogously, using part $(iii)$ of \cref{prop:reduction-to-structure} instead of part $(ii)$ (in this case we have all $\mP_i = \mP$ and all $R_i = R$). It remains to prove the additional claim that \cref{eq:dense-pairs} can be replaced with
\begin{equation} \label{eq:dense-pairs-1}
    ||f \one_{\E(f \mid \mP' \otimes \mP') > K} \one_{\diag(\mP')^c} ||_{L^1} \le \frac{\epsilon}{8},
\end{equation}
provided that we also have $f_1 = \cdots = f_k = f$; here $\mP'$ is any partition of $X$ into parts with probabilities $\ge \tau$. To this end, we will show that by choosing a smaller value of $\delta$, \cref{eq:dense-pairs-1} actually self-improves to \cref{eq:dense-pairs} (in the form $||f \one_{\E(f \mid \mP' \otimes \mP') > K}||_{L^1} \le \frac{\epsilon}{4}$), for the specific partitions $\mP'$ that were used in the preceding proof (to obtain \cref{eq:edges-third-term}). More precisely, our application of \cref{prop:reduction-to-structure}.$(iii)$ produced a partition $\mP$ of $X$ and a part $R \in \mP$ with $\P(R) \le \delta^2$, such that
\begin{equation} \label{eq:reduced-to-structure}
    \E_{x_1, \ldots, x_k} \prod_{i=1}^k \E(f \mid \mP \otimes \mP)(x_i,x_{i+1}) \one_{R^c}(x_i) 
    \ll_k
    \delta + \E_{x_1, \ldots, x_k} \one_{x_1,\ldots,x_k \textnormal{ distinct}}\prod_{i=1}^k f(x_i,x_{i+1}) 
    \le 2\delta,
\end{equation}
and the dense pairs condition was applied for a partition $\mP'$, obtained by possibly combining $R$ with another part of $\mP$.
Now fixing some part $A \in \mP$ with $A \neq R$ (and thus $A \cap R = \emptyset$), and restricting the implied summation in \cref{eq:reduced-to-structure} to those $x_1, \ldots, x_k \in A$, we get 
\[
\begin{aligned}
    2\delta \gg_k \E_{x_1, \ldots, x_k} \prod_{i=1}^k \E[f \mid A \times A] \one_{A}(x_i) 
    &=
    \E[f \mid A \times A]^k \E_{x_1, \ldots, x_k} \prod_{i=1}^k \one_{A}(x_i)
    \\
    &= 
    \E[f \mid A \times A]^k \P(A)^k.
\end{aligned}
\]
Summing over $A$ yields
\[
    \sum_{\substack{A \in \mP \\ A \neq R}} \E[f \mid A \times A] \P(A)^2
    \ll_k 
    \sum_{A \in \mP} (2\delta)^{1/k} \P(A)
    = (2\delta)^{1/k},
\]
and using the bound $||f||_{S^\infty} \le ||f||_{S^{q_1}} \le C$ (coupled with \cref{eq:Schatten-infty}), we further obtain
\[
\begin{aligned}
    ||f \one_{\diag(\mP')}||_{L^1} &= \sum_{A \in \mP'} \E[f \one_{A \times A}]
    \\
    &\le 
    \E[f \one_{R \times X}] + \E[f \one_{X \times R}]
    +
    \sum_{A \in \mP'} \E[f \one_{(A \setminus R) \times (A \setminus R)}]
    \\
    &\le 
    2||f||_{S^\infty}\sqrt{\P(R)}
    +
    \sum_{\substack{A \in \mP \\ A \neq R}} \E[f \one_{A \times A}]
    \\
    &\le 
    2C \delta +
    \sum_{\substack{A \in \mP \\ A \neq R}} \E[f \mid A \times A] \P(A)^2
    \qquad 
    \ll_k 
    C\delta + (2\delta)^{1/k}.
\end{aligned}
\]
By making $\delta$ smaller (if necessary) in terms of $\epsilon, k, C$, we can thus guarantee that $||f \one_{\diag(\mP')}||_{L^1} \le \epsilon/8$. Combining this with \cref{eq:dense-pairs-1} recovers the original condition in \cref{eq:dense-pairs} (for this specific partition $\mP'$, in the form $||f\one_{\E(f \mid \mP' \otimes \mP') > K}||_{L^1} \le \epsilon/4$), and this completes our proof.
\end{proof}

To make sense of the ``dense pairs condition with large parts'' and to compare our result with prior work, let us also state the immediate corollaries of \cref{thm:sparse-k-cycle-removal} for (unweighted) graphs.

\begin{corollary}[Cycle removal in multiple sparse graphs] \label{cor:sparse-k-cycle-removal-graphs}
For any $k \in \Z_{\ge 3}$, $q_1, \ldots, q_k \in [1, \infty)$ with $\sum_i \frac{1}{q_i} > 1$, and $\epsilon, K, C > 0$, there exist $n_0, \delta, \tau > 0$ such that the following holds. Let $n \ge n_0$, $p_1, \ldots, p_k > 0$, and $G_1, \ldots, G_k$ be $n$-vertex graphs with the same vertex set $V$, such that for each $i$,
\begin{itemize}
    \item[$(i)$] \emph{(Bounded Schatten norm).} Equipping $V$ with uniform probability, one has $||\one_{G_i}||_{S^{q_i}}^{q_i} \le C p_i^{q_i}$;
    \item[$(ii)$] \emph{(Dense pairs condition with large parts).} For any partition of $V$ into parts $V_1, V_2, \ldots$ with $\ge \tau n$ elements each, a total of at most $\epsilon p_i n^2/4$ edges of $G_i$ lie between pairs $(V_j, V_{j'})$ with 
    \[
        e_{G_i}(V_j, V_{j'}) \ge Kp_i|V_j||V_{j'}|.
    \]
\end{itemize}
(Here we allow pairs with $j = j'$.) If additionally
\[
    \sum_{x_1, \ldots, x_k \text{ distinct}} \one_{G_1}(x_1, x_2) \cdots \one_{G_k}(x_k, x_1) \le \delta p_1 \cdots p_k n^k,
\]
then there exist subgraphs $G_i'$ of $G_i$ with $|E(G_i) \setminus E(G_i')| \le \epsilon p_i n^2$ for each $i$, such that
\[
    \sum_{x_1, \ldots, x_k} \one_{G_1'}(x_1, x_2) \cdots \one_{G_k'}(x_k, x_1) = 0.
\]
\end{corollary}

\begin{remark}
Recall that if $q_i$ is an even integer, then $||\one_{G_i}||_{S^{q_i}}^{q_i} = n^{-q_i} \Hom(C_{q_i}, G_i)$, so condition $(i)$ becomes a bound on $G_i$'s homomorphic even cycle counts.
\end{remark}

\begin{proof}
This is just \cref{thm:sparse-k-cycle-removal} applied to $X = V$, $f_i = \one_{G_i}/p_i$, $q_i = 2\ell_i$, $2\epsilon$ instead of $\epsilon$, and $C^{1/\max q_i}$ instead of $C$.
\end{proof}

We can similarly deduce our main removal lemma from \cref{sec:introduction}, concerning a single graph:

\begin{proof}[Proof of \cref{thm:sparse-removal-odd-cycles}]
This is just (the last paragraph of) \cref{thm:sparse-k-cycle-removal} applied to $X = V(G)$, $f_i = \one_{G}/p$ and $q_i = 2k$ for $i \in \{1, \ldots, 2k+1\}$, $2\epsilon/(2k+1)$ instead of $\epsilon$, and $C^{1/(2k+1)}$ instead of $C$. Unlike in \cref{cor:sparse-k-cycle-removal-graphs}, it now suffices to consider the `off-diagonal' dense pairs $(V_i, V_j)$ with $i \neq j$. Indeed, the set $\diag(\mP)$ from \cref{thm:sparse-k-cycle-removal} corresponds to the `diagonal' pairs $(V_i, V_i)$, which can be ignored in the dense pairs condition since $f_1 = \cdots = f_{2k+1}$. 

Also, note that \cref{thm:sparse-k-cycle-removal} produces a graph which contains no homomorphic copies of $C_{2k+1}$, and such a graph cannot contain any cycle $C_{2\ell+1}$ with $1 \le \ell \le k$.
\end{proof}

\section{Applications} \label{sec:applications}

\subsection{Graphs with mildly-pseudorandom majorants}

\cref{thm:sparse-k-cycle-removal} also implies a \emph{relative} removal result, assuming the existence of a majorant $\nu \ge f$ satisfying the mild pseudorandomness assumption $||\nu - 1||\sq = o(1)$. The fact that the partitions from the dense pairs condition in \cref{thm:sparse-k-cycle-removal} consist of parts of lower-bounded probability will be critical here.

\begin{corollary}[Relative cycle removal] \label{cor:relative-k-cycle-removal}
For any $k \in \Z_{\ge 3}$ and $q_1, \ldots, q_k \in [1, \infty)$ with $\sum_i \frac{1}{q_i} > 1$, and any $\epsilon, K, C > 0$, there exist $\delta, \tau > 0$ such that the following holds.
Let $(X_i, \P_i)$ be finite probability spaces and $f_{i}, \nu_{i} : X_i \times X_{i+1} \to [0, \infty)$, for $i \in \Z/k\Z$. Assume that for each $i$, one has
\begin{itemize}
    \item[$(i)$] \emph{(Bounded Schatten norm).} $||f_{i}||_{S^{q_i}} \le C$,
    \item[$(ii)$] \emph{(Mildly-pseudorandom majorant).} $f_i \le K\nu_i$ and $||\nu_i - 1||\sq \le \tau$.
\end{itemize}
If additionally
\begin{equation} \label{eq:few-k-cycles-2}
    \E_{x_1, \ldots, x_k} \prod_{i=1}^k f_{i}(x_i,x_{i+1}) \le \delta,
\end{equation}
then there exist subsets $E_{i} \subset X_i \times X_{i+1}$ with $||f_{i} \one_{E_{i}^c}||_{L^1} \le \epsilon$ for each $i$, such that
\[
    \E_{x_1, \ldots, x_k} \prod_{i=1}^k \one_{E_{i}}(x_i,x_{i+1}) = 0.
\]
Moreover, if all $(X_i, \P_i) = (X, \P)$ are equal and $\P$ is the uniform probability distribution, and $|X|$ is sufficiently large in terms of $k,q_1,\ldots,q_k,\epsilon,K,C$, then instead of \cref{eq:few-k-cycles-2} one can require that $\E_{x_1, \ldots, x_k} \one_{x_1,\ldots,x_k\textnormal{ distinct}} \prod_{i=1}^k f_{i}(x_i,x_{i+1}) \le \delta$.
\end{corollary}

\begin{remark}
The pseudorandomness assumption in \cref{cor:relative-k-cycle-removal} is fairly general for the following reasons:
\begin{itemize}
    \item[(1).] Firstly, for any function $\nu : X \times Y \to [0, \infty)$, \cref{eq:cut} and \cref{eq:Schatten-infty} imply that
    \[
        ||\nu - 1||\sq 
        =
        \sup_{\substack{A \subset X \\ B \subset Y}}
        |\E[(\nu - 1) \one_{A \times B}]|
        \le 
        \sup_{\substack{A \subset X \\ B \subset Y}}
        \frac{|\E[(\nu - 1) \one_{A \times B}]|}{\sqrt{\P_X(A) \P_Y(B)}}
        \le 
        ||\nu - 1||_{S^\infty}.
    \]
    When $|X| = |Y| = n$ and $X, Y$ are equipped with the uniform probability distributions, the second-to-last quantity above is $o(1)$ if and only if $\nu$ is $(1, o(n))$-jumbled (or equivalently, if $p \nu$ is $(p, o(pn))$-jumbled for any $p > 0$), using the terminology from \cite{conlon2014extremal}.
    So our assumption that $||\nu - 1||\sq = o(1)$ is strictly weaker than the $(p, o(pn))$-jumbledness of $p\nu$, which is weaker than $||\nu - 1||_{S^\infty} = o(1)$, and much weaker than $||\nu - 1||_{S^4} = o(1)$.
    \item[(2).] Secondly, the conditions $||\nu_i - 1||_{S^4} = o(1)$ are themselves weaker than the linear forms conditions of \cite{conlon2014green,tao2006gaussian}, due to \cref{eq:Schatten-2k}. Importantly, our assumptions about $\nu_{i}$ are \emph{separated}, in the sense that each condition concerns only one of the majorants $\nu_{i}$; by contrast, a linear forms-type condition would also include (among many others) an estimate of the form
    \[
        \E_{x_1,\ldots,x_k} \prod_{i=1}^k \nu_{i}(x_i, x_{i+1}) = 1 + o(1).
    \]
    Of course, the linear forms conditions do not include a bounded-Schatten-norm condition for $f_i$, which is an advantage if $q_i < 4$, and particularly when $k = 3$; see point (4) below. But when $k = 3$, our result can still recover Roth's theorem in the primes, by choosing $q_i \in (2, 3)$ and using the \emph{discrete majorant property} from Green's original proof \cite[Lemma 6.6]{green2005roth}. 
    \item[(3).] Thirdly, even the \emph{dense model} theorem in \cite{conlon2014green} (which is the more flexible part of the transference process therein) assumes that $||\nu - 1||\sq = o(1)$, so this may be the weakest meaningful pseudorandomness condition that one can assume in practice.
    \item[(4).] Finally, provided that $q_i >= 4$, one can let $r_i := \lfloor q_i/2 \rfloor$ and note that
    \[
      ||f_i||_{S^{q_i}} \le ||f_i||_{S^{2r_i}} \le ||\nu_i||_{S^{2r_i}},
    \]
    by \cref{eq:Schatten-2k}. Hence the bounded-Schatten-norm condition can be replaced by a pseudorandomness condition $||\nu_i||_{S^{2r_i}} \ll 1$. In particular, it suffices to assume that $||\nu_i - 1||_{S^{2r_i}} = o(1)$, which also supersedes the condition $||\nu_i - 1||\sq = o(1)$.
\end{itemize}
\end{remark}

\begin{proof}[Proof of \cref{cor:relative-k-cycle-removal}]
Apply \cref{thm:sparse-k-cycle-removal} with $2K$ in place of $K$ and $\sqrt{\tau}$ in place of $\tau$. This produces some $\eps, \tau > 0$ such that our desired conclusion holds, assuming the following condition:
\begin{equation} \label{eq:bounded-averages-pseudorandom}
    \text{For any $A_j \subset X_j$ with $\P_j(A_j) \ge \sqrt{\tau}$, one has $\E[f_i \mid A_i \times A_{i+1}] \le 2K$.}
\end{equation}
Indeed, if \cref{eq:bounded-averages-pseudorandom} holds, then for any partitions $\mP_i, \mP_{i+1}$ of $X_i, X_{i+1}$ into parts of probabilities $\ge \sqrt{\tau}$, we will have $\E(f_i \mid \mP_i \otimes \mP_{i+1}) \le 2K$ everywhere, making the dense pairs condition in \cref{thm:sparse-k-cycle-removal} trivial. It thus remains to verify \cref{eq:bounded-averages-pseudorandom}; given such $A_j$, we can write
\[
\begin{aligned}
    \E[f_i \mid A_i \times A_{i+1}] = 
    &\le 
    K\E[\nu_i \mid A_i \times A_{i+1}]
    \\
    &= 
    K +
    K\E[\nu_i - 1 \mid A_i \times A_{i+1}]
    \\
    &\le
    K +
    K\frac{\E[(\nu_i - 1) \one_{A_i \times A_{i+1}}]}{\P(A_i)\ \P(A_{i+1})}
    \le 
    K + K \frac{||\nu_i - 1||\sq}{\tau} \le 2K
\end{aligned}    
\]
for each $i$, which completes our proof.
\end{proof}

For the remainder of this subsection, we deduce \cref{cor:subgraphs-pseudorandom,cor:linear-patterns-pseudorandom} from \cref{cor:relative-k-cycle-removal}. 

\begin{corollary}[Rephrasing of \cref{cor:subgraphs-pseudorandom}]
For any $k \in \Z_{\ge 2}$ and $\epsilon, K, C > 0$, there exist $n_0, \delta, \tau > 0$ such that the following holds. Let $n \ge n_0$, and $G$ be an $n$-vertex graph with $\Inj(C_{2k+1}, G) \le \delta p^{2k+1} n^{2k+1}$ and $\one_G/p \le K\nu$ for some $\nu : V(G) \times V(G) \to [0, \infty)$ with $||\nu||_{S^{2k}} \le C$ and $||\nu - 1||\sq \le \tau$. Then one can remove at most $\epsilon p n^2$ edges from $G$ to obtain a graph $G'$ with $\Hom(C_{2k+1}, G') = 0$.
\end{corollary}

\begin{proof}
This follows from (the last paragraph of) \cref{cor:relative-k-cycle-removal} applied to $X = V(G)$, $2k+1$ in place of $k$, $2\epsilon/(2k+1)$ in place of $\epsilon$, $f_i = \one_G/p$, $q_i = 2k$, and $KC$ in place of $C$. Note that the condition $||\nu ||_{S^{2k}} \le C$ implies $||f_i||_{S^{2k}} \le K ||\nu_i||_{S^{2k}} \le KC$, due to \cref{eq:Schatten-2k}.
\end{proof}

\begin{corollary}[Extension of \cref{cor:linear-patterns-pseudorandom}] \label{cor:linear-patterns-pseudorandom-weighted}
For any $k \in \Z_{\ge 3}$, $q \in [1, k)$, and $K, C > 0$, there exist $\delta, \tau > 0$ such that the following holds. Let $a_1, \ldots, a_k$ be nonzero integers with $a_1 + \cdots + a_k = 0$, $n \in \Z_{\ge 1}$ with $\gcd(n, a_1 \cdots a_k) = 1$, $G$ be an abelian group of order $n$, and $f, \nu : G \to [0, \infty)$ satisfy
\begin{itemize}
    \item[$(i)$] \emph{(Not-too-large Fourier coefficients).}  $||\hat{f}||_{q} \le C\hat{f}(0)$, and
    \item[$(ii)$] \emph{(Mildly-pseudorandom majorant).} $f \le K \frac{\hat{f}(0)}{n} \nu$ and $||\nu - 1||\sq \le \tau$. 
\end{itemize}
Then one has
\[
    \sum_{\substack{g_1, \ldots, g_k \in G \\ a_1 g_1 + \cdots + a_k g_k = 0}} 
    f(g_1) \cdots f(g_k) 
    \ge 
    \delta \frac{\hat{f}(0)^k}{n}.
\]
\end{corollary}

\begin{proof}
Apply \cref{cor:relative-k-cycle-removal} for $q_i = k$ and $1/(2k)$ in place of $\epsilon$ to obtain values of $\delta, \tau > 0$, and let $f$ be as in the corollary's statement. The conclusion is trivial if $f$ is the identically zero function, so we may assume without loss of generality that $\hat{f}(0) \neq 0$. Moreover, we may scale $f$ by a constant to further assume that $\hat{f}(0) = n$ (in other words, $f$ has average $1$). 

Let $b_1, \ldots, b_k \in \Z \setminus \{0\}$ be inverses modulo $n$ for $a_1, \ldots, a_k$, $X_1 = \cdots = X_k := G$ (equipped with the uniform probability measure), and $f_i, \nu_i : X_i \times X_{i+1} \to [0, \infty)$ be given by
\[
    f_i(x_i, x_{i+1}) := f(b_i(x_i - x_{i+1})), 
    \qquad\qquad 
    \nu_i(x_i, x_{i+1}) := \nu(b_i(x_i - x_{i+1})),
\]
for $i \in \Z/k\Z$. By \cref{lem:Fourier-singular} and assumption $(i)$, we have $||f_i||_{S^q} = \frac{1}{n} ||\hat{f}||_{q} \le C$. Moreover, $||\nu_i - 1||\sq = ||\nu - 1||\sq$ since the latter was defined as the cut norm of the function $(x, y) \mapsto (\nu-1)(x+y)$, and the maps $x \mapsto \pm b_i x$ are automorphisms of $G$. Thus by assumption $(ii)$, we have $f_i\le K \nu_i$ and $||\nu_i - 1||\sq \le \tau$.

Now assume towards a contradiction that $\sum_{a_1 g_1 + \cdots + a_k g_k = 0} f(g_1) \cdots f(g_k) < \delta n^{k-1}$, which translates to $\E_{x_1, \ldots, x_k} \prod_{i=1}^k f_i(x_i, x_{i+1}) < \delta$. Then we may apply the conclusion of \cref{cor:relative-k-cycle-removal} to get subsets $E_i \subset X_i \times X_{i+1}$ with $||f_i \one_{E_i^c}||_{L^1} \le 1/(2k)$ for each $i$, such that $\prod_{i=1}^k \one_{E_i}(x_i, x_{i+1}) = 0$ for any $x_1, \ldots, x_k$. In particular, this implies that $\sum_{i=1}^k \one_{E_i^c} (x_i, x_{i+1}) \ge 1$ always.

To obtain a contradiction, we define $T \subset G^k$ by 
\[
    T := \{(x_1, \ldots, x_k) \in X_1 \times \cdots \times X_k : b_1(x_1 - x_2) = b_2(x_2 - x_3) = \cdots = b_k(x_k - x_1)\}.
\]
Note that any pair $(x_i, x_{i+1}) \in X_i \times X_{i+1}$ uniquely determines a $k$-tuple $(x_1, \ldots, x_k) \in T$ containing it; the uniqueness follows easily from the definition of $T$, while the existence follows from the fact that $a_1 + \cdots + a_k = 0$. Indeed, denoting $g_i = b_i(x_i - x_{i+1})$, we have $\sum_{i=1}^k a_i g_i = 0$; thus if we assume that $k-1$ of the $g_i$'s are equal, then all of them must be equal. Moreover, for $(x_1, \ldots, x_k) \in T$, all of the values $f_{i}(x_i, x_{i+1}) = f(b_i(x_i - x_{i+1}))$ are equal. Hence we may write
\[
\begin{aligned}
    ||f_1||_{L^1} 
    =
    \frac{1}{n^2} \sum_{(x_1, \ldots, x_k) \in T} f_{1}(x_1, x_2)
    &\le
    \frac{1}{n^2} 
    \sum_{(x_1, \ldots, x_k) \in T} f_{1}(x_1, x_2) \sum_{i=1}^k \one_{E_{i}^c}(x_i, x_{i+1})
    \\
    &=
    \frac{1}{n^2} 
    \sum_{i=1}^k \sum_{(x_1, \ldots, x_k) \in T} f_{i}(x_i, x_{i+1})  \one_{E_{i}^c}(x_i, x_{i+1})
    \\
    &=
    \sum_{i=1}^k\frac{1}{n^2} 
    \sum_{(x_i, x_{i+1}) \in X_i \times X_{i+1}} f_{i}\one_{E_{i}^c}  (x_i, x_{i+1})
    =
    \sum_{i=1}^k ||f_i \one_{E_i^c}||_{L^1},
\end{aligned}
\]
which implies that $1 \le k (1/2k) = 1/2$. Thus we must have $\sum_{a_1g_1 + \cdots + a_k g_k = 0} f(g_1) \cdots f(g_k) \ge \delta n^{k-1}$, as we wanted.
\end{proof}

\begin{proof}[Proof of \cref{cor:linear-patterns-pseudorandom}]
Take $q = 2\ell$ and $f = \one_S$ in \cref{cor:linear-patterns-pseudorandom-weighted}, and note that $||\hat{\one_S}||_{2\ell} = n||\one_S||_{S^{2\ell}} \le \epsilon^{-1}|S| \cdot ||\nu||_{S^{2\ell}} \le \epsilon^{-1} C |S|$ by $(i)$ and $(ii)$.
\end{proof}

\subsection{Graphs with few 4-cycles} 

Following \cite{conlon2021regularity}, we first state a lemma relating the dense pairs condition to counts of $4$-cycles.

\begin{lemma}[Dense pairs with large parts I] \label{lem:dense-pairs-1}
Let $\tau, q > 0$, and $G$ be an $n$-vertex graph with $T := \Inj(C_4, G)$. Consider a partition of $V(G)$ into parts $V_1, V_2, \ldots$ with at least $\tau n$ elements each. Then $G$ has at most
\begin{equation} \label{eq:dense-pairs-bound}
    O\left(\frac{n}{q} + \frac{1}{q\tau^2} + \frac{T^{1/4}n}{\tau}\right)
\end{equation}
edges lying between pairs $(V_j, V_{j'})$ with $e_G(V_j, V_{j'}) \ge q|V_j||V_{j'}|$. Moreover, if one only counts the `off-diagonal' pairs with $j \neq j'$, then one can replace $T$ with $\sup_{G \supset H \textnormal{ bipartite}} \Inj(C_4, H)$.
\end{lemma}

\begin{proof}
This is a slight simplification of the proof of \cite[Lemma 2.4]{conlon2021regularity}, since in our case we know that $|V_j| \ge \tau n$ for each $j$. 

Let $W_j$ be the union of all $V_{j'}$'s with $e_G(V_j, V_{j'}) \ge q|V_j||V_{j'}|$; in particular, $e_G(V_j, W_j) \ge q|V_j||W_j|$. Let us consider the bipartite graph $H_j$ with vertex sets $V_j$ and $W_j^0$, where $W_j^0$ is a copy of $W_j$ (so that $W_j^0 \cap V_j = \emptyset$), and edges induced by the edges of $G$. While $H_j$ is not necessarily isomorphic to a subgraph of $G$ (since we may have $V_j \subset W_j$, in which case the vertices in $V_j$ have copies in $W_j^0$), it is true that every $4$-cycle in $H_j$ induces a $4$-cycle in $G$ (the fact that $G$ is a simple graph causes the right pairs of vertices to be distinct in $V(G)$). Moreover, each $4$-cycle in $G$ corresponds to at most two $4$-cycles in $H_j$, and thus $\Inj(C_4, H_j) \ll \Inj(C_4, G)$. This line of reasoning is implicit in \cite[Lemma 2.4]{conlon2021regularity}.

Now if $e(V_j, W_j) < 4|V_j||W_j|^{1/2} + 4|W_j|$, then we have
\[
\begin{aligned}
    e(V_j, W_j) \le \frac{e(V_j, W_j)^2}{q|V_j||W_j|}
    &\ll 
    \frac{|V_j|^2|W_j| + |W_j|^2}{q|V_j||W_j|}
    =
    \frac{|V_j|}{q} + \frac{|W_j|}{q|V_j|}
    \le 
    \frac{|V_j|}{q} + \frac{1}{q\tau}.
\end{aligned}
\]
Otherwise, \cite[Lemma 2.3]{conlon2021regularity} implies that $e(V_j, W_j)^4 |V_j|^{-2} |W_j|^{-2} \ll \Inj(C_4, H_j) \ll \Inj(C_4, G) = T$ (provided that $W_j$ is nonempty). Thus
\[
    e(V_j, W_j) \ll T^{1/4} |V_j|^{1/2} |W_j|^{1/2} \le T^{1/4} n.
\]
Thus in both cases, we have
\[
    e(V_j, W_j) \ll \frac{|V_j|}{q} + \frac{1}{q\tau} + T^{1/4} n.
\]
Summing over all $j$ (which gives a total of at most $n/(\tau n) = 1/\tau$ terms), we recover the desired bound in \cref{eq:dense-pairs-bound}.

If we only count the `off-diagonal' pairs $(V_j, V_{j'})$ with $j \neq j'$, then we may repeat the same reasoning by replacing $W_j$ with $W_j \setminus V_j$. The difference is that $H_j$ is now isomorphic to a subgraph of $G$, since there are no repeated vertices in $V_j$ and $W_j$; thus we have
\[
    \Inj(C_4, H_j) \le \sup_{G \supset H \textnormal{ bipartite}} \Inj(C_4, H),
\]
and we may use the right-hand side as the parameter $T$.
\end{proof}

We now state a further corollary of \cref{cor:sparse-k-cycle-removal-graphs} for graphs with few $4$-cycles, which should be compared with the result for $k$-partite graphs in \cite[Theorem 1.6]{conlon2021regularity} (with $k = 5$).

\begin{corollary}[Cycle removal in multiple graphs with few $4$-cycles] \label{cor:k-cycle-removal-few-4-cycles}
For any $k \in \Z_{\ge 5}$ and any $\epsilon, C > 0$, there exist $n_0, \delta > 0$ such that the following holds. Let $n \ge n_0$, $p_1, \ldots, p_k \ge n^{-1/2}$, and $G_1, \ldots, G_k$ be $n$-vertex graphs with the same vertex set $V$, such that for each $i$,
\begin{equation} \label{eq:few-4-cycles}
    \Inj(C_4, G_i) \le \delta p_i^4 n^4.
\end{equation}
If additionally
\[
    \sum_{x_1, \ldots, x_k \text{ distinct}} \one_{G_1}(x_1, x_2) \cdots \one_{G_k}(x_k, x_1) \le \delta p_1 \cdots p_k n^k,
\]
then there exist subgraphs $G_i'$ of $G_i$ with $|E(G_i) \setminus E(G_i')| \le \epsilon p_i n^2$ for each $i$, such that
\[
    \sum_{x_1, \ldots, x_k} \one_{G_1'}(x_1, x_2) \cdots \one_{G_k'}(x_k, x_1) = 0.
\]
\end{corollary}

\begin{proof}
Fix $k \in \Z_{\ge 5}$ and $\epsilon > 0$, and let $n_0, \delta, \tau > 0$ be obtained from \cref{cor:sparse-k-cycle-removal-graphs} with $q_1 = \cdots = q_k := 4$, $K$ a large enough multiple of $1/\epsilon$, and $C$ a large enough absolute constant. Note that we may always decrease our choice of $\delta$ and increase our choice of $n_0$ (in terms of $\epsilon$ and $k$), without breaking the conclusion of \cref{cor:sparse-k-cycle-removal-graphs}; so for a start, let us take $\delta \le 1$ without loss of generality.

Now let $n \ge n_0$, $p_1, \ldots, p_k \ge n^{-1/2}$, and $G_1, \ldots, G_k$ be graphs with the same vertex set $V$ (of size $n$), such that $\Inj(C_4, G_i) \le \delta p_i^4 n^4$. To complete our proof, it suffices to verify that each $G_i$ satisfies the conditions in \cref{cor:sparse-k-cycle-removal-graphs}:
\begin{itemize}
    \item[$(i)$.] One has $\Hom(C_4, G_i) \le C(p_i n)^{4}$;
    \item[$(ii)$.] $G_i$ obeys the dense pairs condition with parts of sizes $\ge \tau n$.
\end{itemize}
To verify condition $(i)$, note that \cref{eq:4-cycle-counts} implies 
\[
    \Hom(C_4, G) \ll n^2 + \delta p_i^4 n^4 \ll p_i^4 n^4,
\]
which is enough since we chose $C$ to be a large enough absolute constant.

Concerning condition $(ii)$, let $V = V_1 \cup V_2 \cup \ldots$ be a partition into parts of sizes $\ge \tau n$. Applying \cref{lem:dense-pairs-1} with $q := Kp_i$ and $T = \Inj(C_4, G_i) \le \delta p_i^4 n^4$, we find that $G_i$ has at most
\[
    O\left(\frac{n}{Kp_i} + \frac{1}{Kp_i \tau^2} + \frac{\delta^{1/4}p_i n^2}{\tau} \right)
\]
edges between pairs $(V_j, V_{j'})$ with $e_G(V_j, V_{j'}) \ge Kp_i |V_j| |V_{j'}|$. We would like this to be at most equal to $\epsilon p_i n^2/4$, which can be achieved if:
\begin{itemize}
    \item[(1).] $K$ is a large enough multiple of $1/\epsilon$ and $p_i^2 n \ge 1$ (both of which are true);
    \item[(2).] $1/\tau^2 \le n$, which is true provided that we increase the value of $n_0$ accordingly;
    \item[(3).] $\delta^{1/4}/\tau$ is smaller than an absolute multiple of $\epsilon$, which is true provided that we decrease the value of $\delta$ accordingly.
\end{itemize}
This completes our proof with (new) values of $n_0$ and $\delta$ depending only on $\epsilon$ and $k$ (since $K$ and $\tau$ also only depend on $\epsilon$ and $k$).
\end{proof}

We can similarly deduce the last part of \cref{cor:graphs-few-4-cycles}, which concerns a single graph $G$. In fact, we prove a slightly more general statement, where we assume that $\Hom(C_{2k}, G) \ll p^{2k} n^{2k}$ instead of $\Inj(C_4, G) \ll p^4 n^4$; the former follows from the latter due to the inequalities \cref{eq:4-cycle-counts}, $p \ge n^{-1/2}$, and $||\one_G/p||_{S^{2k}} \le ||\one_G/p||_{S^4}$.

\begin{corollary}[Generalizing part $(iii)$ of \cref{cor:graphs-few-4-cycles}]
For any $k \in \Z_{\ge 2}$ and $\epsilon, C > 0$, there exist $n_0, \delta > 0$ such that the following holds. For $n \ge n_0$ and $p \ge n^{-1/2}$, any $n$-vertex graph $G$ with $\Hom(C_{2k}, G) \le Cp^{2k} n^{2k}$, $\sup_{G \supset H \textnormal{ bipartite}} \Inj(C_4, H) \le \delta p^4 n^4$, and $\Inj(C_{2k+1}, G) \le \delta p^{2k+1} n^{2k+1}$ can be made $\{C_3, C_5, \ldots, C_{2k+1}\}$-free by removing at most $\epsilon pn^2$ edges. 
\end{corollary}

\begin{proof}
We apply \cref{thm:sparse-removal-odd-cycles} with the given value of $C$, and repeat the same reasoning as in the previous proof to verify the dense pairs condition. The difference is that it now suffices to count the `off-diagonal' pairs $(V_j, V_{j'})$ with $j \neq j'$, which allows us to use $T = \sup_{G \supset H \textnormal{ bipartite}} \Inj(C_4, H)$ in \cref{lem:dense-pairs-1}.
\end{proof}

Finally, we prove \cref{cor:polygonal-patterns}, rewriting the rational slopes as fractions $s_i = a_i/b_i$ (where $(a_i, b_i) = (1, 0)$ if $s_i = \infty$).

\begin{corollary}[Rephrasing of \cref{cor:polygonal-patterns}] \label{cor:polygonal-patterns-rephrased}
For any $\epsilon > 0$, $k \in \Z_{\ge 5}$, and $(a_i, b_i) \in \Z^2$ such that $\gcd(a_i, b_i) = 1$ and $a_i b_{i+1} \neq a_{i+1} b_i$ for $i \in \Z/k\Z$, there exists $\delta > 0$ such that the following holds. Let $n \in \Z_{\ge 1}$ and $S \subset [n]^2$ satisfy
\begin{itemize}
    \item[$(i)$.] $|S| \ge \epsilon n^{3/2}$, and
    \item[$(ii)$.] For each $i$, $S$ contains at most $\delta |S|^4/n^4$ solutions in distinct points $(x_1, y_1), \ldots, (x_4, y_4)$ to
    \[
    \begin{aligned}
        &a_i(x_2 - x_1) = b_i(y_2 - y_1), \qquad\qquad 
        a_{i+1}(x_3 - x_2) = b_{i+1}(y_3 - y_2), \\
        &a_i(x_4 - x_3) = b_i(y_4 - y_3), \qquad\qquad 
        a_{i+1}(x_1 - x_4) = b_{i+1}(y_1 - y_4).
    \end{aligned}
    \]
\end{itemize}
Then $S$ contains at least $\delta |S|^k/n^k$ solutions in $(x_1, y_1), \ldots, (x_k, y_k)$ to the equations
\begin{equation} \label{eq:polygonal-solutions}
    a_i(x_{i+1} - x_i) = b_i(y_{i+1} - y_i), \qquad i \in \Z/k\Z.
\end{equation}
\end{corollary}

\begin{proof}
Let $A := k (2\max_i (|a_i| + |b_i|) + 1)$.
We choose $n_0, \delta_0 > 0$ as in \cref{cor:k-cycle-removal-few-4-cycles}, for the same value of $k$ and with $\epsilon/(4kA^2)$ in place of $\epsilon$; we set $\delta := \delta_0 / \max(\epsilon^4, \epsilon^k)$ here.

Let $n, S$ be as in the corollary's statement; we may assume without loss of generality that $n \ge n_0$, since the conclusion is trivial for $n \le n_0$ provided that $\delta < n_0^{2-k}$ (due to the $|S|$ diagonal solutions $(x_1, y_1) = \cdots = (x_k, y_k)$).

For $i \in \Z/k\Z$, we define the invertible $2 \times 2$ matrix 
\[
    M_i := 
    \begin{pmatrix}
        a_i & b_i \\
        a_{i+1} & b_{i+1}
    \end{pmatrix}
\]
and the set
\[
    V_i := [-(|a_i|+|b_i|)n, (|a_i|+|b_i|)n] \cap \Z.
\]
Finally, we let $V$ be a disjoint union of all the $V_i$'s (so in particular $n_0 \le n \le |V| \le An$), and define the (bipartite) graph $G_i$ with vertex set $V$ by
\[
    \one_{G_i}(u, v) := \one_{V_{i} \times V_{i+1}}(u, v)\ \one_S\left( M_i^{-1}\begin{psmall} u \\ v \end{psmall} \right),
\]
where we identify pairs $(x, y) \in S$ with vectors $\begin{psmall} x \\ y \end{psmall}$. Note that for $(x_i, y_i), (x_{i+1}, y_{i+1}) \in S$, we have $a_i(x_{i+1} - x_i) = b_i(y_{i+1} - y_i)$ iff $a_i x_i + b_i y_i = a_i x_{i+1} + b_i y_{i+1} =: u_i$ for some $u_i \in V_i$. With this change of variables, the system of equations in \cref{eq:polygonal-solutions} can be rewritten as
\[
    M_i \begin{pmatrix} x_{i+1} \\ y_{i+1} \end{pmatrix} = 
    \begin{pmatrix} u_i \\ u_{i+1} \end{pmatrix},
    \qquad i \in \Z/k\Z,
\]
so that the number of solutions to \cref{eq:polygonal-solutions} becomes $\sum_{u_1, \ldots, u_k \in V} \prod_{i=1}^k \one_{G_i}(u_i, u_{i+1})$;
we want to show that this is at least equal to $\delta |S|^k/n^k$. Suppose for the sake of contradiction that this is false, so
\[
    \sum_{u_1, \ldots, u_k \in V} 
    \prod_{i=1}^k \one_{G_i}(u_i, u_{i+1}) \le 
    \frac{\delta_0}{\epsilon^k} \frac{|S|^k}{n^k}
    =
    \delta_0 p_1 \cdots p_k n^k
    \le \delta_0 p_1 \cdots p_k |V|^k,
\]
where $p_1 = \cdots = p_k := |S|/(\epsilon n^2) \ge n^{-1/2}$ (by assumption $(i)$). Similarly, assumption $(ii)$ implies that for each $i$,
\[
    \sum_{\substack{u_1, u_2, u_3, u_4 \in V \\ u_1 \neq u_3,\ u_2 \neq u_4}} \one_{G_i}(u_1, u_2) \one_{G_i}(u_3, u_2) \one_{G_i}(u_1, u_4) \one_{G_i}(u_2, u_4) 
    \le \frac{\delta|S|^4}{n^4} \le \frac{\delta_0}{\epsilon^4} \frac{|S|^4}{n^4}
    \le \delta_0 p_i^4 |V|^4,
\]
which is to say, $\Inj(C_4, G) \le \delta_0 p_i^4 |V|^4$. By the conclusion of our initial application of \cref{cor:k-cycle-removal-few-4-cycles}, we find that there exist subsets $E_i \subset V_i \times V_{i+1}$ such that
\[
    \sum_{u_i, u_{i+1}} \one_{G_i} \one_{E_i^c} (u_i, u_{i+1})
    \le 
    \frac{\epsilon}{2kA^2} p_i |V|^2
    \le 
    \frac{\epsilon n^2}{2k} p_i
    =
    \frac{|S|}{2k}
\]
for each $i$, and also $\sum_{u_1, \ldots, u_k \in V} \prod_{i=1}^k \one_{E_i}(u_i, u_{i+1}) = 0$; in other words, $\sum_{i=1}^k \one_{E_i^c}(u_i, u_{i+1}) \ge 1$ everywhere. To obtain a contradiction, we consider the set
\[
    T := \{(u_1, \ldots, u_k) \in V_1 \times \cdots \times V_k : M_1^{-1}\begin{psmall} u_1 \\ u_2 \end{psmall} = \cdots = M_k^{-1}\begin{psmall} u_k \\ u_1 \end{psmall} \in S \},
\]
which has natural bijections to $S$ and to the edges of each $G_i$ (through $S$). We conclude that
\[
\begin{aligned}
    |S| = |T|
    &\le 
    \sum_{(u_1, \ldots, u_k) \in T} \sum_{i=1}^k \one_{E_i^c}(u_i, u_{i+1})
    \\
    &=
    \sum_{i=1}^k \sum_{(u_1, \ldots, u_k) \in T}  \one_{E_i^c}(u_i, u_{i+1})
    \\
    &= 
    \sum_{i=1}^k \sum_{(u_i, u_{i+1}) \in V_i \times V_{i+1}} \one_{G_i}\one_{E_i^c}(u_i, u_{i+1})
    \le 
    k \frac{|S|}{2k} = \frac{|S|}{2},
\end{aligned}
\]
a contradiction. Thus it must be the case that there are at least $\delta |S|^k/n^k$ solutions to \cref{eq:polygonal-solutions}, as we wanted.
\end{proof}

\subsection{Graphs free of quasi-smooth families}

It remains to prove \cref{cor:smooth-free-graphs,cor:linear-patterns-generalized-Sidon}; we start with a lemma similar to \cref{lem:dense-pairs-1}, but in the context of quasi-smooth families (recall \cref{def:quasi-smooth}).

\begin{lemma}[Dense pairs with large parts II] \label{lem:dense-pairs-2}
Let $2 > \alpha > \beta \ge 1$ and $\mF$ be an $(\alpha, \beta)$-quasi-smooth family of bipartite graphs. Also, let $\tau, q > 0$, $G$ be an $\mF$-free $n$-vertex graph, and $V_1 \cup V_2 \cup \ldots$ be a partition $V(G)$ into parts with at least $\tau n$ elements each. Then $G$ has at most
\[
    O_\mF\left(
    \frac{n}{q^{(\alpha-1)/(2-\alpha)}}
    \left(1 + \frac{1}{n^{(\alpha - \beta)/(2 - \alpha)} \tau^{\alpha/(2-\alpha)}}\right)\right)
\]
edges lying between pairs $(V_j, V_{j'})$ with $e(V_j, V_{j'}) \ge q |V_j| |V_j'|$ and $j \neq j'$. If $\mF$ only consists of complete bipartite graphs $K_{s,t}$, then one can also include the diagonal pairs $(V_j, V_j)$.
\end{lemma}

\begin{proof}
For each $j$, denote by $W_j$ the union of all parts $V_{j'}$ with $e(V_j, V_{j'}) \ge q|V_j||V_{j'}|$ and $j \neq j'$; thus in particular $e(V_j, W_j) \ge q|V_j||W_j|$ and $V_j \cap W_j = \emptyset$. Let $H_j$ be the bipartite subgraph of $G$ with vertex sets $V_j$ and $W_j$, and $e(V_j, W_j)$ edges. Then $H_j$ is also $\mF$-free, so by the the quasi-smoothness of $\mF$ we have
\[
\begin{aligned}
    e(V_j, W_j) &\ll_\mF \min(|V_j|, |W_j|) \max(|V_j|, |W_j|)^{\alpha - 1} + \max(|V_j|, |W_j|)^\beta
    \\
    &\le |V_j| |W_j|^{\alpha - 1} + n^\beta,
\end{aligned}
\]
where we used that $\alpha - 1 \in (0, 1)$.
If $W_j$ is nonempty, this further implies that upon setting $\gamma := \frac{1}{2-\alpha}$,
\[
\begin{aligned}
    e(V_j, W_j) &\ll_\mF \frac{e(V_j, W_j)^\gamma}{q^{\gamma - 1} |V_j|^{\gamma - 1} |W_j|^{\gamma - 1}}
    \\
    &\ll 
    \frac{|V_j|^{\gamma} |W_j|^{\gamma(\alpha-1)} + n^{\gamma \beta}}{q^{\gamma - 1} |V_j|^{\gamma - 1} |W_j|^{\gamma - 1}}
    \le
    \frac{|V_j|}{q^{\gamma - 1}}
    + 
    \frac{n^{\gamma \beta - 2(\gamma - 1)}}{(q\tau^2)^{\gamma - 1}}.
\end{aligned}
\]
Summing over all (at most $1/\tau$) values of $j$ and noting that $\gamma \beta - 2(\gamma - 1) = 1 - (\alpha - \beta)/(2 - \alpha)$, we recover the desired upper bound.

Finally, if $\mF$ only consists of complete bipartite graphs, then it is okay to include $V_j$ into $W_j$ (provided that $e(V_j, V_j) \ge q|V_j|^2$). Indeed, as in the proof of \cref{lem:dense-pairs-1}, we can let $H_j$ have vertex sets $V_j$ and $W_j^0$ instead, where $W_j^0$ is a copy of $W_j$ that is disjoint with $V_j$. While $H_j$ may not be isomorphic to a subgraph of $G$, it must remain $K_{s,t}$-free for each $K_{s,t} \in \mF$ (since the existence of such a $K_{s,t}$ would force all the involved vertices in $V_j$ and $W_j^0$ to originate from distinct vertices in $G$, thus inducing a copy of $K_{s,t}$ in $G$). 
\end{proof}

\begin{corollary}[Rephrasing of \cref{cor:smooth-free-graphs}]
Fix $k \in \Z_{\ge 2}$, $2 > \alpha > \beta \ge 1$, and an $(\alpha, \beta)$-quasi-smooth family $\mF$. Then for any $\epsilon, C > 0$, there exist $n_0, \delta > 0$ such that the following holds. For $n \ge n_0$ and $p \ge n^{\alpha - 2}$, any $\mF$-free $n$-vertex graph $G$ with $\Hom(C_{2k}, G) \le Cp^{2k} n^{2k}$ and $\Inj(C_{2k+1}, G) \le \delta p^{2k+1} n^{2k+1}$ can be made $\{C_3, C_5, \ldots, C_{2k+1}\}$-free by removing at most $\epsilon pn^2$ edges. 
\end{corollary}

\begin{proof}
Fix $\epsilon, C$ and let $n_0, \delta, \tau > 0$ be obtained by applying \cref{thm:sparse-removal-odd-cycles} with the given value of $C$, and $K$ a large enough constant (to be chosen later in terms of $\epsilon$ and $\alpha$). Let $n \ge n_0$, $p \ge n^{\alpha - 2}$, and $G$ be an $\mF$-free $n$-vertex graph with $\Hom(C_{2k}, G) \le Cp^{2k}n^{2k}$ and $\Inj(C_{2k+1}, G) \le \delta p^{2k+1} n^{2k+1}$.

Arguing as in the proof of \cref{cor:k-cycle-removal-few-4-cycles}, it suffices to verify the dense pairs condition in \cref{thm:sparse-removal-odd-cycles} (which only concerns `off-diagonal' dense pairs $(V_j, V_{j'})$ with $j \neq j'$ and $e_G(V_j, V_{j'}) \ge K p |V_j| |V_{j'}|$). By \cref{lem:dense-pairs-2} for $q = Kp$, $G$ contains a total of at most
\[
    O\left(\frac{n}{(Kp)^{(\alpha-1)/(2-\alpha)}}
    \left(1 + \frac{1}{n^{(\alpha - \beta)/(2 - \alpha)} \tau^{\alpha/(2-\alpha)}}\right) \right)
\]
edges lying between such pairs, and we want this to be less than $\epsilon p n^2/8$. This is true provided that $n$ is large enough in terms of $\tau, \alpha, \beta$ (which holds after increasing the value of $n_0$ accordingly), and that
\begin{equation} \label{eq:dense-pairs-necessary}
    \frac{n}{p^{(\alpha-1)/(2-\alpha)}}
    \le 
    p n^2,
\end{equation}
since we chose $K$ to be large enough in terms of $\epsilon$ and $\alpha$. But \cref{eq:dense-pairs-necessary} rearranges precisely to our assumption that $p \ge n^{\alpha - 2}$, which completes our proof.
\end{proof}

One can similarly deduce a corollary for multiple graphs, using \cref{cor:sparse-k-cycle-removal-graphs} instead of \cref{thm:sparse-removal-odd-cycles}; however, the necessity to consider diagonal dense pairs $(V_j, V_j)$ restricts our result to families that consist only of complete bipartite graphs.

\begin{corollary}[Cycle removal in multiple $K_{s,t}$-free graphs] \label{cor:k-cycle-removal-kst-free}
For any $k \in \Z_{\ge 3}$, any $\epsilon, C > 0$, and any $t_i \ge s_i \ge 1$ for $1 \le i \le k$, there exist $n_0, \delta > 0$ such that the following holds. Let $n \ge n_0$, $p_i \ge n^{-1/s_i}$, and $G_1, \ldots, G_k$ be $n$-vertex graphs with the same vertex set $V$, such that for each $i$,
\begin{itemize}
    \item[$(i)$.] Equipping $V$ with uniform probability, one has $||\one_{G_i}||_{S^{q_i}}^{q_i} \le C p_i^{q_i}$, and
    \item[$(ii)$.] $G_i$ is $K_{s_i, t_i}$-free.
\end{itemize}
If additionally
\[
    \sum_{x_1, \ldots, x_k \text{ distinct}} \one_{G_1}(x_1, x_2) \cdots \one_{G_k}(x_k, x_1) \le \delta p_1 \cdots p_k n^k,
\]
then there exist subgraphs $G_i'$ of $G_i$ with $|E(G_i) \setminus E(G_i')| \le \epsilon p_i n^2$ for each $i$, such that
\[
    \sum_{x_1, \ldots, x_k} \one_{G_1'}(x_1, x_2) \cdots \one_{G_k'}(x_k, x_1) = 0.
\]
\end{corollary}

\begin{proof}
We follow the proof of \cref{cor:k-cycle-removal-few-4-cycles}, using \cref{lem:dense-pairs-2} instead of \cref{lem:dense-pairs-1} (and the given values of $q_i, p_i \ge n^{-1/s_i}$). The Schatten norm bounds are now given explicitly in our hypothesis, while the dense pairs condition follows from the same computation as in \cref{eq:dense-pairs-necessary} for each $i$ (recall that $K_{s_i,t_i}$ is $(\alpha_i, \beta_i)$-quasi-smooth \cite{furedi1996upper} with $\alpha_i = 2 - 1/s_i$ and $\beta_i = 2 - 2/s_i$, so we have $p_i \ge n^{-1/s_i} = n^{\alpha_i - 2}$).
\end{proof}

Finally, we prove (a generalization of) \cref{cor:linear-patterns-generalized-Sidon}, using a Sidon-type condition corresponding to $K_{s,t}$-free graphs. We say that $S$ is $(s, t)$-Sidon if and only if $S^{s \times t}$ contains no $(s, t)$-matrix $M$ with distinct rows and distinct columns, such that 
\[
    M_{i,i} + M_{j,j} = M_{i,j} + M_{j,i}
\] 
for all $1 \le i \le s$ and $1 \le j \le t$ (i.e., the rows of $M$ are translations of each other by elements in $S$). This recovers classical Sidon sets when $s = t = 2$; note that $(s, t)$-Sidon sets are also $(t, s)$-Sidon and $(s', t')$-Sidon for $s' \ge s$ and $t' \ge t$. We expect that $(s, t)$-Sidon subsets of an abelian group of order $n$ have size $\ll n^{1 - 1/\min(s, t)}$, so that the following corollary concerns the largest of these sets.

\begin{corollary}[Extension of \cref{cor:linear-patterns-generalized-Sidon}] \label{cor:linear-patterns-st-Sidon}
For any $\epsilon, C > 0$, $s, t \in \Z_{\ge 2}$, $k \in \Z_{\ge 3}$ and $q \in [1, k)$, there exists $\delta > 0$ such that the following holds. Let $a_1, \ldots, a_k$ be nonzero integers with $a_1 + \cdots + a_k = 0$, $G$ be an abelian group of order $n$ with $\gcd(n, a_1 \cdots a_k) = 1$, and $S \subset G$ satisfy
\begin{itemize}
    \item[$(i).$] $||\hat{\one_S}||_q \le C|S|$,
    \item[$(ii).$] $|S| \ge \epsilon n^{1-1/\min(s,t)}$,
    \item[$(iii).$] $S$ is $(s, t)$-Sidon.
\end{itemize}
Then $S$ contains $\ge \delta |S|^k/n$ solutions in $x_1, \ldots, x_k$ to the linear equation $a_1 x_1 + \cdots + a_k x_k = 0$.
\end{corollary}

\begin{proof}
We can assume without loss of generality that $t \ge s$ and that $n$ is large enough in terms of $\epsilon, C, s, t, k, q$. Indeed, for small values $n \le n_0(\epsilon, C, s, t, k, q)$, the conclusion is trivial once we take $\delta < n_0^{2-k}$, due to the $|S|$ diagonal solutions $x_1 = \cdots = x_k$.

We follow the proof of \cref{cor:linear-patterns-pseudorandom-weighted} (with $f = n\one_S/|S|$ and a smaller choice of $\epsilon$), using \cref{cor:k-cycle-removal-kst-free} instead of \cref{cor:relative-k-cycle-removal}. We let $V$ consist of $k$ distinct copies $V_1, \ldots, V_k$ of the abelian group $G$ of order $n$, and set $\one_{G_i}(x_i,x_{i+1}) := \one_{V_i \times V_{i+1}}(x_i, x_{i+1}) \one_S(b_i(x_i - x_{i+1}))$, and $p := k^{-1/s}|S|/(\epsilon n)$; note that we have $p \ge |V|^{-1/s}$ by the third assumption in our hypothesis. The conclusion that $S$ contains at least $\delta |S|^k/n$ solutions to $a_1 x_1 + \cdots + a_k x_k$ is naturally the same as that of \cref{cor:linear-patterns-pseudorandom}.

The only significant difference is that instead of a pseudorandom majorant condition, we now need to verify that each $G_i$ is $K_{s,t}$-free. Suppose for the sake of contradiction that this is false, so there exist distinct $x_1, \ldots, x_s \in V_i = G$ and distinct $y_1, \ldots, y_t \in V_{i+1} = G$ such that
\[
    \forall u \le s,\ \forall v \le t: \qquad \qquad 
    b_i(x_u - y_v) \in S.
\]
But then, the matrix $M \in S^{s \times t}$ with entries given by $M_{u,v} := b_i(x_u - y_v)$ has distinct rows and columns (since $a_i b_i \equiv 1 \pmod{n}$, the $x_u$'s are distinct, and the $y_v$'s are distinct), and satisfies $M_{u,u} + M_{v,v} = M_{u,v} + M_{v,u}$ for each $u$ and $v$. This contradicts the fact that $S$ is $(s, t)$-Sidon, and completes our proof.
\end{proof}

\begin{remark}
If $S$ is an $(s,t)$-Sidon subset of $[n] := \{1, \ldots, n\}$ instead of a finite abelian group, one can identify $S$ with a subset of $\Z/N\Z$ with the same property, where $N \ll_{s,t,a_1,\ldots,a_k} n$ can be chosen to be relatively prime with $a_1, \ldots, a_k$. One can then combine \cref{cor:linear-patterns-generalized-Sidon} with an argument of Marcinkiewicz–-Zygmund \cite{zygmund2002trigonometric} (specifically, see \cite[Lemma 6.5]{green2005roth}) to obtain an analogous statement for $(s, t)$-Sidon subsets of $[n]$, where the condition on Fourier coefficients becomes continuous:
\[
    \int_{\R/\Z} \left\vert \frac{1}{|S|} \sum_{m \in S} e^{2\pi i \alpha m} \right\vert^q d\alpha  \le \frac{C}{n}.
\]
\end{remark}

\begin{proof}[Proof of \cref{cor:linear-patterns-generalized-Sidon}]
As noted in the remark above, it suffices to prove an analogous statement for subsets of abelian groups $G$ of order $n$, where $\gcd(n, a_1\cdots a_k) = 1$. Suppose that $S \subset G$ has $|S| \ge \epsilon n^{1/2}$, and that for each $h \neq 0$, $S$ contains at most $C$ pairs of the form $(x, x+h)$. If $t$ denotes any integer strictly larger than $\max(C, 1)$, this implies that $S^{2 \times t}$ contains no matrix of the form
\[
    \begin{pmatrix}
        x_1 & x_2 & \cdots & x_t \\
        x_1 + h & x_2 + h & \cdots & x_t + h
    \end{pmatrix},
\]
where $x_1, \ldots, x_t$ are distinct. In other words, $S$ is $(2, t)$-Sidon.

The conclusion would now follow from \cref{cor:linear-patterns-st-Sidon}, using $s = 2$ and $q = 4$, provided that we can show that
\[
    ||\hat{\one_S}||_4 \ll_{\epsilon,C} |S|.
\]
This is equivalent to a bound for the \emph{additive energy} of $S$:
\begin{equation} \label{eq:additive-energy}
    \#\left\{(a, b, c, d) \in S^4 : a - b = c - d \right\} \ll_{\epsilon,C} \frac{|S|^4}{n}.
\end{equation}
But we have 
\[
\begin{aligned}
    \#\left\{(a, b, c, d) \in S^4 : a - b = c - d \right\}
    &=
    \sum_{h \in G} \#\left\{(a, b, c, d) \in S^4 : a - b = c - d = h\right\}
    \\
    &=
    \sum_{h \in G} \left(\#\left\{(a, b) \in S^2 : a - b = h \right\}\right)^2.
\end{aligned}
\]
The contribution of $h = 0$ is exactly $|S|^2$. Moreover, by our second assumption, the contribution of each $h \neq 0$ is at most $C^2$. Overall, we obtain
\[
    \#\left\{(a, b, c, d) \in S^4 : a - b = c - d \right\} \le |S|^2 + C^2 n,
\]
which is acceptable in \cref{eq:additive-energy} since $|S| \ge \epsilon n^{1/2}$. This completes our proof.
\end{proof}

\section{Acknowledgements}

The author wishes to thank his advisor, Professor James Maynard, for his guidance and support, and Professors Ben Green, Terence Tao and Yufei Zhao for helpful comments and suggestions. This work was completed while the author was sponsored by the EPSRC Scholarship at the University of Oxford.


\appendix

\section{Spectral norms and entropy} \label{sec:spectral-norms-entropy}

In this section we prove \cref{lem:modify-spectral}.

\begin{proof}
We start by proving a couple of claims about any linear function $\lambda(z) = az + b$, with $a, b \in \C$:
\begin{itemize}
    \item[(a).] Given any (bounded) line segment $S \subset \C$, the maximum $\max_{z \in S} |\lambda(z)|$ is attained at one of the ends of $S$.
    \item[(b).] Given any $z_0 \in \C$, there exists
    $z_1 \in \{0\} \cup \{2^k, i2^k, -2^k, -i2^k : k \in \Z_{\ge 0}\}$ such that $|z_1| \le \max(1, 2\sqrt{2}|z_0|)$, and $|\lambda(z_1)| \ge |\lambda(z_0)|$.
\end{itemize}

Claim $(a)$ follows by writing $S = \{ct + d : t \in [0, 1]\}$ for some $c, d \in \C$, and expressing $|\lambda(ct+d)|^2$ as a convex function of $t$. 

For claim $(b)$, draw a line segment through $z_0$ with slope in $\{\pm 1\}$, intersecting both semi-axes of $z_0$'s quadrant at $w_1$ and $w_2$ (if $z_0$ is purely real or imaginary, we may ignore this step). By claim (a), we must have $|\lambda(z_0)| \le |\lambda(w_i)|$ for some $i \in \{1, 2\}$. Then, draw a line segment between $0$ and $2^k w_i/|w_i|$, where $k \ge 0$ is minimal such that $|w_i| \le 2^k$, and apply claim (a) again.
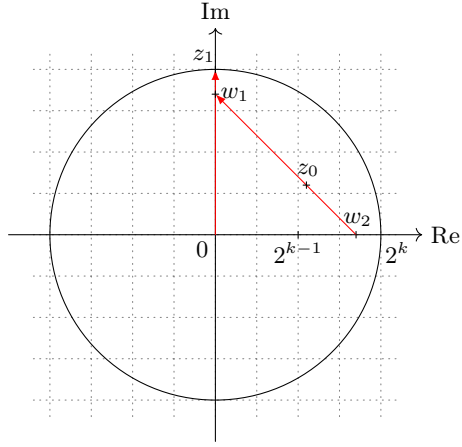
\begin{figure}[ht]
\begin{center}
\begin{tikzpicture}[scale=2.2,cap=round]
  \footnotesize
  \tikzstyle{axes}=[]
  \tikzstyle{important line}=[very thick]
  \draw[style=dotted,opacity=0.5,step=0.25cm] (-1.1,-1.1) grid (1.1,1.1);

  \draw (0,0) circle (1cm);
  \begin{scope}[style=axes]
    \draw[->] (-1.25,0) -- (1.25,0) node[right] {$\Re$};
    \draw[->] (0,-1.25) -- (0,1.25) node[above] {$\Im$};
    
    \draw [xshift=1.1 cm,yshift=0.01cm]
    node[below]
        {$2^k$};
    \draw (0.5,-0.02) -- (0.5,0.02);
    \draw [xshift=0.5 cm,yshift=0.01cm]
    node[below]
        {$2^{k-1}$};
    \draw [xshift=-0.08 cm,yshift=0.01cm]
    node[below]
        {$0$};
    \draw [xshift=0.56 cm,yshift=0.3cm]
    node[above]
        {$z_0$};
    \draw[latex-] [color = red] (0, 0.85) -- (0.85, 0);
    \draw [color = black] (0.53, 0.3) -- (0.57, 0.3);
    \draw [color = black] (0.55, 0.28) -- (0.55, 0.32);
    \draw[-latex] [color = red] (0, 00) -- (0, 1);
    \draw [color = black] (-0.02, 0.85) -- (0.02, 0.85);
    \draw [color = black] (0, 0.83) -- (0, 0.87);
    \draw [color = black] (0.85, -0.02) -- (0.85, 0.02);
    \draw [color = black] (0, 0.98) -- (0, 1.02);
    \draw [xshift=0.86 cm,yshift=0cm]
    node[above]
        {$w_2$};
    \draw [xshift=-0.02 cm,yshift=0.85cm]
    node[right]
        {$w_1$};
    \draw [xshift=0.05cm,yshift=1.08cm]
    node[left]
        {$z_1$};
  \end{scope}
\end{tikzpicture}
\end{center}
\label{fig:modify-linear-function}
\caption{An example of finding $z_1$ starting from $z_0$.}
\end{figure}

Next, to prove the lemma, recall that
\[
    ||f||_{S^\infty} = \sup_{||g||_{L^2}=||h||_{L^2}=1} \left\vert \E_{x,y}[f(x,y)g(x)h(y)] \right\vert.
\]
For fixed functions $g, h$ with $||g||_{L^2} = ||h||_{L^2} = 1$, note that the expectation $\E_{x,y}[f(x,y)g(x)h(y)]$ is linear in each value of $g$ and $h$. Hence we can apply claim $(b)$ to modify each of these values ($g(x)$ for $x \in X$ and $h(y)$ for $y \in Y$, one at a time), while not decreasing the value of $|\E_{x,y}[f(x,y)g(x)h(y)]|$, until the resulting functions $g'$ and $h'$ only have values in $\{0\} \cup \{2^k, i2^k, -2^k, -i2^k : k \in \Z_{\ge 0}\}$. Then we have
\[
    \left\vert \E_{x,y}[f(x,y)g(x)h(y)] \right\vert 
    \le 
    \left\vert \E_{x,y}[f(x,y)g'(x)h'(y)] \right\vert,
\]
\[
\begin{aligned}
    ||g'||_{L^2}^2 = \E_x\left[ |g'(x)|^2 \right] 
    &\le 
    \E_x \left[1 + \left(2\sqrt{2}|g(x)|\right)^2\right]
    =
    1 + 8||g||_{L^2}^2 = 9,
\end{aligned}
\]
and similarly $||h'||_{L^2} \le 3$. The desired bound for $||f||_{S^\infty}$ follows after taking a supremum over $g, h, g', h'$.

Finally, to establish the entropy bounds for $g$ and $h$, let $g : X \to \{0\} \cup \{2^k, i2^k, -2^k, -i2^k : k \in \Z_{\ge 0}\}$ have $||g||_{L^2} \le 3$. Denoting by $\P_X(g = 0)$ the probability of the set $\{g = 0\} = g^{-1}(0)$, we have
\[
    \H(g) = \P_X(g = 0) \log \frac{1}{\P_X(g = 0)} + \sum_{k=0}^\infty \sum_{\omega^4 = 1} \P_X(g = \omega 2^k) \log \frac{1}{\P_X(g = \omega 2^k)},
\]
and 
\[
    9 \ge ||g||_{L^2}^2 = \sum_{k=0}^\infty \sum_{\omega^4 = 1} 4^k \P_X(g = \omega 2^k).
\]
In particular, it follows that $\P_X(g = \omega 2^k) \le 9/4^k$. Now consider the function $u(p) := p \log \frac{1}{p}$ on $[0, 1]$ (with $u(0) := 0$); this function is bounded, nonnegative, and strictly increasing near $0$. Hence for sufficiently large $k$, we must have
\[
    u(\P_X(g = \omega 2^k)) \le u\left(\frac{9}{4^k}\right),
\]
and thus
\[
\begin{aligned}
    \H(g) &\ll 1 + 4\sum_{k = 2}^\infty u\left(\frac{9}{4^k}\right)
    =
    1 + 4\sum_{k = 2}^\infty \frac{9}{4^k} (k \log 4 - \log 9) \ll 1.
\end{aligned}
\]
The same argument shows that $\H(h) \ll 1$, which completes our proof.
\end{proof}

\begin{remark}
If in \cref{lem:modify-spectral} we assume that $f$ takes real values (which is enough for all combinatorial applications), then in \cref{eq:Schatten-infty} it suffices to consider real-valued functions $g, h$. Consequently, one can ignore the possible values $\pm i 2^k$ from \cref{lem:modify-spectral} in this case.
\end{remark}

\section{The ``super-dense'' removal lemma} \label{sec:super-dense-removal}

In this section, we prove (a slight generalization of) \cref{thm:super-dense-removal}. We will work with a finite probability space $(X, \P)$ with full $\sigma$-algebra; all expectations involving variables $(x_u)_{u \in V(H)}$ are understood to sample each $x_u$ from $(X, \P)$, independently of the other variables.

\begin{theorem}[Weighted version of \cref{thm:super-dense-removal}]
For any graph $H$, $q \in [1, 2)$, $C > 0$ and $\epsilon \in (0, 1)$, there exists
\[
    \delta \ge \exp \left( - O_{H,q,C}\left( \epsilon^{-O_{H,q}(1)}\right) \right)
\] 
such that the following holds. Let $f : X \times X \to [0, \infty)$ be a symmetric function (i.e., $f(x, y) = f(y, x)$) satisfying $||f||_{L^\infty} \le 1$, $||f||_{S^q} \le C$, and
\begin{equation} \label{eq:weighted-H-count}
    \E \prod_{\{u, v\} \in E(H)} f(x_u, x_v) \le \delta.
\end{equation}
Then there is a symmetric set $E \subset X \times X$ such that $||f \one_{E^c}||_{L^1} \le \epsilon$ and $\prod_{\{u, v\} \in E(H)} \one_E(x_u, x_v) = 0$ for any choice of the variables $(x_u)_{u \in V(H)}$.
\end{theorem}

\begin{remark}
One can remove the assumption $||f||_{L^\infty} \le 1$ at the expense of a slightly worse bound. To recover \cref{thm:super-dense-removal}, one can take $X = V(G)$ (equipped with uniform probability) and $f = \one_G$.
\end{remark}

\begin{proof}
Note that $E(H)$ cannot be empty, since $\E[1] = 1 > \epsilon$. 

Let us apply our regularity lemma, \cref{thm:regularity}, for some parameter $\tau \in (0, 1)$ to be chosen later in the proof (in place of $\eps$), the function $f$, and $r = 2$. This yields a partition $\mP$ of $X$ with 
\[
    \H(\mP) \le \frac{c(q)}{\tau^{4/(2-q)}},
\]
for some constant $c(q) > 0$, such that
\[
    ||f - \E(f \mid \mP \otimes \mP)||_{L^2} \le \tau C.
\]
For brevity, we write $f_1 := \E(f \mid \mP \otimes \mP)$ and $f_2 := f - f_1$, so that $f_1$ is $(\mP \otimes \mP)$-measurable and $||f_2||_{L^2} \le \tau C$. Note that this has the effect of a \emph{strong} regularity lemma, but with no uniform component $f_3$ (and no tower-exponential dependencies).

We define $E^c \subset X \times X$ as the union of the following subsets $A \times B \in \mP \otimes \mP$, using three parameters $\alpha, \beta, \gamma$ to be chosen later:
\begin{itemize}
    \item[$(i)$.] Those where $f_1$ is small (more precisely, smaller than $\alpha$);
    \item[$(ii)$.] Those where $|f_2|$ is large on average (more precisely, where $\E[|f_2| \mid A \times B] \ge \beta$);
    \item[$(iii)$.] Those with small probability (more precisely, with $\P(A) \le \gamma$ or $\P(B) \le \gamma$).
\end{itemize}
In other words, we defined $E$ by removing the blocks of edges above from $X \times X$. Now suppose for the sake of contradiction that $\prod_{\{u, v\} \in E(H)} \one_E(x_u, x_v)$ is nonzero for some choice of $(x_u)_{u \in V(H)}$. Letting $A_u$ be the part of $\mP$ containing $x_u$, we must have
\[
    (A_u, A_v) \in E,
    \qquad
    \forall \{u, v\} \in E(H).
\]
Assuming that \cref{eq:weighted-H-count} holds for some $\delta \in (0, 1)$ to be chosen later, we have
\[
    \E \left[ \prod_{w \in V(H)} \one_{A_w}(x_w) \prod_{\{u, v\} \in E(H)} f(x_u, x_v) \right] \le \delta,
\]
and we may split $f = f_1 + f_2$ in each factor from the product above. In the resulting expansion, each term containing some instance of $f_2(x_u, x_v)$ is bounded, due to the $L^\infty$ bound on $f$, by
\[
\begin{aligned}
     \E \left[ |f_2|(x_u, x_v) \prod_{w \in V(H)} \one_{A_w}(x_w) \right]
     &=
     \E \left[ \E(|f_2| \mid \mP \otimes \mP)(x_u, x_v) \prod_{w \in V(H)} \one_{A_w}(x_w) \right]
     \\
     &=
     \E \left[ \E[|f_2| \mid A_u \times A_v] \prod_{w \in V(H)} \one_{A_w}(x_w) \right]
     \\
     &\le 
     \beta \prod_{w \in V(H)} \P(A_w).
\end{aligned}
\]
On the other hand, the term consisting only of $f_1$'s is lower-bounded, due to the removal of the type-$(i)$ blocks, by
\[
    \alpha^{|E(H)|} \prod_{w \in V(H)} \P(A_w).
\]
Putting the last three bounds together, we obtain $\left(\alpha^{|E(H)|} - \beta \right) \prod_{w \in V(H)} \P(A_w) \le \delta$, 
and we now choose $\beta := \alpha^{|E(H)|}/2$ to conclude that
\[
    \delta \ge \frac{\alpha^{|E(H)|}}{2} \prod_{w \in V(H)} \P(A_w)
    \ge 
    \frac{\alpha^{|E(H)|} \gamma^{|V(H)|}}{2}.
\]
Finally, we set 
\begin{equation} \label{eq:delta-choice}
    \delta := \frac{\alpha^{|E(H)|} \gamma^{|V(H)|}}{4}
\end{equation}
to obtain a contradiction (proving that $\prod_{\{u, v\} \in E(H)} \one_E(x_u, x_v)$ is zero everywhere). It remains to choose suitable values for $\tau, \alpha, \gamma$ such that $||f \one_{E^c}||_{L^1} \le \epsilon$. By the definition of $E^c$,
\[
    ||f \one_{E^c}||_{L^1}
    \le
    \E[f \one_{f_1 \le \alpha}]
    +
    \P(\E(|f_2| \mid \mP \otimes \mP) \ge \beta)
    +
    \P(R_\gamma)^2,
\]
where $R_\gamma$ denotes the union of all sets $A \in \mP$ with $\P(A) \le \gamma$. But by the same reasoning as in \cref{eq:entropy-bound}, we have
\[
    \P(R_\gamma) \le \frac{\H(\mP)}{\log(1/\gamma)}
    \le 
    \frac{c(q)}{\tau^{4/(2-q)} \log(1/\gamma)}.
\]
At the same time, we have $\E[f \one_{f_1 \le \alpha}] = \E[f_1 \one_{f_1 \le \alpha}] \le \alpha$ and
\[
    \P(\E(|f_2| \mid \mP \otimes \mP) \ge \beta)
    \le 
    \frac{||f_2||_{L^1}}{\beta} 
    \le 
    \frac{2||f_2||_{L^2}}{\alpha^{|E(H)|}}
    \le 
    \frac{2\tau C}{\alpha^{|E(H)|}},
\]
so it makes sense to pick 
\[
    \alpha(\tau) := \tau^{1/2|E(H)|}
    \qquad\quad 
    \text{and} 
    \qquad\quad 
    \gamma(\tau) := \exp \left( - \frac{1}{\tau^{1+4/(2-q)}} \right),
\]
to conclude that 
\[
    ||f\one_{E^c}||_{L^1} 
    \le 
    \tau^{1/2|E(H)|} + 2C \sqrt{\tau} + 
    c(q)^2 \tau^2.
\]
To ensure that the right-hand side is at most equal to $\epsilon$ (and that $\tau \in (0, 1)$), it suffices to pick some
\[
    \tau = \left(\frac{\epsilon}{O_{C,q}(1)}\right)^{2 E(H)}.
\]
Noting that both $\alpha(\tau)$ and $\gamma(\tau)$ are increasing functions of $\tau$, this produces (due to \cref{eq:delta-choice}) a value of
\[
    \delta 
    \gg 
    \exp \left(- \left(\frac{\epsilon}{O_{C,q}(1)}\right)^{-O_{H,q}(1)} \right)^{O_H(1)},
\]
as we wanted.
\end{proof}

\section{A sharp regularity lemma depending only on edge density}
\label{sec:sharp-regularity}




Recall that using Schatten norms in the upper bounds of regularity lemmas allowed us to work with sparse graphs $G$, with edge density $p = \E[\one_G] \searrow 0$, since the Schatten $q$-norm $||\one_G/p||_{S^q}$ is not a function of $p$ (and may remain bounded as $p \searrow 0$ for certain families of graphs). We now ask for the best upper bound that one can use in a (weak) regularity lemma, which \emph{only} depends on the edge density $p$. In fact, we will prove a more general \emph{relative} regularity lemma, leading to a dense model theorem. We will apply our abstract energy optimization argument (\cref{lem:energy-optimization}) with a logarithmic version of energy, which may be of independent interest. 

\begin{definition}[Relative logarithmic energy]
Let $(X, \P)$ be a finite probability space (with full $\sigma$-algebra), $\mP$ be a partition of $X$, and $f, \nu : X \to [0, \infty)$ satisfy $f = 0$ whenever $\nu = 0$. Let us normalize $f, \nu$ such that $\E[f] = \E[\nu] = 1$. We define the \emph{relative logarithmic energy} of $\mP$ with respect to $f$ and $\nu$ by
\[
    \mE_f^\nu(\mP) := 
    \E \left[f \log \frac{\E(f \mid \mP)}{\E(\nu \mid \mP)} \right],
\]
where we interpret ratios of $0/0$ as $0$. Note that the trivial partition has energy $0$.
\end{definition}

\begin{lemma}[Relative logarithmic energy increment] \label{lem:logarithmic-energy-increment}
Let $\mP, \mP'$ be partitions of a finite probability space $(X, \P)$ such that $\mP'$ is finer than $\mP$. Then one has
\[
    \E \left\vert \E(f \mid \mP') - \frac{\E(f \mid \mP)}{\E(\nu \mid \mP)} \E(\nu \mid \mP') \right\vert 
    \le \sqrt{2\left(\mE_f^\nu(\mP') - \mE_f^\nu(\mP)\right)}.
\]
In particular, this shows that relative logarithmic energy is nondecreasing and nonnegative.
\end{lemma}

\begin{proof}
This is a consequence of \emph{Pinsker's inequality} in the form
\[
    \E |g - h| \le \sqrt{2 \E \left[ g \log \frac{g}{h} \right]},
\]
when $\E[g] = \E[h] = 1$ (and $g, h \ge 0$, with $g = 0$ whenever $h = 0$).
\end{proof}

\begin{theorem}[Weak relative regularity lemma] \label{thm:weak-relative-regularity}
Let $(X_1 \times X_2, \P)$ be a finite probability space, $\eps \in (0, 1)$, and $f, \nu : X_1 \times X_2 \to [0, \infty)$ be such that $f = 0$ whenever $\nu = 0$, and $\E[f] = \E[\nu] = 1$. Then there is a partition $\mP = \mP_1 \otimes \mP_2$ of $X_1 \times X_2$ with $2^{O(1/\eps^2)}$ parts, such that
\[
    \left\vert \left\vert f - \frac{\E(f \mid \mP)}{\E(\nu \mid \mP)}\nu\right\vert\right\vert\sq \le \eps \sqrt{\E\left[f \log \frac{f}{\nu}\right]}.
\]
If $X_1 = X_2$, one can also take $\mP_1 = \mP_2$. In particular, if $G$ is a finite graph with edge density $p = \E[\one_G]$ (allowing self-edges), one can find a partition $\mP_0$ of $V(G)$ with $2^{O(1/\eps^2)}$ parts such that
\[
    \left\vert\left\vert \one_{G} - \E(\one_{G} \mid \mP_0 \otimes \mP_0)\right\vert\right\vert\sq \le \eps p \sqrt{\log(1/p)},
\]
where $\mP = \mP_0 \otimes \mP_0$. This last bound is sharp, in the sense that one cannot replace $p \sqrt{\log(1/p)}$ with a function of the form $o_{p \to 0}(p \sqrt{\log(1/p)})$ in the right-hand side, uniformly in $G$ and $\eps$.
\end{theorem}

\begin{remark}
The cut norm of a function on $X_1 \times X_2$ is defined as in \cref{eq:cut}, even if there is now a joint probability distribution on $X_1 \times X_2$ (which may not split as a product of distributions on $X_1$ and $X_2$); more precisely,
\[
    ||f||\sq := \sup_{A_i \subset X_i}
    |\E[f \one_{A_1 \times A_2}]|.
\]
\end{remark}

\begin{proof}[Proof of \cref{thm:weak-relative-regularity}]
Let $\mathscr{P}$ be the family of partitions of the form $\mP_1 \otimes \mP_2$, where $\mP_i$ is a partition of $X_i$. We apply \cref{lem:energy-optimization} for the family $\mathscr{P}$, $\eps^2/4$ in place of $\eps$, the complexity function $\C(\mP) = \log_2 |\mP|$, and the energy function $\E_f^\nu(\mP)$; this produces a partition $\mP = \mP_1 \otimes \mP_2$ with at most $2^{4/\eps^2}$ parts, such that
\[
    \mE_f^\nu (\mP \vee \mQ) - \mE_f^\nu(\mP) \le 
    \frac{\eps^2}{2} \E\left[f \log \frac{f}{\nu}\right] \log_2 |\mQ|,
\]
for any partition $\mQ = \mQ_1 \otimes \mQ_2$ of $X_1 \times X_2$. Now consider the cut norm
\[
\begin{aligned}
    \left\vert \left\vert f - \frac{\E(f \mid \mP)}{\E(\nu \mid \mP)}\nu\right\vert\right\vert\sq
    &=
    \sup_{A_i \subset X_i} \left\vert \E \left[ f \one_{A_1 \times A_2} - \frac{\E(f \mid \mP)}{\E(\nu \mid \mP)}\nu \one_{A_1 \times A_2} \right] \right\vert.
\end{aligned}
\]
For fixed $A_1$ and $A_2$, let $\mQ_1 := \{A_1, A_1^c\}$, $\mQ_2 := \{A_2, A_2^c\}$ and $\mQ := \mQ_1 \otimes \mQ_2$. Then by \cref{lem:logarithmic-energy-increment}, we have
\[
\begin{aligned}
    \left\vert \E \left[ f \one_{A_1 \times A_2} - \frac{\E(f \mid \mP)}{\E(\nu \mid \mP)}\nu \one_{A_1 \times A_2} \right] \right\vert
    &=
    \left\vert \E \left[ \E(f \mid \mP \vee \mQ) \one_{A_1 \times A_2} - \frac{\E(f \mid \mP)}{\E(\nu \mid \mP)}\E(\nu \mid \mP \vee \mQ) \one_{A_1 \times A_2} \right] \right\vert
    \\
    &\le 
    \E \left\vert 
    \E(f \mid \mP \vee \mQ) - \frac{\E(f \mid \mP)}{\E(\nu \mid \mP)}\E(\nu \mid \mP \vee \mQ) \right\vert
    \\
    &\le 
    \sqrt{2\left( \mE_f^\nu (\mP \vee \mQ) - \mE_f^\nu(\mP) \right)}
    \\
    &\le
    \sqrt{2\frac{\eps^2}{4}\E \left[f \log \frac{f}{\nu} \right] \log_2 4},
\end{aligned}
\]
which is what we wanted. When $X_1 = X_2$, one can instead use the family $\mathscr{P}$ of partitions of the form $\mP_0 \otimes \mP_0$, and make the necessary modifications to $\mQ$ (i.e., take $\mQ_0 := \mQ_1 \vee \mQ_2$ and $\mQ := \mQ_0 \otimes \mQ_0$).

The statement about graphs follows by letting $X_1 = X_2 := V(G)$, equipping $V(G) \times V(G)$ with the uniform probability measure, and letting $f := \one_G/p$ and $\nu \equiv 1$.

Finally, let us show that for any constant $C > 0$, one cannot in general choose a partition $\mP = \mP_0 \otimes \mP_0$, where $\mP_0$ contains at most $2^{C/\eps^2}$ parts, such that 
\[
    \left\vert\left\vert \one_G - \E(\one_G \mid \mP)\right\vert\right\vert\sq \le \eps o_{p \to 0}\left(p \sqrt{\log(1/p)}\right)
\]
holds simultaneously for all $\eps \in (0, 1)$ and all finite graphs $G$ (where $p := \E[\one_G]$), allowing self-edges. Let $n \ge 1$ be large, $p := 1/n$, and $\eps := \sqrt{C/\log_2 (n/2)}$; we will construct a graph $G = G_n$ on $n$ vertices, with edge density $p$, such that
\[
    \left\vert\left\vert \one_{G_n} - \E(\one_{G_n} \mid \mP) \right\vert\right\vert\sq \gg 
    \frac{1}{n},
\]
for all $\mP = \mP_0 \otimes \mP_0$ such that $\mP_0$ contains at most $2^{C/\eps^2} = n/2$ sets. This will complete our proof since 
\[
    \eps p \sqrt{\log(1/p)} = \frac{1}{n} \sqrt{\frac{C \log n}{\log_2(n/2)}} 
    \ll_C 
    \frac{1}{n},
\]
and thus the ratio $||\one_G - \E(\one_G \mid \mP)||\sq / \left(\eps p \sqrt{\log(1/p)}\right)$ does not vanish as $n \to \infty$ (i.e., $p \searrow 0$).

The construction is simple, using the same example as in \cite[below Theorem 2.2]{conlon2021regularity}: take $V(G) = \{1, \ldots, n\}$ and $\one_G(a, b) = \one_{a = b}$ for $a, b \in V(G)$, so that $G$ consists of $n$ loops. Then any partition $\mP_0$ of $V(G)$ containing at most $n/2$ sets can use at most $n/2$ singletons (trivially); if $k_1, \ldots, k_m \ge 2$ denote the sizes of the remaining sets $A_1, \ldots, A_m$ in $\mP_0$, we thus have
\[
    k_1 + \ldots + k_m \ge n - \frac{n}{2} = \frac{n}{2}.
\]
Now consider a cut $C \subset V(G)$ which splits each $A_i$ roughly in half; more precisely, choose $C \subset A_1 \cup \cdots \cup A_m$ such that $|C \cap A_i| = \lceil k_i/2 \rceil$ for $i \in \{1, \ldots, m\}$. Then we have 
\[
\begin{aligned}
    \left\vert\left\vert \one_{G_n} - \E(\one_{G_n} \mid \mP) \right\vert\right\vert\sq
    &\ge 
    \E \left[ \left( \one_{G_n} - \E(\one_{G_n} \mid \mP)\right)\one_{C \times C} \right]
    \\
    &=
    \frac{1}{n^2}\sum_{i=1}^m \sum_{a, b \in C \cap A_i} \left(\one_{a=b} - \frac{k_i}{k_i^2}\right)
    \\
    &=
    \frac{1}{n^2}\sum_{i=1}^m \left(|C \cap A_i| - \frac{|C \cap A_i|^2}{k_i} \right)
    =
    \frac{1}{n^2}\sum_{i=1}^m \left\lceil \frac{k_i}{2} \right\rceil \left(1 - \frac{\lceil k_i/2 \rceil}{k_i} \right).
\end{aligned}
\]
Since each $k_i \ge 2$, it is easy to check that $\lceil k_i/2 \rceil/k_i \le 2/3$, and thus
\[
    \left\vert\left\vert \one_{G_n} - \E(\one_{G_n} \mid \mP) \right\vert\right\vert\sq
    \gg \frac{1}{n^2} \sum_{i=1}^m \frac{k_i}{2} \gg \frac{1}{n},
\]
which completes our proof.
\end{proof}

\begin{remark}
One can also prove a \emph{strong} relative regularity lemma similar to \cref{thm:weak-relative-regularity} using an additional level of iteration; we leave the details to the interested reader.
\end{remark}

The corollary below should be compared to the dense model theorem from \cite{conlon2014green}, noting the dependency of $\eps'$ on $\eps$ and $\delta$.

\begin{corollary}[Dense model theorem]
For any $\eps \in (0, 1)$, there exists $\eps' > 0$ such that the following holds. Let $(X_1 \times X_2, \P)$ be a finite probability space and $f, \nu : X_1 \times X_2 \to [0, \infty)$ satisfy
\[
    f \le \nu \qquad \text{and} \qquad 
    ||\nu - 1||\sq \le \eps'.
\]
Then there exists a function $\tilde{f} : X_1 \times X_2 \to [0, \infty)$ such that 
\[
    \tilde{f} \le 1 \qquad \text{and} \qquad ||f - \tilde{f}||\sq \le \eps.
\]
In fact, if $\delta := \E[f] > 0$ and $\E[\nu] = 1$, then one can take
$\eps' \gg \eps \exp\left(-O\left( \delta^2 \log \frac{1}{\delta}\right) \eps^{-2} \right)$.
\end{corollary}

\begin{proof}
If $f \equiv 0$, there is nothing to prove; taking $\eps' < 1/2$, we can also guarantee that $\E[\nu] \asymp 1$. So let $\delta := \E[f] > 0$, $g := f/\delta$, and $\mu := \nu/\E[\nu]$. Then by \cref{thm:weak-relative-regularity}, for any $\tau \in (0, 1)$, there exists a partition $\mP = \mP_1 \otimes \mP_2$ of $X_1 \times X_2$ with $2^{O(1/\tau^2)}$ parts such that 
\[
    \left\vert \left\vert 
    g - \frac{\E(g \mid \mP)}{\E(\mu \mid \mP)} \mu 
    \right\vert \right\vert\sq 
    \le 
    \tau \sqrt{\E\left[g \log \frac{g}{\mu} \right]}.
\]
Renormalizing and using the inequality $f \le \mu \E[\nu]$, this translates to
\[
    \left\vert \left\vert 
    f - \frac{\E(f \mid \mP)}{\E(\nu \mid \mP)} \nu 
    \right\vert \right\vert\sq 
    \le 
    \delta \tau \sqrt{\E\left[\frac{f}{\delta} \log \frac{f}{\delta \mu} \right]}
    \le 
    \delta \tau \sqrt{\log \frac{\E[\nu]}{\delta}}.
\]
On the other hand, since $\E(f \mid \mP) \le \E(\nu \mid \mP)$, one can obtain the inequality
\[
    \left\vert \left\vert 
    \frac{\E(f \mid \mP)}{\E(\nu \mid \mP)} (\nu - 1) 
    \right\vert \right\vert\sq 
    \le 
    2^{O(1/\tau^2)}
    ||\nu - 1||\sq,
\]
by expanding the right-hand side as a sum of $2^{O(1/\tau^2)}$ terms (corresponding to the parts of $\mP_1 \otimes \mP_2$).

Putting the last two inequalities together and assuming that $||\nu - 1||\sq \le \eps'$ (for some $\eps' > 0$ to be chosen), we get that
\[
    ||f - \tilde{f}||\sq 
    \le 
    2^{O(1/\tau^2)} \eps' + \delta \tau \sqrt{\log \frac{\E[\nu]}{\delta}},
\]
where $\tilde{f} := \E(f \mid \mP)/\E(\nu \mid \mP) \le 1$. The conclusion follows by choosing $\tau$ and then $\eps'$ small enough such that each of the above terms is at most $\eps/2$.
\end{proof}



\bibliographystyle{plain}
\bibliography{main}

\end{document}